\documentclass[10pt,a4paper]{amsart}
\usepackage[utf8]{inputenc}
\usepackage{amsmath}
\usepackage{amsfonts}
\usepackage{amssymb}
\usepackage{graphicx}
\usepackage[left=2cm,right=2cm,top=2cm,bottom=2cm]{geometry}
\usepackage{bbm}
\usepackage{mathrsfs}
\usepackage[all]{xy}
\usepackage{bm}
\usepackage{tikz}
\usepackage[backref=page]{hyperref}
\author{Geoffrey Powell}
\title{On the Passi and the Mal'cev functors}
\thanks{This work was partially supported by the ANR Project {\em ChroK}, {\tt ANR-16-CE40-0003}.}
\keywords{}
\subjclass[2000]{}
\newtheorem{THM}{Theorem}

\newtheorem{thm}{Theorem}[section]
\newtheorem{prop}[thm]{Proposition}
\newtheorem{cor}[thm]{Corollary}
\newtheorem{lem}[thm]{Lemma}
\theoremstyle{definition}
\newtheorem{defn}[thm]{Definition}
\newtheorem{exam}[thm]{Example}
\theoremstyle{remark}
\newtheorem{rem}[thm]{Remark}
\newtheorem{nota}[thm]{Notation}

\newcommand{\f}{\mathcal{F}}
\newcommand{\com}{\mathfrak{Com}}
\newcommand{\cat}{\mathbf{Cat}\hspace{1pt}}
\newcommand{\catlie}{\cat\lie}
\newcommand{\g}{\mathfrak{g}}
\renewcommand{\phi}{\varphi}
\renewcommand{\hom}{\mathrm{Hom}}
\newcommand{\sym}{\mathfrak{S}}
\newcommand{\gr}{\mathbf{gr}}
\newcommand{\kmod}{\mathtt{Mod}_\kring}
\newcommand{\cala}{\mathscr{A}}
\newcommand{\calc}{\mathcal{C}}
\newcommand{\cald}{\mathcal{D}}
\newcommand{\calm}{\mathcal{M}}
\newcommand{\nat}{\mathbb{N}}
\newcommand{\ab}{\mathbf{ab}}
\newcommand{\zed}{\mathbb{Z}}
\newcommand{\rat}{\mathbb{Q}}
\newcommand{\A}{\mathfrak{a}}
\newcommand{\op}{^\mathrm{op}}
\newcommand{\ob}{\mathrm{Ob}\hspace{2pt}}
\newcommand{\kring}{\mathbbm{k}}
\newcommand{\dash}{\hspace{-2pt}-\hspace{-2pt}}
\newcommand{\modules}{\mathrm{mod}}
\newcommand{\fb}{{\bm{\Sigma}}}
\newcommand{\ucom}{\com^u}
\newcommand{\lie}{\mathfrak{Lie}}
\newcommand{\uass}{\mathfrak{Ass}^u}
\newcommand{\ass}{\mathfrak{Ass}}
\newcommand{\qgr}{\mathsf{q}^\gr}
\newcommand{\tgr}{\tau_\zed^{\gr}}
\newcommand{\tgrop}{\tau_\zed^{\gr\op}}
\newcommand{\dgr}{\delta^{\gr}}
\newcommand{\dgrop}{\delta^{\gr\op}}
\newcommand{\id}{\mathrm{Id}}
\newcommand{\pbar}{\overline{P}}
\newcommand{\abel}{\mathsf{ab}}
\newcommand{\aug}{\mathcal{I}}
\newcommand{\grad}{\mathfrak{gr}}
\newcommand{\opd}{\mathcal{O}}
\newcommand{\ppd}{\mathcal{P}}
\newcommand{\lmod}[1][A]{{}_{#1}\mathtt{Mod}}
\newcommand{\rmod}[1][A]{\mathtt{Mod}_{#1}}
\newcommand{\bimod}[1][A]{{}_{#1}\mathtt{Mod}_{#1}}
\newcommand{\la}{\mathsf{lie}}
\newcommand{\liealg}{\mathrm{Lie}}
\newcommand{\cre}{\mathrm{cr}}
\newcommand{\qhat}[1]{\widehat{\mathsf{q}_{#1}^\gr}}
\newcommand{\fbcr}{\mathsf{grad}^\gr_\fb}
\newcommand{\propoly}{{\f_{\mathrm{pro}<\infty}(\gr)}}
\newcommand{\compl}{\mathfrak{c}}
\newcommand{\obar}{\overline{\otimes}}
\newcommand{\malcev}{\mathfrak{malcev}}
\newcommand{\lr}{\mathfrak{l}}
\newcommand{\prim}{\mathscr{P}}
\newcommand{\gplike}{\mathscr{G}}
\newcommand{\pbif}{P_{(-)}}
\newcommand{\pgrop}{\mathsf{p}}
\newcommand{\qgrop}{\mathsf{q}}
\newcommand{\pmap}{\mathbf{pmap}}
\newcommand{\filt}{\mathfrak{f}}
\newcommand{\triv}{\mathrm{triv}}
\newcommand{\finset}{\mathbf{Fin}}
\newcommand{\cyclie}{\mathfrak{CLie}}
\newcommand{\jac}{\mathfrak{MLie}}
\newcommand{\pjac}{\mathbb{P}}
\newcommand{\ajac}{\mathfrak{A}}
\newcommand{\chord}{\mathsf{Chord}}
\newcommand{\ogr}{\ {\widetilde{\otimes}_\gr}}
\numberwithin{equation}{section}
\begin{document}

\begin{abstract}
The author has shown that the category of analytic contravariant functors on $\mathbf{gr}$, the category of finitely-generated free groups, is equivalent to the category of left modules over the PROP associated to the Lie operad, working over $\mathbb{Q}$. This  exploited properties of the polynomial filtration of the category of contravariant functors on $\mathbf{gr}$.

The first purpose of this paper is to strengthen the corresponding result for {\em covariant} functors on $\mathbf{gr}$. This involves introducing the appropriate analogue of the category of analytic contravariant functors, namely a certain category of towers of polynomial functors on $\mathbf{gr}$. This category is abelian and has a natural symmetric monoidal structure induced by the usual tensor product of functors. Moreover, the projective generators of this category are described in terms of the {\em Mal'cev} functors that are introduced here. It follows that this category is equivalent to the category of right modules over the PROP associated to the Lie operad.  As a fundamental example, the Passi functors arising from the group ring functors are described explicitly. 

The theory is applied to consider bifunctors on $\mathbf{gr}$. This allows the $\mathbb{Q}$-linearization of the category of free groups to be described, up to polynomial filtration. 

As a stronger application of the theory, this is generalized to the Casimir PROP associated to the Lie operad, as studied by Hinich and Vaintrob. Up to polynomial filtration, this recovers the category $\mathbf{A}$ introduced by Habiro and Massuyeau in their study of bottom tangles in handlebodies. 
\end{abstract}

\maketitle
\section{Introduction}
\label{sect:intro}

For $G$ a discrete group and $\kring$ a unital commutative ring, one can consider the group ring $\kring G$ and its augmentation ideal $\aug G$. This leads to the filtration of $\kring G$ by powers of the augmentation ideal $\aug^n G$, for $n \in \nat$, and the quotient rings $\kring G / \aug^n G$; their underlying functors to $\kring$-modules are sometimes termed the Passi functors in reference to  \cite{MR537126}, for example. These constructions are natural with respect to the group $G$; here we usually restrict to the full subcategory $\gr$ of the category of groups with objects the finite-rank free groups. 

Forgetting the multiplicative structure, it is still interesting to ask what is the structure of these functors from $\gr$ to $\kring$-modules, denoted $\f (\gr)$. For example, the associated graded of the filtration of $G \mapsto \kring G$ by powers of the augmentation ideal is 
\[
\bigoplus_{n \in \nat} \A^{\otimes n}
\]
where $\A$ is the functor $G \mapsto \kring \otimes _\zed G_\abel$,  $G_\abel$ the abelianization of $G$. This filtration does not split, and understanding its structure is a key to analysing the structure of the category of functors from $\gr$ to $\kmod$. Moreover, the filtration can be constructed without using the multiplicative structure of $\kring G$ by using the {\em polynomial filtration}, based on a generalization of the notion of polynomial functor due to Eilenberg and Mac Lane \cite{MR65162}. Indeed, the functor $\A^{\otimes n}$ is a basic example of a polynomial functor of degree $n$.

One can go further, replacing the group ring functor by $P_\Gamma : G \mapsto \kring \hom_\gr (\Gamma, G)$, where $\Gamma$ is a finite rank free group. The notation $P_\Gamma$ reflects the fact that it is the projective that corepresents evaluation on $\Gamma$. Moreover, $P_\Gamma$ is contravariantly functorial with respect to $\Gamma$, leading to the consideration of the bifunctor $P_\bullet$, a functor from $\gr\op \times \gr$ to $\kring$-modules. One of the original motivations for this work was to analyse this bifunctor by exploiting the polynomial filtration. This is related to studying the $\kring$-linearization $\kring \gr$ of the category $\gr$, since $P_\bullet$ encodes the morphisms  of this category. 

Specializing to $\kring = \rat$, the question can be rephrased using the results of  \cite{2021arXiv211001934P}, which give a model for the category $\f_\omega (\gr\op)$ of {\em analytic} functors on $\gr\op$  in terms of the `infinitesimal structure', namely the action of the  $\rat$-linear category $\cat \lie$ associated to the Lie operad. Namely,  $\f_\omega (\gr\op)$ is equivalent to the category of left $\cat \lie$-modules (this theory is reviewed and revisited here in Section \ref{sect:catlie}).

Here the case of {\em covariant} functors is developed, improving upon the result established in \cite{2021arXiv211001934P}, which only treated polynomial functors with a finite composition series.  The solution is to work with pro-polynomial functors, exploiting properties of the polynomial filtration: the appropriate analogue of $\f_\omega (\gr\op)$ is the category $\propoly$ introduced in Section \ref{sect:propoly}. This has objects given by towers 
\[
\ldots \twoheadrightarrow F_d \twoheadrightarrow F_{d-1} \twoheadrightarrow \ldots \twoheadrightarrow F_0 \twoheadrightarrow F_{-1} =0, 
\]
where $F_d$ has polynomial degree $d$ and the structure map $F_d \rightarrow F_{d-1}$ induces an isomorphism $\qgr_{d-1} F_d \cong F_{d-1}$, where $\qgr_{d-1}$ is the $(d-1)$st Taylor functor that yields the universal quotient  of polynomial degree $d-1$.  

Moreover, $\propoly$ is shown to be abelian and is equipped with a symmetric monoidal structure $\obar$ induced by the tensor product on $\f (\gr)$. The category $\propoly$ can be considered as being `dual' to $\f_\omega (\gr\op)$; in particular, vector space duality induces an adjunction:
\[
\propoly \op \rightleftarrows \f_\omega (\gr\op).
\]

The category $\propoly$ is related to $\f (\gr)$ by the  completion functor $\compl : \propoly \rightarrow \f(\gr)$ given by passage to the inverse limit. This is right adjoint to the functor $\qgr_\bullet : \f (\gr) \rightarrow \propoly$ that sends a functor $F$ to the tower provided by its polynomial filtration, $(\qgr_\bullet F)$.

For example, taking $P_\zed : G \mapsto \rat G$ as above,  $\qgr_\bullet P_\zed$ encodes  the tower of Passi functors; applying the completion functor yields the completed group ring functor $G \mapsto \widehat{\rat G}$. The adjunction unit is the natural map 
$
\rat G \rightarrow \widehat{\rat G}
$, which  
 is injective when $G$ is a finite rank free group. For current purposes, it is  more natural to work at the level of $\propoly$, without passing to the completion, since one needs to retain the information on the polynomial filtration, anyway.

The key ingredient is that $\propoly$ has enough projectives; these are introduced in Section \ref{sect:malcev}. The fundamental object is $\malcev \in \ob \propoly$, which is shown to be the projective cover of $\A$ in $\propoly$ (see Theorem \ref{thm:proj_cover_malcev_s}). The functor $\malcev$ is the counterpart in $\propoly$ of  $\prim \widehat{\rat G}$, the primitives of the completed group ring, as considered by Quillen \cite[Appendix A]{MR258031}.  In particular, $\malcev$ is a Lie algebra in $\propoly$.

For each $s \in \nat$, one has $\malcev^{\obar s}\in \ob \propoly$. By Corollary \ref{cor:malcev_s_proj_generators}, $\{ \malcev^{\obar s} \ | \ s\in \nat\}$ is a set of projective generators for $\propoly$. Moreover, using the Lie algebra structure of $\malcev$, these assemble  to a left $\cat \lie$-module, denoted by $\underline{\malcev}$. This structure encodes {\em all} the morphisms between the above projective generators, by Proposition \ref{prop:malcev_full_subcat}.

Then, denoting the category of right $\cat\lie$-modules by $\rmod[\cat\lie]$, one has the following:

\begin{THM}
(Theorem  \ref{thm:equiv_propoly_modcatlie}.)
The functor $\hom_{\propoly} (\underline{\malcev}, -) $ induces an equivalence of categories
\[
\hom_{\propoly} (\underline{\malcev}, -) \ : \ 
\propoly \stackrel{\cong}{\rightarrow} \rmod[\catlie].
\]

The inverse equivalence is given by $
- \otimes _{\cat \lie} \underline{\malcev} \ : \ \rmod[\catlie] \rightarrow \propoly$.
\end{THM}

This is the counterpart of the result of \cite{2021arXiv211001934P} establishing the equivalence between $\f_\omega (\gr\op)$ and the category $\lmod[\cat\lie]$ of left $\cat \lie$-modules. The latter result is reformulated here, based on the following observation: starting from the Lie operad $\lie$, one can form the universal enveloping algebra $U \lie$, which is a unital associative algebra in right $\cat \lie$-modules. This has the structure of a cocommutative Hopf algebra in right $\cat \lie$-modules, so that one can form the exponential functor $\Phi U\lie$, a functor from $\gr\op$ to $\rmod [\cat \lie]$. The equivalence can be restated as:

\begin{THM}
(Theorem \ref{thm:analytic_grop}.)
The functor $\Phi U \lie \otimes_{\cat \lie} - $ induces an equivalence of categories
\[
\Phi U \lie \otimes_{\cat \lie} - : \lmod[\cat \lie] \rightarrow \f_\omega (\gr\op).
\] 
\end{THM}

These results are then combined to study bifunctors on $\gr$, i.e., functors from $\gr\op \times \gr$ to $\rat$-vector spaces. The category of bifunctors is denoted $\f (\gr\op \times \gr)$; it is equivalent to $\f (\gr\op; \f (\gr))$, the category of functors from $\gr\op$ to $\f (\gr)$.  

Clearly one cannot expect to recover {\em all} bifunctors by the current techniques derived from polynomial approximation. Instead, one uses $\f _\omega (\gr\op; \propoly)$, the category of analytic functors from $\gr\op$ to $\propoly$. This category will (abusively) be referred to as the category of (pro)bifunctors;  it contains as a full subcategory the category of bifunctors that are polynomial  in the appropriate sense. Moreover,  it comes equipped with a symmetric monoidal structure induced by $\obar$ on $\propoly$.

The category of $\cat \lie$-bimodules is written $\bimod[\cat \lie]$; this is equipped with a symmetric monoidal structure based on the convolution product for $\rat\fb$-bimodules (where $\fb$ is the category of finite sets and bijections). The main result of Section \ref{sect:bifunctors} is:

\begin{THM}
\label{THM:bimod}
(Theorem 
\ref{thm:bifunctors_sym_monoidal}.)
The functor 
\[
\Phi U \lie \otimes_{\cat \lie}( - )\otimes_{\cat \lie} \underline{\malcev} \ : \ 
\bimod[\cat \lie] \stackrel{\cong}{\rightarrow} \f_\omega (\gr\op; \propoly).
\]
is a symmetric monoidal equivalence.
\end{THM}

This is applied in Section \ref{sect:model_qgr_pbif} to study the $\rat $-linearization  $\rat \gr$ of $\gr$. More precisely, in Section \ref{sect:polyq}, it is shown that the polynomial filtration induces a tower of categories under $\rat \gr$ (see Theorem \ref{thm:tower}):
\[
\xymatrix{
\rat \gr 
\ar@{.>}[d]
\ar[rd]
\ar[rrd]
\ar@{.>}[rrrd]
\ar[rrrrd]
\\
\ldots 
\ar[r]
&
\qgr_d \rat \gr
\ar[r]
&
\qgr_{d-1} \rat \gr
\ar[r]
&
\ldots
\ar[r]
&
\qgr_0 \rat \gr.
}
\]

Likewise, there is a tower of $\rat$-linear categories under $\cat \lie$:

\[
\xymatrix{
\cat \lie
\ar@{.>}[d]
\ar[rd]
\ar[rrd]
\ar@{.>}[rrrd]
\ar[rrrrd]
\\
\ldots
\ar[r]
&
\cat^{\leq d+1}  \lie
\ar[r]
&
\cat^{\leq d}  \lie 
\ar[r]
&\ldots 
\ar[r]
&
\cat^{\leq 0} \lie.
}
\]

These are related:

\begin{THM}
\label{THM:tower_iso}
(Theorem \ref{thm:tower_isomorphism}.)
Under the equivalence of Theorem \ref{THM:bimod}, the tower 
\[
\ldots \rightarrow \cat ^{\leq d} \lie  \rightarrow \cat^{\leq d-1} \lie \rightarrow 
\ldots 
\rightarrow 
\cat^{\leq 0} \lie
\]
corresponds to the tower 
\[
\ldots \rightarrow \qgr_{d+1}\rat \gr \rightarrow \qgr_d \rat \gr \rightarrow \ldots \rightarrow \qgr_0 \rat \gr.
\]
\end{THM}

This is the starting point for analysing the more general situation in which $\cat \lie$ is replaced by the PROP $\pjac$ that encodes Casimir Lie algebras, as described in the work of Hinich and Vaintrob \cite{MR1913297}; this relies on the fact that the operad $\lie$ has a cyclic structure. The PROP $\pjac$ is $\nat$-graded and $\pjac_0$ identifies with $\cat \lie$; this means that each $\pjac_n$ has a natural $\cat \lie$-bimodule structure and the composition in $\pjac$ is compatible with these bimodule structures. 

One can thus form the associated (pro)bifunctors:
\[
\ajac(n) := \Phi U \lie \otimes _{\cat \lie} \pjac_n \otimes _{\cat \lie} \underline{\malcev}
\]
in $\f_\omega (\gr\op; \propoly)$, for each $n \in\nat$. 

Then, applying the general theory developed here,  the composition in $\pjac$ induces natural transformations:
\[
\ajac(n) \otimes_\gr \ajac (m) \rightarrow \ajac (m+n)
\]
for $m,n \in \nat$. Unlike the situation of Theorem \ref{THM:tower_iso}, this cannot be analysed one polynomial degree at a time; the more general framework proposed here is necessary. 

These structures are the basis of the following result:

\begin{THM}
(Theorem \ref{thm:ajac}.)
The (pro)bifunctors $\ajac(*)$ have the structure of a unital, $\nat$-graded associative monoid in 
$\f_\omega (\gr\op; \propoly)$, with $\ajac(0)$ isomorphic to $(\qgr_\bullet \rat \gr)\op$ considered as a monoid.
\end{THM}

Part of the interest of this result is that $\ajac(*)$ is related to Habiro and Massuyeau's category $\mathbf{A}$ \cite{MR4321214}, introduced in their work on bottom tangles in handlebodies. By definition, $\ob \mathbf{A}= \nat$  and $\mathbf{A}$ is $\nat$-graded. The category $\mathbf{A}$ is equipped with an embedding $\rat \gr\op \hookrightarrow \mathbf{A}$ that sends the free group $\zed^{\star r}$ to the object $r \in \nat$ and induces an equivalence $\rat \gr\op \cong \mathbf{A}_0$. In particular, for each grading $n \in \nat$, $\mathbf{A}_n(-,-)$ is a bifunctor in $\f (\gr\op \times \gr) \cong \f (\gr\op ; \f (\gr))$.

Applying the functor $\qgr_\bullet$ gives 
$
\qgr_\bullet \mathbf{A}_n (-, -) \in \f _\omega(\gr\op; \propoly).
$ 
 The composition of $\mathbf{A}$ induces an associative composition operation:
\[
\qgr_\bullet \mathbf{A} \otimes_\gr \qgr_\bullet \mathbf{A} 
\rightarrow 
\qgr_\bullet \mathbf{A}
\] 
making $\qgr_\bullet \mathbf{A}$ into an associative monoid with unit $\qgr_\bullet \mathbf{A}_0 \cong (\qgr_\bullet \rat \gr)\op$. 

A  proof of the following  is sketched in Section \ref{subsect:sketchy}:

\begin{THM}
(Theorem \ref{thm:ajac_versus_A}.)
\label{THM:ajac_HM}
There is an isomorphism $
\ajac (*) \cong \qgr_\bullet \mathbf{A}
$ of $\nat$-graded associative monoids in 
$\f_\omega (\gr\op; \propoly)$.
\end{THM}

The significance of this result is the following: the construction of $\ajac (*)$ is direct, with the composition following directly from that of $\pjac$. However, the construction does not  give natural `models' for the $\ajac (*)$ as bifunctors (without passing to completion).

By contrast, the construction of  the composition in $\mathbf{A}$ given in \cite[Section 4]{MR4321214} may appear somewhat complex on first acquaintance. However, that construction has the advantage of furnishing {\em genuine} bifunctors, rather than (pro)bifunctors.

%%%%%%%%%%%%%%%%%%%%%%%%%%%%%%%%%%%%%%%%%%%%%%%%
\subsection{Notation and conventions}
\label{subsect:nota_conv}

Throughout $\kring$ is a unital commutative ring (for the general theory, this can  be taken to be $\zed$; for the applications it will be the field $\rat$ of rational numbers);  $\kmod$ denotes the category of $\kring$-modules, considered as symmetric monoidal with respect to $\otimes_\kring$ (this will be written $\otimes$ when no confusion can result). 

\begin{nota}
Denote by
\begin{itemize}
\item[]
 $\gr$ the category of finitely-generated free groups, considered as a full subcategory of the category of groups; 
\item[]
$\fb$ the category of finite sets and bijections; 
\item[] 
 $\mathbf{n}$, for $n \in \nat$, the finite set $\{1, \ldots , n\}$ (with $\mathbf{0} = \emptyset$, by convention), so that $\fb$ has small skeleton with objects $\{ \mathbf{n}\ | \ n \in \nat\}$;
\item[]  
$\sym_n$ the symmetric group on $n$ letters (aka. $\mathrm{Aut}(\mathbf{n})$),  for $n \in \nat$ .
\end{itemize}
\end{nota}

For an essentially small category $\calc$, $\f (\calc)$ denotes the category of functors from $\calc$ to $\kmod$. This is abelian, equipped with the pointwise tensor product induced by $\otimes_\kring$. (The category $\f (\calc)$ is sometimes also referred to as the category of $\calc$-modules.) More generally, for any category $\mathcal{E}$, $\f (\calc; \mathcal{E})$ denotes the category of functors from $\calc$ to $\mathcal{E}$, so that $\f (\calc)$ is shorthand for $\f (\calc; \kmod)$. If $\mathcal{E}$ is abelian, then so is $\f (\calc; \mathcal{E})$; moreover, if $\mathcal{E}$ is symmetric monoidal, then so is $\f (\calc; \mathcal{E})$ for the induced pointwise structure.

For small categories $\calc, \cald$, the exterior tensor product is denoted $\boxtimes : \f(\calc) \times \f(\cald) \rightarrow \f (\calc \times \cald)$, given for  $F_1 \in \ob \f(\calc)$ and $F_2 \in \ob \f (\cald )$ by $F_1 \boxtimes F_2 : (X,Y) \mapsto F_1 (X) \otimes F_2 (Y)$. 

Taking $\calc = \cald\op$, the composite of $\boxtimes$ with the coend construction is denoted: 
\[
\otimes_\cald : \f(\cald\op) \times \f(\cald) \rightarrow \kmod.
\]

\begin{rem}
\label{rem:kring_C_modules}
The category $\f (\calc)$ will also be referred to here as the category of $\kring \calc$-modules, where $\kring \calc$ is the $\kring$-linear category constructed from $\calc$. In the literature, the terminology $\calc$-modules is also used for this; this is avoided here since we also work with modules over $\kring$-linear categories which are not of the form $\kring \calc$. 
\end{rem}

\begin{exam}
\label{exam:fb-modules}
The category $\f (\fb)$ (or $\kring \fb$-modules) is equivalent to the category of sequences $M(n)$ of left $\kring \sym_n$-modules, for $n \in \nat$, and equivariant maps. Likewise, the category $\f (\fb\op)$ is equivalent to the category of sequences of right $\kring \sym_n$-modules. 

The isomorphism $\fb \cong \fb\op$ of categories (provided by the passage to inverse morphisms for the groupoid $\fb$) induces the isomorphism $\f (\fb) \cong \f (\fb\op)$. This corresponds (for each $n \in \nat$) to the isomorphism of categories between left $\kring \sym_n$-modules and right $\kring \sym_n$-modules obtained using the group isomorphism $\sym_n \cong \sym_n\op$ provided by the inverse $g \mapsto g^{-1}$. 

Given two $\fb$-modules $M_1$, $M_2$, by the above, one can consider $M_1$ as a $\fb\op$-module and then form $M_1 \otimes_\fb M_2$. This is given explicitly by 
\[
M_1 \otimes_\fb M_2 = \bigoplus_{n \in \nat} M_1 (n) \otimes_{\sym_n} M_2 (n).
\]
\end{exam}

\begin{rem}
\label{rem:fb_op}
The passage between $\f (\fb)$ and $\f (\fb\op)$-modules as above will be used freely.
\end{rem}

Post-composing with the $\kring$-linear duality functor $(-)^\sharp := \hom_{\kmod} (-, \kring)$ yields the  functors $\f (\calc) \op \rightarrow \f (\calc\op)$ and $\f (\calc\op ) \op \rightarrow \f (\calc)$, both denoted by $D$, since context should make clear what is intended. If $\kring$ is a field, these are adjoint and restrict to equivalences between the respective full subcategories of functors taking finite-dimensional values.

%%%%%%%%%%%%%%%%%%%%%%%%%%%%%%%%%%%%%%%%%%%%%%%%%%%%%%%%%%%%%%%%%%
\subsection{Organization of the paper}

An overview of functors on $\gr$ and $\gr\op$ is provided in Section \ref{sect:gr} and this is illustrated in Section \ref{sect:polyq}, where the polynomial filtration of the category $\kring \gr$ is introduced.

Section \ref{sect:propoly} introduces the category $\propoly$ that serves to encode the polynomial filtration of functors on $\gr$. This is applied in Section \ref{sect:group-ring}, where the group ring is considered as a functor on $\gr$. This feeds into the definition of the {\em Mal'cev} functors in Section \ref{sect:malcev} and their tensor products. 

Section \ref{sect:opd} is an interlude to introduce background on operadic structures. This is used in Section \ref{sect:endo_malcev} to relate $\propoly$ and right $\catlie$-modules.

Section \ref{sect:catlie} revisits the relationship between left $\catlie$-modules and analytic functors on $\gr\op$. Section \ref{sect:compare} then shows how this is related to the framework for $\propoly$ introduced here. These results are put together in Section \ref{sect:bifunctors} to treat bifunctors. 

Finally, Section \ref{sect:model_qgr_pbif} and Section \ref{sect:modular} apply the theory to the two fundamental examples.

%%%%%%%%%%%%%%%%%%%%%%%%%%%%%%%%%%%%%%%%%%%%%%%%%%%%%%%%%%%%%%%%%%%
\subsection{Acknowledgement}

The author is grateful to Christine Vespa, who drew Habiro and Massuyeau's work \cite{MR4321214} and   Katada's work \cite{2021arXiv210206382K} to his attention, in relation to the study of functors on $\gr$. This formed one of  the original motivations for studying the relationship between bifunctors and $\catlie$-bimodules. Moreover, the basic results on the polynomial filtration of functors on $\gr$ were first considered in the joint work with Vespa, \cite{PV}.

\tableofcontents

%\newpage
%\input{gr}
\section{Background on functors on $\gr$}
\label{sect:gr}

This section surveys  various flavours of functors on $\gr$, so as to make this paper reasonably self-contained. 
Most of this material is covered in \cite{PV} and elsewhere; in particular, no claim to originality is made.

Throughout $\kring$ is a unital commutative ring; for some arguments this will be taken to be a field and, for the analysis of the polynomial filtration, $\kring$ is taken to be  $\rat$, in which case the theory is particularly powerful.  $\f (\gr)$ denotes the category of functors from $\gr$ to $\kmod$ (respectively $\f (\gr\op)$ for contravariant functors). As in Section \ref{subsect:nota_conv}, the tensor product on $\kmod$ induces symmetric monoidal structures on $\f(\gr)$ and on $\f(\gr\op)$, denoted simply $\otimes$.

\begin{exam}
\label{exam:abelianization}
The abelianization functor $\A$ in $\f(\gr)$ is given by $\A (G) = G_\abel\otimes_\zed \kring$.
 The linear dual $\A ^\sharp$ in $\f(\gr\op)$ is given by $G \mapsto \hom (G, \kring) \cong \hom(G_\abel, \kring)$, where $\hom$ denotes homomorphisms of groups, using the underlying additive structure of $\kring$. 
\end{exam}

The free product defines a symmetric monoidal structure on $\gr$ and on $\gr\op$. This gives the respective shift functors $\tgr$ and $\tgrop$ defined by precomposition with $- \star \zed$; for instance, for $F\in \ob \f (\gr)$, one has $\tgr F (\zed^{\star r}) = F (\zed^{\star r} \star \zed) \cong  F (\zed^{\star r+1})$.
The associated difference functors $\dgr$ and $\dgrop$ are defined so that there are natural isomorphisms:
\begin{eqnarray*}
\tgr & \cong & \dgr \oplus \id_{\f(\gr)} \\
\tgrop & \cong & \dgrop \oplus \id_{\f(\gr\op)}.
\end{eqnarray*}

\begin{nota}
\label{nota:PZ_pbar}
Let $P_\zed$ denote the functor $H \mapsto \kring H$ (the  $\kring$-module generated by the underlying set of $H$). 
 This splits canonically in $\f (\gr)$ as $P_\zed \cong \kring \oplus \pbar$, where $\pbar$ is the reduced part (a functor $F$ on $\gr$ is reduced if $F(\zed^{\star 0}) =0$).
\end{nota}

\begin{rem}
\label{rem:PZ_pbar}
\ 
\begin{enumerate}
\item 
The functor $P_\zed$ is projective in $\f(\gr)$; it corepresents the evaluation functor $F \mapsto F (\zed)$.
\item 
For $H \in \ob \gr$, $P_\zed (H)$ is the $\kring$-module underlying the group ring $\kring H$; as explained in Section \ref{sect:group-ring}, this implies that $P_\zed$ is a cocommutative Hopf algebra in $\f (\gr)$.
\item 
The functor $\pbar$ identifies as the functor to $\kring$-modules underlying the augmentation ideal functor $H \mapsto \aug H$, where $\aug H := \ker \big ( \kring H \rightarrow \kring \big).$
\end{enumerate}
\end{rem}

%%%%%%%%%%%%%%%%%%%%%%%%%%%%%%%%%%%%%%%%%%%%%%%%%%%%%%%%%%%%%%%%%%%%%%%%%%%
\subsection{Exponential functors on $\gr\op$} 
\label{subsect:expo}

Suppose that $(\calm, \odot, \mathbbm{1})$ is a symmetric monoidal category, then one can consider Hopf algebras in $\calm$ and, in particular, cocommutative Hopf algebras in $\calm$.

Recall (see \cite{PV}, for example) that an exponential functor on $\gr\op$ with values in $\calm$ is a symmetric monoidal functor from $\gr\op$ to $\calm$.

\begin{nota}
\label{nota:Phi}
For $H$ a cocommutative Hopf algebra in the symmetric monoidal category $\calm$, $\Phi H : \gr\op \rightarrow \calm$ denotes the associated exponential functor. This has values given by 
 $
 \zed^{\star n} \mapsto H^{\odot n}$,  and with morphisms of $\gr\op$ acting via the Hopf algebra structure of $H$.
\end{nota}

The construction $\Phi$ is natural: it defines a functor from cocommutative Hopf algebras in $\calm$ to functors from $\gr\op$ to $\calm$. Moreover, it is natural with respect to the symmetric monoidal category $\calm$: namely, if $\alpha: \calm_1 \rightarrow \calm_2$ is a symmetric monoidal functor and $H$ is a cocommutative Hopf algebra in $\calm_1$ then  $\alpha H$ is a cocommutative Hopf algebra in $\calm_2$ and there is a natural isomorphism 
$
\alpha (\Phi H) 
\cong 
\Phi (\alpha H)$,
 where $\Phi H$ is constructed in $\calm_1$ and $\Phi (\alpha H)$ in $\calm_2$.

%%%%%%%%%%%%%%%%%%%%%%%%%%%%%%%%%%%%%%%%%%%%%%%%%%%%%%%%%%%%%%%%%%%%%%%%%%
\subsection{Bifunctors}
\label{subsect:bifunctors}

We  consider the category of bifunctors on $\gr$, by which we mean $\f (\gr\op\times \gr)$. The exterior tensor product $\boxtimes : \f(\gr\op) \times \f(\gr) \rightarrow \f (\gr\op \times \gr)$ gives one way of constructing bifunctors.

Another fundamental example of a bifunctor  is provided by the standard projectives $P_\Gamma$ of the category $\f (\gr)$:

\begin{nota}
\label{nota:pbif}
Denote by $\pbif$ the bifunctor 
$
\pbif : (\Gamma,H) 
\mapsto 
P_\Gamma (H) := 
\kring \hom_\gr (\Gamma, H). 
$
\end{nota}

Since $\star$ is the coproduct in $\gr$, there is a natural isomorphism:
\begin{eqnarray}
\label{eqn:P_bullet_star}
\quad 
\kring \hom_\gr (G_1 \star G_2, H) \cong 
\kring ( \hom_\gr (G_1 , H) \times \hom_\gr (G_2 , H))
\cong 
 \kring \hom_\gr (G_1 , H) \otimes \kring \hom_\gr (G_2 , H).
\end{eqnarray}

This property can be restated as:

\begin{prop}
\label{prop:P_expo}
Considered as a functor from $\gr \op$ to $\f (\gr)$, $\pbif$ is exponential. Explicitly, $\pbif \cong \Phi P_\zed$, where $P_\zed$ is considered as a cocommutative Hopf algebra in $\f (\gr)$.
\end{prop}

In particular, one has the following:

\begin{lem}
\label{lem:dgrop_P_bullet}
There are natural isomorphisms 
$\tgrop \pbif \cong  \pbif \otimes P_\zed $ and  $\dgrop \pbif \cong  \pbif \otimes \pbar$ 
 in $\f(\gr\op \times \gr)$,  
where $P_\zed$ and $\pbar$ are constant with respect to $\gr\op$.
\end{lem}

\begin{proof}
Using (\ref{eqn:P_bullet_star}), it is clear that $\tgrop \pbif \cong \pbif \otimes P_\zed$. The splitting $\tgrop \cong \dgrop \oplus \id$ then leads to the  result for $\dgrop$.   
\end{proof}

This is the main ingredient in the proof of the following standard result:

\begin{prop}
\label{prop:dgr_adjoint}
As endofunctors of $\f (\gr)$, the shift functor $\tgr$ is right adjoint to $ -\otimes P_\zed$ and the difference functor $\dgr$ is right adjoint to $- \otimes \pbar$. 
\end{prop}

%%%%%%%%%%%%%%%%%%%%%%%%%%%%%%%%%%%%%%%%%%%%%%%%%%%%%%%%%%%%%%%%%%%%%%%%%%ù
\subsection{Polynomial functors}

The functors $\dgr$ and $\dgrop$ allow the following definition of polynomiality:

\begin{defn}
\label{defn:polynomial}
\ 
\begin{enumerate}
\item 
A functor $F \in \ob \f(\gr)$ has polynomial degree $d \in \nat$ if $(\dgr)^{d+1} F = 0$.
\item 
A functor $F' \in \ob \f(\gr\op)$ has polynomial degree $d \in \nat$ if $(\dgrop)^{d+1} F' = 0$.
\end{enumerate}
The  full subcategories of polynomial functors of degree at most $d$ are denoted $\f_d (\gr)$ and $\f_d (\gr\op)$ 
 respectively.
\end{defn}

\begin{rem}
\ 
\begin{enumerate}
\item 
For $d \in \nat$, $\f_d (\gr)$ is an abelian subcategory of $\f(\gr)$ and $\f_d(\gr\op)$ an abelian  subcategory of $\f(\gr\op)$.
\item 
An alternative, equivalent definition uses the {\em cross-effect functors}, as in \cite{MR3340364}, generalizing the approach of Eilenberg and Mac Lane for functors on additive categories. 
\end{enumerate}
\end{rem}

Not all functors considered here are polynomial:

\begin{exam}
\label{exam:non_polynomiality}
\ 
\begin{enumerate}
\item 
The functor $P_\zed$ is not polynomial  in $\f (\gr)$: for $G \in \ob \gr$, the set map $G \rightarrow G \star \zed$ given by $g \mapsto g x$, where $x$ denotes the generator of $\zed$, is injective. On linearization, this induces an injective natural transformation $P_\zed \hookrightarrow \delta^\gr P_\zed$. Since $\delta ^\gr$ is exact, a straightforward induction shows that $(\delta^\gr)^{d+1}P_\zed$ is non-zero, for all $d \in \nat$. 
 More generally, for $\Gamma \in \ob \gr\op$ of positive rank, $P_\Gamma$ is not polynomial in $\f (\gr)$. 
\item 
For $G \in \ob \gr$ of positive rank, the functor   $\Gamma \mapsto P_\Gamma (G)$ is not polynomial  in $ \f (\gr\op)$. Indeed, by Lemma \ref{lem:dgrop_P_bullet}, $\dgrop \pbif (G) \cong  \pbif (G) \otimes \pbar (G)$. Since $G$ has positive rank, $\pbar (G)$ is a non-zero, free $\kring$-module. Iterating this shows that $(\dgrop)^{d+1} \pbif (G)$ is non-zero for all $d \in \nat$.
\end{enumerate}
\end{exam}

\begin{nota}
\label{nota:fin_poly}
Denote by 
\begin{enumerate}
\item 
$\f_{<\infty}(\gr)$ the full subcategory of $\f (\gr)$ given by $\bigcup_{d \in \nat} \f_d (\gr)$; 
\item 
$\f_{< \infty}(\gr\op)$ the full subcategory of $\f (\gr\op)$ given by $\bigcup_{d \in \nat} \f_d (\gr\op)$.
\end{enumerate}
\end{nota}

There are inclusions 
\begin{eqnarray*}
0 = \f_{-1} (\gr) \subset \f_0 (\gr) \subset \f_1 (\gr) \subset \ldots \subset \f_d (\gr) \subset \f_{d+1}(\gr) \subset \ldots \subset \f_{<\infty}(\gr) \subset \f(\gr)
\end{eqnarray*}
and likewise for $\f(\gr\op)$.

\begin{exam}
\label{exam:deg_leq_1}
Let $F$ be a  functor in $\f(\gr)$.
\begin{enumerate}
\item 
$F$ has polynomial degree $0$ if and only if it is constant. 
\item 
If $F$ is reduced (i.e., $F(\zed^{\star 0})=0$), then $F$ has polynomial degree $1$ if and only if $F$ is additive (i.e., the obvious morphisms induce an isomorphism $F (\zed^{\star r} \star \zed ^{\star s}) \cong F(\zed^{\star r} ) \oplus F(\zed^{\star s})$).
\end{enumerate}
The abelianization functor $\A$ in $\f(\gr)$ is reduced and additive; in particular it has  polynomial degree one. Likewise, the linear dual $\A ^\sharp$ in $\f(\gr\op)$ is reduced and additive, thus has  polynomial degree one. 
\end{exam}

The tensor product $\otimes$ is compatible with the polynomial filtration. Namely, for $d, e \in \nat$ the tensor product on $\f (\gr)$ restricts to 
 $
\f_d (\gr) \times \f_e (\gr) \stackrel{\otimes}{\longrightarrow} \f_{d+e} (\gr)$.
 The corresponding statement holds for functors on $\gr\op$. 

\begin{exam}
\label{exam:A_otimes_d}
For $d \in \nat$, the iterated tensor product $\A^{\otimes d}$ is polynomial of degree $d$, i.e.,  belongs to $\f_d (\gr)$. This is the fundamental example of a degree $d$ polynomial functor. Likewise, in the contravariant case, $(\A^\sharp)^{\otimes d}$ belongs to $\f_d (\gr\op)$.
\end{exam}

For $d \in \nat$, the inclusion $\f_d (\gr) \hookrightarrow \f (\gr)$ has both a left and a right adjoint. Here it is the right adjoint 
\[
\qgr_d : \f (\gr) \rightarrow \f_d (\gr)
\]
that is exploited; this will usually be considered as a functor $\qgr_d : \f(\gr) \rightarrow \f(\gr)$.
With this interpretation,  for a functor $F$, $\qgr_d F$ is the largest quotient of $F$ that has polynomial degree $d$, so that there is a canonical surjection $F \twoheadrightarrow \qgr_d F$, corresponding to the adjunction unit.

The inclusion $\f_d (\gr) \subset \f_{d+1} (\gr)$ then gives the commutative (up to natural isomorphism) diagram of  natural surjections 
\[
\xymatrix{
\id_{\f(\gr)}
\ar@{->>}[r]
\ar@/^1pc/@{->>}[rr]
&
\qgr_{d+1}
\ar@{->>}[r]
&
\qgr_d .
}
\]

\begin{exam}
\label{exam:qgr_PZ}
Applied to the standard projective $P_\zed$, one has $
\qgr_0 P_\zed \cong  \kring $
and $\qgr_1 P_\zed \cong  \A \oplus \kring$. 
Hence $\qgr_0 \pbar =0$ and $\qgr_1 \pbar =\A$. 

The corresponding surjection $P_\zed \twoheadrightarrow \A$ evaluated on $G\in \ob \gr$ is the map $\kring G \twoheadrightarrow \kring \otimes_\zed G_\abel$; this is the $\kring$-linearization of the map $G \rightarrow \kring \otimes_\zed G_\abel$, $g \mapsto 1 \otimes \overline{g}$, where $\overline{g}$ denotes the image of $g$ in $G_\abel$. This restricts to the surjection $\pbar \twoheadrightarrow \A$.

An analysis of the functors $\qgr_d \pbar$, for $d \in \nat$, is given in Section \ref{sect:group-ring}.
\end{exam}

The behaviour of  $\qgr_d$ is illustrated by the following:

\begin{lem}
\label{lem:qgr_tensor}
For $d < n \in \nat$, $\qgr_d (\A^{\otimes n}) =0 = \qgr_d (\pbar^{\otimes n})$.
\end{lem}

\begin{proof}
Consider the  composite surjection $\pbar^{\otimes n} \twoheadrightarrow \A^{\otimes n} \twoheadrightarrow \qgr_d (\A^{\otimes n})$, where the first map is the $n$-fold tensor product of the  surjection $\pbar \twoheadrightarrow \A$ of Example \ref{exam:qgr_PZ}. To show that $\qgr_d (\A^{\otimes n})=0$, it suffices to show that the composite is zero.  By the adjunction of Proposition \ref{prop:dgr_adjoint}, this is adjoint to a map $\kring \rightarrow (\dgr)^n \qgr_d (\A^{\otimes n})$. By construction, $\qgr_d (\A^{\otimes n})$ has polynomial degree $d$ hence, since $d<n$, by the definition of polynomiality, $(\dgr)^n \qgr_d (\A^{\otimes n})=0$. The result follows in this case; the argument for $\qgr_d (\pbar^{\otimes n})$ is similar.
\end{proof}

This leads to the following:

\begin{prop}
\label{prop:proj_gen_Fd} 
For $d \in \nat$, the functors $\qgr_d (\pbar^{\otimes n})$,  $0 \leq n \leq d$, form a set of projective generators of $\f_d (\gr)$. 
\end{prop}

\begin{proof}
Yoneda's lemma implies that the set of functors $P_{\zed^n}$, for $n \in \nat$, is a set of projective generators of $\f (\gr)$; from this, one deduces readily that $\pbar^{\otimes n}$, $n \in \nat$, also forms a set of projective generators. 

The functor $\qgr_d :\f (\gr) \rightarrow \f_d (\gr)$ is left adjoint to an exact functor, hence preserves projectives. Using Lemma \ref{lem:qgr_tensor}, it follows that $\{\qgr_d (\pbar^{\otimes n}) \ | \ 0 \leq n \leq d\}$ is a set of projective generators of $\f_d (\gr)$. 
\end{proof}

We record the following, which is clear from the constructions:

\begin{prop}
\label{prop:F_d_complete_cocomplete}
For $d \in \nat$ $\f_d (\gr)$ and $\f_d(\gr\op)$ are both complete and cocomplete.
\end{prop}

%%%%%%%%%%%%%%%%%%%%%%%%%%%%%%%%%%%%%%%%%%%%%%%%%%%%%%%%%%%%%%%%%%%%%%%%%
\subsection{The tensor product over $\gr$ and duality}
\label{subsect:tensor_gr}

The tensor product $\otimes_\gr : \f(\gr\op) \times \f (\gr) \rightarrow \kmod$ is defined as usual: this can be taken to be the composite of the external tensor product $\boxtimes : \f(\gr\op ) \times \f (\gr) \rightarrow \f (\gr\op \times \gr)$ followed by the coend $\f(\gr\op \times \gr) \rightarrow \kmod$. 

\begin{exam}
\label{exam:otimes_gr}
Consider the standard projective $P_\Gamma$, for $\Gamma \in \ob \gr\op$. Then $- \otimes _\gr P_\Gamma : \f (\gr\op) \rightarrow \kmod$ is naturally isomorphic to the evaluation functor $F \mapsto F(\Gamma)$, for $F \in \ob \f(\gr\op)$.

Likewise, for $H \in \ob\gr$, consider $P_{(-)}(H) = \kring \hom_\gr (-, H)$ in $\f (\gr\op)$, which is projective. The functor $P_{(-)}(H) \otimes_\gr - : \f(\gr) \rightarrow \kmod$ is naturally isomorphic to the evaluation functor $G \mapsto G(H)$, for $G \in \ob \f (\gr)$.  
\end{exam}

\begin{rem}
\label{rem:otimes_gr_colimits}
The bifunctor $\otimes_\gr$ commutes with colimits with respect to both variables. For current purposes it is colimits in $\f (\gr)$ that are of most interest.
\end{rem}

For the rest of this subsection, we suppose that $\kring$ is a field. Post-composing with vector space duality yields $D: \f (\gr)\op \rightarrow \f (\gr\op)$ and $D: \f (\gr\op) \op \rightarrow \f(\gr)$. These  are exact and are adjoint: for $F \in \ob \f(\gr\op)$ and $G \in \ob \f(\gr)$, there is a  natural isomorphism:
\[
\hom_{\f(\gr)}(G, DF) \cong \hom_{\f(\gr\op) } (F, DG).
\]
Moreover, the duality functors restrict to equivalences between the respective full subcategories of functors taking finite-dimensional values.

Duality restricts to the full subcategories of polynomial functors: 

\begin{lem}
\label{lem:duality_polynomial}
For $d \in \nat$, the duality adjunction restricts to
$
D : \f_d (\gr) \op \rightleftarrows \f_d (\gr\op) : D$.
\end{lem}

The tensor product $\otimes_\gr$ relates to duality as follows:

\begin{prop}
\label{prop:otimes_gr_duality}
For $F \in \ob \f(\gr\op)$ and $G \in \ob \f(\gr)$, there are natural isomorphisms:
\[
\hom_{\kmod} (F \otimes _\gr G, \kring) \cong 
\hom_{\f(\gr)}(G, DF) \cong \hom_{\f(\gr\op) } (F, DG).
\]
\end{prop}

Now consider bifunctors, i.e., the category $\f (\gr\times \gr\op)$. For two bifunctors $F_1$, $F_2$, one can form $F_1 \otimes_\gr F_2$, using the $\gr$-structure of $F_2$ and the $\gr\op$-structure of $F_1$. This has  a natural bifunctor structure, with $\gr$ acting via $F_1$ and $\gr\op$ via $F_2$. As usual, one has:

\begin{prop}
\label{prop:gr_bifunctors_monoidal}
The functor $\otimes_\gr$ yields a monoidal structure $(\f (\gr\times \gr\op) , \otimes_\gr, \pbif)$. 
\end{prop}

%%%%%%%%%%%%%%%%%%%%%%%%%%%%%%%%%%%%%%%%%%%%%%%%%%%%%%%%%%%%%%%%%%%%%%%%
\subsection{The polynomial filtration for functors on $\gr$}
\label{subsect:poly_Q}

For $d \in \nat$, the relationship between $\f_d(\gr)$ and the category $\kring \sym_d \dash \modules$ of $\kring \sym_d$-modules is important. This uses the functors introduced below, exploiting the action  of $\sym_d$ by place permutations on  $\A^{\otimes d}$. 

\begin{nota}
For $d \in \nat$, denote by 
\begin{enumerate}
\item 
$\alpha_d : \kring \sym_d \dash\modules \rightarrow \f_d (\gr)$ the functor $\A^{\otimes _d} \otimes_{\sym_d } -$; 
\item 
$\cre_d: \f_d (\gr) \rightarrow \kring \sym_d \dash\modules$ the functor $\hom_{\f_d (\gr)} (\A^{\otimes d}, -)$.
\end{enumerate}
\end{nota}

The functor $\cre_d$ is sometimes referred to as the {\em cross-effect} functor (more precisely, it is the restriction to $\f_d (\gr)$ of the $d$th cross-effect functor as in \cite{MR3340364}); it is exact. The functor $\alpha_d$ is left adjoint to $\cre_d$ and is right exact. The composite $\cre_d \alpha_d$ is naturally isomorphic to the identity on $\kring \sym_d$-modules via the unit of the adjunction;  for $F \in \ob \f_d (\gr)$, the adjunction counit
$
\alpha_d \cre_d F \rightarrow F
$ 
has  kernel and cokernel that are both polynomial of degree $d-1$. 

\begin{rem}
\label{rem:beta}
The cross-effect functor also has a right adjoint, $\beta_d : \kring \sym_d\dash \modules \rightarrow \f_d (\gr)$. This is of significant importance in the theory of polynomial functors, as exploited in \cite{PV} and \cite{2021arXiv211001934P}, for example. 
\end{rem}

We now specialize to $\kring = \rat$.

\begin{prop}
\label{prop:alpha_Q}
\cite{PV}
For $\kring = \rat$, 
\begin{enumerate}
\item 
the functor $\alpha_d$ is exact; 
\item 
the adjunction counit $\alpha_d \cre_d \rightarrow \id_{\f_d(\gr)}$ is injective and fits into the natural short exact sequence (for $F \in \ob \f_d (\gr)$):
\[
0
\rightarrow 
\alpha_d \cre_d F \rightarrow F
\rightarrow 
\qgr_{d-1} F 
\rightarrow 
0;
\]
\item 
$\alpha_d \cre_d F $ is semisimple and all of its composition factors have polynomial degree exactly $d$. 
\end{enumerate}
\end{prop}

\begin{rem}
A polynomial functor of the form $\alpha_d M$ for $M $ a $\rat \sym_d$-module will be referred to as a {\em homogeneous polynomial functor of degree $d$}. (Thus a functor $F \in \ob \f_d(\gr)$ is homogeneous  of degree $d$ if and only if the adjunction counit $\alpha_d \cre_d F \rightarrow F$ is an isomorphism.) 
\end{rem}

Now, for any $F \in \ob \f (\gr)$, one has the natural tower under $F$: 
\[
\ldots \twoheadrightarrow 
\qgr_{d+1}F 
\twoheadrightarrow 
\qgr_d F
\twoheadrightarrow 
\ldots 
\twoheadrightarrow 
\qgr_{-1} F=0.
\]
By considering the kernels of the natural surjections $F \twoheadrightarrow 
\qgr_d F$, this is equivalent to giving a decreasing filtration of $F$, the {\em polynomial filtration}. This filtration stabilizes if and only if $F$ is polynomial.

We introduce the following notation for the subquotients of the polynomial filtration:

\begin{nota}
\label{nota:qhat}
For $d\in \nat$, denote by $\qhat{d}$  the kernel of the natural surjection $\qgr_d \twoheadrightarrow \qgr_{d-1}$.
\end{nota}

The associated graded to the polynomial filtration of $F\in \ob \f (\gr)$ is thus 
$
\bigoplus_{d \in \nat} \qhat{d} F
$.
 The following is used in Corollary \ref{cor:fbcr} to encode this as a $\fb$-module:

\begin{prop}
\label{prop:qhat}
\cite{PV}
For $\kring =\rat$ and $d \in \nat$,
\begin{enumerate}
\item 
the functors $\qgr_d$ and  $\qhat{d}$ are exact on $\f_{< \infty} (\gr)$; 
\item 
for $F \in \f (\gr)$, there is a natural isomorphism $\qhat{d} F \cong \alpha_d \cre_d \qgr_d F$; in particular, $\qhat{d} F$ is a homogeneous polynomial functor of degree $d$.
\end{enumerate}
\end{prop}

The category $\f (\fb)$ used below was introduced in Example \ref{exam:fb-modules}.

\begin{nota}
\label{nota:fbcr}
Let $\fbcr : \f (\gr) \rightarrow \f (\fb)$ be the functor that sends $F$ to the $\fb$-module $\mathbf{t} \mapsto \cre_t \qgr_t F = \cre_t \qhat{t} F$, for $\mathbf{t} = \{ 1, \ldots , t \} \in \fb$.  
\end{nota}

The following consequence of Proposition \ref{prop:qhat} shows that $\fbcr$ determines the associated graded of the polynomial filtration.

\begin{cor}
\label{cor:fbcr}
Suppose $\kring= \rat$.
\begin{enumerate}
\item 
For $F \in \ob \f(\gr)$, the associated graded of the polynomial filtration is naturally isomorphic to 
\[
\bigoplus_{t \in \nat} \alpha_t \big((\fbcr F) (\mathbf{t})\big).
\]
\item 
The restriction of $\fbcr$ to $\f_{< \infty}(\gr)$ is exact.
\end{enumerate}
\end{cor}

%%%%%%%%%%%%%%%%%%%%%%%%%%%%%%%%%%%%%%%%%%%%%%%%%%%%%%%%%%%%%%%%%%%%%
\subsection{The polynomial filtration for $\f (\gr\op)$ and analytic functors} 
\label{subsect:analytic_grop}

In \cite{2021arXiv211001934P} it is the polynomial filtration for functors in $\f (\gr\op)$ that is used. 
When working with $\f (\gr\op)$, the relevant  counterpart of  $\qgr_d$ is the {\em  right} adjoint to the inclusion $\f_d (\gr\op) \hookrightarrow \f (\gr\op)$.  We also introduce the left adjoint, since this occurs when considering the behaviour of $\otimes _\gr$.
 
\begin{nota}
\label{nota:pgrop}
For $d \in \nat$, denote by
\begin{enumerate}
\item 
 $\pgrop_d : \f (\gr\op) \rightarrow \f_ d (\gr\op)$  the right adjoint to the inclusion $\f_d (\gr\op) \hookrightarrow \f (\gr\op)$;
 \item 
 $\qgrop_d : \f (\gr\op) \rightarrow \f_ d (\gr\op)$  the left adjoint to the inclusion $\f_d (\gr\op) \hookrightarrow \f (\gr\op)$.
\end{enumerate}
These may also be considered as functors $\f(\gr\op) \rightarrow \f (\gr\op)$.
\end{nota}

Every functor $G \in \ob \f (\gr\op)$ has a canonical polynomial filtration 
\[
0 = \pgrop _{-1} G \subset \pgrop_0 G \subset \pgrop _1 G \subset \ldots 
\subset \pgrop_d G \subset \ldots \subset G.
\]
The functor $G$ is {\em analytic} if the canonical morphism $\lim_{\substack{\rightarrow\\ d} } \pgrop_d G \rightarrow G$ is an isomorphism. 

\begin{defn}
\label{defn:analytic_grop}
The category $\f_\omega (\gr\op) \subset \f (\gr\op)$ is the full subcategory of analytic functors. 
\end{defn}

\begin{rem}
\label{rem:polynomial_filt_grop}
Properties of the polynomial filtration for $\f (\gr\op)$ are studied in 
 \cite{2021arXiv211001934P}; these are  analogous (via `duality') to those of the polynomial filtration of $\f (\gr)$.
\end{rem}

%\newpage
%\input{polyq}
\section{The tower of categories $\qgr_\bullet \kring \gr$}
\label{sect:polyq}

This section considers the bifunctor $\qgr_d \pbif$ (for $d \in \nat$) that is constructed from  $\pbif = \kring \hom_\gr (-,-)$. These bifunctors lead to the tower of categories $\qgr_\bullet \kring \gr$, providing `polynomial approximations' to $\kring \gr$. 

Here $\kring$ is taken to be a field, so as to simplify the exposition. For the main applications, it will be taken to be $\rat$.
%%%%%%%%%%%%%%%%%%%%%%%%%%%%%%%%%%%%%%%%%%%%%%%%%%%%%%%%%%
\subsection{Polynomial functors and $\otimes_\gr$}

Proposition \ref{prop:otimes_gr_duality} allows a quick proof of the following:

\begin{prop}
\label{prop:otimes_gr_poly}
For $F \in \ob \f(\gr\op)$ and $G \in \ob \f(\gr)$,
\begin{enumerate}
\item 
 if $F$ has polynomial degree $d$, then the canonical surjection $G \twoheadrightarrow \qgr_d G$ induces an isomorphism 
$
F \otimes _\gr G \stackrel{\cong}{\rightarrow} F \otimes_\gr \qgr_d G
$;
\item 
if $G$ has polynomial degree $d$, then the canonical surjection $F \twoheadrightarrow \qgrop_d F$ induces an isomorphism
$F \otimes_\gr G \stackrel{\cong}{\rightarrow} \qgrop_d F \otimes_\gr  G$.
\end{enumerate}
Hence, in general, there are natural isomorphisms:
\[
\qgrop_d F \otimes_\gr  G 
\cong 
\qgrop_d F \otimes_\gr \qgr_d G
\cong 
F \otimes_\gr \qgr_d G.
\]
\end{prop}

\begin{proof}
Consider the first statement. By Proposition \ref{prop:otimes_gr_duality}, there is a natural isomorphism 
$\hom_{\kring} (F \otimes_\gr G, \kring) \cong \hom_{\f(\gr)} (G, DF)$. Moreover, $DF$ has polynomial degree $d$ in $\f (\gr)$, by Lemma \ref{lem:duality_polynomial}. Hence,  there is a natural isomorphism $\hom _{\f(\gr)} (G, DF)\cong 
\hom_{\f(\gr)} (\qgr_d G, DF)$. Applying the isomorphism of Proposition \ref{prop:otimes_gr_duality} again gives $\hom_{\kring} (\qgr_d F \otimes_\gr G, \kring)$. More precisely, the induced map $F \otimes _\gr G \stackrel{\cong}{\rightarrow} F \otimes_\gr \qgr_d G$ gives an isomorphism on applying $\hom_\kring (-, \kring)$, hence is an isomorphism.

The proof of the second statement is categorically dual. The final statement then follows. 
\end{proof}

We now seek to apply this to bifunctors (i.e., working with $\f (\gr\op \times \gr)$).  The key example is the bifunctor $\pbif$ that sends $(\Gamma, H) \in \ob \gr\op \times \gr$ to $\kring \hom_{\gr} (\Gamma, H)$.

When working with bifunctors, we have the following, which follows directly from naturality:

\begin{lem}
\label{lem:qgr}
The functors  $\qgr_d$ and $\qgrop_d$, for   $d \in \nat$, induce endofunctors of $\f (\gr\op \times \gr)$.
Moreover, the natural surjections $\id \twoheadrightarrow \qgr_{d+1} \twoheadrightarrow \qgr_d $ are natural transformations of endofunctors of $\f (\gr\op \times \gr)$. 
\end{lem}

In particular, a bifunctor $B$ in $\f (\gr\op \times \gr)$ is said to be polynomial of degree $d$ with respect to $\gr$ (respectively $\gr\op$) if the natural surjection $B \twoheadrightarrow \qgr_d B$ (resp. $B \twoheadrightarrow \qgrop_d B$) is an isomorphism. Equivalently, $B$  is  polynomial of degree $d$ with respect to $\gr$ if, for every $\Gamma \in \ob \gr\op$, $B (\Gamma, -)$ (considered as an object of $\f (\gr)$) is polynomial of degree $d$. The corresponding statement holds for $\gr\op$.

\begin{prop}
\label{prop:pbif_bipolynomiality}
For $d \in \nat$, there is an isomorphism 
$
\qgr_d \pbif \cong \qgrop_d \pbif$. 
 In particular, $\qgr_d \pbif$ is polynomial with respect to $\gr\op$.

Moreover, for $F \in \ob \f_d (\gr\op)$ and $G \in \ob \f_d (\gr)$, there are natural isomorphisms:
\begin{eqnarray*}
F \otimes_\gr \qgr_d \pbif & \cong & F \\
\qgr_d \pbif \otimes _\gr G & \cong & G.
\end{eqnarray*}
\end{prop}

\begin{proof}
By Proposition \ref{prop:gr_bifunctors_monoidal}, $\pbif$ is the unit for the monoidal structure on $\f(\gr\op \times \gr)$ given by $\otimes_\gr$. In particular, this gives the isomorphisms $\pbif \otimes_\gr \qgr_d \pbif \cong \qgr_d \pbif$ and $\qgrop_d \pbif \otimes_\gr\pbif \cong \qgrop_d \pbif$.

Proposition \ref{prop:otimes_gr_poly} gives that both of the left hand expressions of the above isomorphisms are isomorphic to $\qgrop_d \pbif \otimes _\gr \qgr_d \pbif$. Hence the first statement follows by composing the isomorphisms of bifunctors. 

The remaining statements are proved similarly.
\end{proof}

\begin{rem}
\label{rem:monoidal_polynomial_bifunctors}
\ 
\begin{enumerate}
\item 
The fact that $\qgr_d$, $d \in \nat$ applied to $\pbif$ (with respect to $\gr\op$) yields an interesting tower of surjections 
 may appear surprising, since it is usually the increasing filtration associated to the functors $\pgrop_d$ that is considered when dealing with $\gr\op$.
\item
Proposition \ref{prop:pbif_bipolynomiality} implies that, restricting the monoidal structure $(\f (\gr\op \times \gr), \otimes_\gr , \pbif)$ to the full subcategory of $\f (\gr\op \times \gr)$ with objects bifunctors that are polynomial degree $d$ with respect to both $\gr$ and $\gr\op$ gives a monoidal structure with unit $\qgr_d \pbif$. 
\end{enumerate}
\end{rem}

\begin{exam}
\ 
\begin{enumerate}
\item 
The bifunctor $\qgr_0 \pbif $ is the constant bifunctor $\kring$. 
\item 
The bifunctor $\qgr_1 \pbif $ is isomorphic to $\kring \oplus (\A ^\sharp \boxtimes \A)$. This can also be viewed as $(\Gamma , H) \mapsto \kring \otimes \hom (\Gamma_\abel, H_\abel)$, where $(-)_\abel :  \gr \rightarrow \ab$ is the abelianization functor. The surjection to $\A ^\sharp \boxtimes \A$ is then induced by the $\kring$-linear extension of 
 $\hom_\gr (\Gamma, H) \rightarrow \hom (\Gamma_\abel, H_\abel) \rightarrow \kring \otimes \hom (\Gamma_\abel, H_\abel)$, where the first map is given by the abelianization functor $(-)_\abel$ and the second is induced by the unit of $\kring$. 
\end{enumerate}
\end{exam}

Proposition \ref{prop:pbif_bipolynomiality} provides the tower of quotients in $\f(\gr\op \times \gr)$:
\[
\xymatrix{
\pbif  
\ar@{.>>}[d]
\ar@{->>}[rd]
\ar@{.>>}[rrd]
\ar@{->>}[rrrd]
\ar@{->>}[rrrrd]
\ar@<.3ex>@{->>}[rrrrrd]
\\
\ldots
\ar@{->>}[r]
&
\qgr_d \pbif  
\ar@{->>}[r]
&
\ldots
\ar@{->>}[r]
&
\qgr_2 \pbif  
\ar@{->>}[r]
&
\qgr_1 \pbif  = \kring \oplus (\A ^\sharp \boxtimes \A)
\ar@{->>}[r]
&
\qgr_0 \pbif  = \kring,
}
\]
in which $\qgr_d \pbif  $ is polynomial of degree $d$ with respect to both $\gr$ and $\gr\op$.

We note the following:

\begin{prop}
\label{prop:qgr_dP_finite}
For $ d \in \nat$, the bifunctor $\qgr_d \pbif $ is finite (i.e., has a finite composition series). In particular, it takes finite-dimensional values.
\end{prop}

\begin{proof}
This can be proved by using the explicit analysis of the polynomial filtration carried out in 
 \cite{MR3505136}. 
\end{proof}

%%%%%%%%%%%%%%%%%%%%%%%%%%%%%%%%%%%%%%%%%%%%%%%%%%%
\subsection{The tower}
\label{subsect:tower_kgr}

Composition of morphisms in $\gr$ induces `composition maps' 
$
\pbif  \otimes  \pbif  \rightarrow \pbif 
$ 
for $\pbif $ by $\kring$-linearizing. Explicitly,  for $H \in \ob \gr$ one has  the morphism of bifunctors
\[
\kring \hom_\gr (H, -) \otimes \kring \hom_\gr (-, H) \rightarrow \kring \hom_\gr(-,-).
\]
Now, there is a  natural morphism $\kring \hom_\gr (H, -) \otimes \kring \hom_\gr (-, H) \rightarrow \kring \hom_\gr(-,-) \otimes_\gr \kring \hom_\gr (-,-) = \pbif \otimes_\gr \pbif $ and the latter is isomorphic to $\pbif$. 

For $d \in \nat$, this can be repeated, replacing $\pbif $ by $\qgr_d \pbif$. This gives the following commutative diagram:
\begin{eqnarray}
\label{eqn:diag_comp}
\xymatrix{
\pbif (H, -) \otimes \pbif (-,H)
\ar[r]
\ar@{->>}[d]
&
\pbif 
\ar@{->>}[d]
\\
\qgr_d \pbif (H,-) 
\otimes
\qgr_d \pbif (-,H) 
 \ar[r]
 &
\qgr_d \pbif 
}
\end{eqnarray}
where the vertical surjections are induced by the canonical map $\pbif \twoheadrightarrow \qgr_d \pbif$.
 (Note that $\pbif (H,-)$ is the same as $P_H (-)$ and $\pbif (-,H)$ is the same as $\pbif (H)$.)
 This diagram determines the lower horizontal `composition map'.

\begin{thm}
\label{thm:tower}
For $d \in \nat$, 
\begin{enumerate}
\item 
there is a  $\kring$-linear category $\qgr_d \kring \gr$ with objects $\ob \gr$ and such that, for $\Gamma, H \in \ob \gr$, $\hom_{\qgr_d \kring \gr} (\Gamma, H) = \qgr_d \pbif (\Gamma, H)$ and composition of morphisms is given by diagram (\ref{eqn:diag_comp}); 
\item
there is a full $\kring$-linear functor $\kring \gr \rightarrow \qgr_d \kring \gr$ that is the identity on objects and,  on morphisms, is the $\kring$-linear surjection 
$
\pbif  \twoheadrightarrow \qgr_d \pbif
$.
\end{enumerate}

These categories form a tower under $\kring \gr$, with projection
$\qgr_d \kring \gr \rightarrow \qgr_{d-1} \kring \gr$ that is the identity on objects and, on morphisms, is 
induced by the natural surjection $\qgr_d \twoheadrightarrow \qgr_{d-1}$:
\[
\xymatrix{
\kring \gr 
\ar@{.>}[d]
\ar[rd]
\ar@{.>}[rrd]
\ar[rrrd]
\ar[rrrrd]
\\
\ldots 
\ar[r]
&
\qgr_d \kring \gr
\ar[r]
&
\ldots 
\ar[r]
&
\qgr_1 \kring \gr
\ar[r]
&
\qgr_0 \kring \gr.
}
\]
\end{thm}

\begin{proof}
That $\qgr_d \kring \gr$ forms a category as stated is essentially a restatement of the assertion of Remark \ref{rem:monoidal_polynomial_bifunctors} that $\qgr_d \pbif$ forms the unit for $\otimes_\gr$ in the appropriate full subcategory of bifunctors. 

This provides the composition of morphisms in $\qgr_d \kring \gr$. The fact that composition is unital and associative follows from the corresponding properties for $\kring \gr$. 

That these constructions, for varying $d$, form a tower is clear. 
\end{proof}

\begin{rem}
This result can be proved directly by exploiting the definition of $\qgr_d$ as a left adjoint to construct the composition maps. The above approach has been preferred, since it explains why $\qgr_d \pbif$ is the natural object to consider. 
\end{rem}

%\newpage
%\input{propoly}
\section{Pro-polynomial functors on $\gr$}
\label{sect:propoly}

We introduce a framework for studying functors on $\gr$ that are {\em not} polynomial. This is based on using appropriate towers of polynomial functors on $\gr$,  such as those arising from the polynomial filtration of an object of $\f (\gr)$.

Throughout, $\kring$ is taken to be $\rat$, so that the results of Section \ref{subsect:poly_Q} apply. 

%%%%%%%%%%%%%%%%%%%%%%%%%%%%%%%%%%%%%%%%%%%%%%
\subsection{Introducing $\propoly$}

Consider $(\nat, \leq)$ as a poset and hence as a category. Thus $\f(\gr)^{\nat\op}$ is the category of towers in $\f(\gr)$ and this inherits an abelian structure from $\f(\gr)$. We use this to define the following category of pro-polynomial functors:

\begin{defn}
\label{defn:propoly}
Let $\propoly$ be the full subcategory of $\f(\gr)^{\nat\op}$ with objects $F_\bullet : \nat \op \rightarrow \f(\gr)$, $n \mapsto F_n$, such that, for all $d \in \nat$,
\begin{enumerate}
\item 
 $F_d \in \f_d (\gr)$; 
\item 
the morphism $\qgr_{d} F_{d+1} {\rightarrow} F_{d}$ induced by $F_{d+1} \rightarrow F_{d}$  is an isomorphism.
\end{enumerate}
\end{defn}

\begin{rem}
\label{rem:propoly_surj}
The hypothesis implies that, for each $d \in \nat$, the structure morphism $F_{d+1} \rightarrow F_{d}$ is surjective.
 In particular, the endofunctor of $\f(\gr)^{\nat\op}$ that sends a tower $F_\bullet$ to the tower $(\qgr_d F_d)$ (with the induced structure morphisms) does not take values in $\propoly$, since this surjectivity property is not in general satisfied.
\end{rem}

\begin{prop}
\label{prop:abelian_propoly}
The category $\propoly$ is an abelian, $\rat$-linear full subcategory of $ \f(\gr)^{\nat\op}$. Moreover, $\propoly$ is cocomplete.  
\end{prop}

\begin{proof}
For the first statement, it suffices to show that the kernel and cokernel of a morphism in $\propoly$ both lie in $\propoly$. The condition on the polynomiality is immediate, hence it suffices to check that the formation of the kernel and cokernel of a morphism between polynomial functors preserves the polynomial filtration in the appropriate sense. This follows since the functors $\qgr_d$ restricted to $\f_{< \infty}(\gr)$ are exact, by Proposition \ref{prop:qhat}.

That $\propoly$ is cocomplete is a consequence of the fact that $\f(\gr)^{\nat \op}$ is cocomplete (since $\f (\gr)$ is) and that the functor $\qgr_d$, $d \in \nat$, commutes with colimits, since it is a left adjoint.
\end{proof}

\begin{prop}
\label{prop:d_evaluate_exact_f_to_propoly}
\ 
\begin{enumerate}
\item 
For $d \in \nat$, the functor $\propoly \rightarrow \f_d (\gr)$ given by $F_\bullet \mapsto F_d$ is exact. 
\item
The functors $\qgr_n$, $n \in \nat$, induce a  $\rat$-linear functor $\qgr_\bullet : \f (\gr) \rightarrow \propoly$ given by 
$F \mapsto (\qgr_n F \ | \ n \in \nat)$. This is exact when restricted to $\f_{< \infty} (\gr)$.
\end{enumerate}
\end{prop}

\begin{proof}
The first statement is clear, as is the fact that the functors $\qgr_n$ give rise to a $\rat$-linear functor $\qgr_\bullet : \f (\gr) \rightarrow \propoly$. The exactness statement follows from that for the functors $\qgr_n$ given by Proposition \ref{prop:qhat}, which requires the restriction to $\f_{< \infty} (\gr)$.
\end{proof}

Restricting to polynomial functors, one has: 

\begin{prop}
\label{prop:poly_propoly}
For $d \in \nat$, the category $\f_d (\gr)$ is equivalent to the full subcategory of $\propoly$ with objects $F_\bullet$ such that the structure morphism $F_{n+1} \rightarrow F_n$ is the identity for all $n \geq d$.

Hence, $\f_{< \infty} (\gr)$ is equivalent to the full subcategory of $\propoly$ with objects $F_\bullet$ such that there exists $d \in \nat$ such that  the structure morphism $F_{n+1} \rightarrow F_n$ is the identity for all $n \geq d$.
\end{prop}

\begin{proof}
The restriction of the functor of Proposition \ref{prop:d_evaluate_exact_f_to_propoly},  $\qgr_\bullet : \f (\gr) \rightarrow \propoly$, to $\f_d (\gr)$ clearly takes values in the given full subcategory. We claim that this has quasi-inverse given by the restriction of the functor of Proposition \ref{prop:d_evaluate_exact_f_to_propoly}. This is verified as follows: for $F$ a functor of polynomial degree $d$, one has the natural isomorphism $(\qgr_\bullet F) _d \cong F$; for $F_\bullet \in \ob \propoly$ lying in the given full subcategory, the definition of $\propoly$ implies that there is a unique morphism $\qgr_\bullet (F_d) \rightarrow F_\bullet$ in $\propoly$ that is the identity on $F_d$. 

The second statement follows from the first. In particular, the quasi-inverse sends an object $(F_\bullet)$ of the given subcategory to $\lim_{\substack{\leftarrow \\ n }} F_n$, which identifies with $F_N$ for some $N \gg 0$. 
\end{proof}

%%%%%%%%%%%%%%%%%%%%%%%%%%%%%%%%%%%%%%%%%%%%%%%%%%%%%%%%%%%%%%%%%%%%
\subsection{The symmetric monoidal structure}

The category $\f(\gr)^{\nat\op}$ comes equipped with the `pointwise' tensor product; namely, for $F_\bullet $ and $G_\bullet$ two such functors, one has $F_\bullet \otimes G_\bullet$ such that, for all $n \in \nat$, $(F_\bullet \otimes G_\bullet)_n = F_n \otimes G_n$. However, $\propoly$ is not stable under this, so one introduces the following truncation, designed so that Proposition \ref{prop:propoly_sym_mon} below holds:

\begin{defn}
\label{defn:obar}
Let $\obar : \propoly \times \propoly \rightarrow \propoly$ be the functor defined on objects by 
$(F_\bullet \obar G_\bullet)_n:= \qgr_n (F_n \otimes G_n)$ and with structure morphism (for $n>0$) determined by the commutative diagram:
\[
\xymatrix{
F_n \otimes G_n 
\ar@{->>}[d]
\ar[r]
&
F_{n-1} \otimes G_{n-1}
\ar@{->>}[d]
\\
\qgr_n(F_n \otimes G_n) 
\ar@{.>}[r]
&
\qgr_{n-1}(F_{n-1} \otimes G_{n-1})
}
\]
in which the vertical maps are the canonical surjections, the top map is the tensor product of the respective structures maps and the dotted arrow is provided by the adjunction defining $\qgr_n$.
\end{defn}

\begin{rem}
\label{rem:obar_well_defined}
One must check that $\obar$ does indeed take values in $\propoly$, namely that, for $0< d \in \nat$, the canonical surjections $F_d \twoheadrightarrow \qgr_{d-1} F_d \cong F_{d-1}  $ and $G_d \twoheadrightarrow  \qgr_{d-1} G_d \cong G_{d-1}$ induce an isomorphism
\[
\qgr_{d-1} (F_d \otimes G_d)
\cong 
\qgr_{d-1} (F_{d-1} \otimes G_{d-1}).
\]
This is established using the properties of the polynomial filtration presented in Section \ref{subsect:poly_Q}, as follows.

By hypothesis one has the short exact sequence $0 \rightarrow \qhat{d} F_d \rightarrow F_{d} \rightarrow F_{d-1} \rightarrow 0$, where $F_d \rightarrow F_{d-1}$ is  the structure morphism.  Forming the tensor product with $G_d$ gives the short exact sequence in $\f_{< \infty}(\gr )$
\[
0 \rightarrow (\qhat{d} F_d) \otimes G_d \rightarrow F_{d} \otimes G_d  \rightarrow F_{d-1} \otimes G_d \rightarrow 0. 
\] 
Since $\qhat {d} F_d $ is homogeneous polynomial of degree $d$ and $G_d$ is polynomial, it is straightforward to see that $\qgr_{d-1} (  (\qhat{d} F_d) \otimes G_d) =0$. Since $\qgr_{d-1}$ is exact  (by Proposition \ref{prop:qhat}), the structure morphism $F_d \rightarrow F_{d-1}$ induces an isomorphism 
 $\qgr_{d-1} (F_d \otimes G_d) \cong \qgr_{d-1} (F_{d-1} \otimes G_d)$. Repeating the argument with the rôles of $F$ and $G$ reversed gives the required result.
\end{rem}

\begin{rem}
The definition of $\obar$ given in Definition \ref{defn:obar} adapts to give the natural symmetric monoidal structure  on $\f_d (\gr)$, for each $d \in \nat$. Namely, for $F, G \in \ob \f_d (\gr)$, 
$
F\obar_d G := \qgr_d (F \otimes G).
$ 
\end{rem}

In the following $\rat \in \ob \f(\gr)$ is considered as an object of $\propoly$ via the functor $\qgr_\bullet$ of Proposition \ref{prop:d_evaluate_exact_f_to_propoly}. Moreover, $\f_{< \infty}(\gr)$ is considered as symmetric monoidal for the structure inherited from $(\f(\gr), \otimes, \rat)$.
   
\begin{prop}
\label{prop:propoly_sym_mon}
The category $(\propoly, \obar ,\rat)$ is symmetric monoidal.  Moreover, the functor $\qgr_\bullet : \f_{< \infty}(\gr) \rightarrow \propoly$ is symmetric monoidal with respect to this structure.
\end{prop}

\begin{proof}
The first statement is clear. The fact that $\qgr_\bullet$ is symmetric monoidal follows by using the argument employed in Remark \ref{rem:obar_well_defined}. 
\end{proof}

%%%%%%%%%%%%%%%%%%%%%%%%%%%%%%%%%%%%%%%%%%%%%%%%%%%%%%%%%%%%%%%%%%%%%%%%%%ù
\subsection{Completion and projectives}

A further relationship between $\propoly$ and $\f (\gr)$ is provided by completion:

\begin{nota}
\label{nota:compl}
Denote by $\compl : \propoly \rightarrow \f (\gr) $ the completion functor given by $F_\bullet \mapsto \lim_{\substack{\leftarrow \\ n }} F_n$. 
\end{nota}

\begin{prop}
\label{prop:compl_right_adjoint}
The completion functor $\compl : \propoly \rightarrow \f (\gr) $ is right adjoint to $\qgr_\bullet : \f (\gr) \rightarrow \propoly$. 
\end{prop}

\begin{proof}
By definition of the category $\propoly$, for any two objects $G'_\bullet$ and $G_\bullet$,  $\hom_\propoly (G'_\bullet, G_\bullet)$ is given by the equalizer of the usual diagram:
\[
\prod_{n\in \nat} \hom_{\f(\gr)} (G'_n, G_n) \rightrightarrows \prod_{j\in \nat} \hom_{\f(\gr)} (G'_{j+1}, G_j).
\]

Taking $G'_\bullet = \qgr_\bullet F$, for $F \in \f (\gr)$, using the definition of $\qgr_n$ as the left adjoint to the inclusion $\f_n (\gr) \hookrightarrow \f (\gr)$, this diagram can be rewritten naturally as: 
\[
\prod_{n\in \nat} \hom_{\f(\gr)} (F, G_n) \rightrightarrows \prod_{j\in \nat} \hom_{\f(\gr)} (F, G_j),
\]
with structure morphisms induced by the inverse system $\hom_{\f(\gr)} (F, G_\bullet)$. The equalizer thus identifies with 
$\lim_\leftarrow \hom_{\f(\gr)} (F, G_\bullet)$. By the universal property of the inverse limit defining $\compl G_\bullet$, the latter is naturally isomorphic to $\hom_{\f(\gr)} (F, \compl G_\bullet)$, as required.
\end{proof}

\begin{rem}
Composing the functor  $\qgr_\bullet : \f (\gr) \rightarrow \propoly$ with the completion functor gives the composite
 $\f (\gr) \rightarrow \f (\gr)$ that sends a functor $F$ to the inverse limit $
 \lim_{\substack{\leftarrow \\ n }} \qgr_n F
 $  of the polynomial filtration of $F$. 

The adjunction unit is the natural map $F \rightarrow \compl (\qgr_\bullet F)$ induced by the canonical surjections $F \twoheadrightarrow q_n F$. (The behaviour of this morphism in the  case $F= P_\zed$ is explained in Corollary \ref{cor:compl_qgr_unit}.)
\end{rem}

\begin{prop}
\label{prop:compl_exact}
The functor $\compl : \propoly \rightarrow \f (\gr) $ is exact.
\end{prop}

\begin{proof}
This follows from the Mittag-Leffler condition by using the fact (cf. Remark \ref{rem:propoly_surj}) that the structure morphisms in the tower of an object of $\propoly$ are surjective. This implies that $\lim_\leftarrow^1$ vanishes, as required.
\end{proof}

We note the following consequence:

\begin{lem}
\label{lem:surj_compl_eval}
For $G_\bullet$ an object of $\propoly$ and $n \in \nat$, the canonical morphism $\compl G_\bullet \rightarrow G_n$ is surjective. 
\end{lem}

\begin{proof}
For the purposes of this proof, write $\overline{G}_\bullet$ for the quotient of $G_\bullet$ with $\overline{G}_t$ equal to $G_t$ for $t \leq n$ and $G_n$ for $t \geq n$, with the obvious structure morphisms. Applying $\compl$ (which is exact by Proposition \ref{prop:compl_exact}) gives the surjection $\compl G_\bullet \twoheadrightarrow \compl \overline{G}_\bullet$. 
The codomain is canonically isomorphic to $G_n$, by construction of $\overline{G}_\bullet$, and the corresponding map identifies with that of the statement, which is thus surjective, as required.
\end{proof}

Combining Proposition \ref{prop:compl_right_adjoint} with Proposition \ref{prop:compl_exact} allows the identification of a set of projective generators of $\propoly$:

\begin{cor}
\label{cor:proj_gen_propoly}
For $\Gamma$ a finite rank free group, $\qgr_\bullet P_\Gamma$ is projective in $\propoly$; it corepresents the functor $G_\bullet \mapsto (\compl G_\bullet) (\Gamma)$. 

The set of functors $\qgr_\bullet P_{\zed^{\star n}}$, for $n \in \nat$, is a set of projective generators for $\propoly$. 
\end{cor}

\begin{proof}
The first statement is an immediate consequence of Propositions \ref{prop:compl_right_adjoint} and  \ref{prop:compl_exact}.

For the second statement, consider an object $G_\bullet$ of $\propoly$. Since $\f (\gr)$ has set of projective generators $\{ P_{\zed^{\star n}}\ | \ n \in \nat \}$, there exists a set of finite rank free groups $\{ \Gamma_i \  | \ i \in \mathcal{I} \}$ and a surjection 
\begin{eqnarray}
\label{eqn:proj_compl}
\bigoplus_{i \in \mathcal{I}} P_{\Gamma_i} \twoheadrightarrow \compl G_\bullet.
\end{eqnarray}
In particular, for each $n \in \nat$, composing with the canonical surjection $\compl G_\bullet \twoheadrightarrow G_n$ given by Lemma \ref{lem:surj_compl_eval}, this induces a surjection $\bigoplus_{i \in \mathcal{I}} P_{\Gamma_i} \twoheadrightarrow  G_n$.

Using the adjunction (which implies that $\qgr_\bullet$ commutes with coproducts), (\ref{eqn:proj_compl}) yields:
$
\bigoplus_{i \in \mathcal{I}} \qgr_\bullet P_{\Gamma_i} \rightarrow G_\bullet$.
  This is surjective since, by the above observation,  the composite 
$\bigoplus_{i \in \mathcal{I}} P_{\Gamma_i} 
\twoheadrightarrow \bigoplus_{i \in \mathcal{I}} \qgr_n P_{\Gamma_i} 
\twoheadrightarrow  G_n$ is surjective (using the canonical factorization) for each $n \in \nat$.
\end{proof}

%%%%%%%%%%%%%%%%%%%%%%%%%%%%%%%%%%%%%%%%%%%%%%%%%%%%%%%%%%%%%%%%%%%%%%%%%%%%%%%
\subsection{Relating to analytic functors on $\gr\op$}

 That the category $\f_\omega (\gr\op)$ of analytic functors on $\gr\op$ (introduced in Section \ref{subsect:analytic_grop}) is related to $\propoly$ via duality $D$ is manifest. This is made explicit by the following statement:

\begin{lem}
\label{lem:propoly_analytic_duality}
\ 
\begin{enumerate}
\item 
The duality functor $D :\f(\gr) \op \rightarrow \f(\gr\op)$ induces an exact functor $D: \propoly \op \rightarrow \f_\omega (\gr\op)$, given on objects by $D(G_\bullet):= \lim_{\substack{\rightarrow \\ n}} (DG_n)$.
\item 
The duality functor $D: \f(\gr\op) \op \rightarrow \f (\gr)$ induces an exact functor $D  : \f _\omega (\gr\op) \op \rightarrow \propoly$, given on objects by $(DF)_n := D (p_n F)$.
\end{enumerate}
\end{lem}

\begin{proof}
Consider the functor $D  : \f _\omega (\gr\op) \op \rightarrow \propoly$. For an analytic functor $F$ on $\gr\op$, we require to show that $DF$ as defined is an object of $\propoly$. By exactness of $D$ on $\f(\gr\op)$, for any $n \in \nat$, it is clear that $\pgrop_{n-1} F \hookrightarrow \pgrop_n F$ induces a surjection $D (\pgrop_n F)\twoheadrightarrow D(\pgrop_{n-1}F)$, with the domain of polynomial degree $n$ and codomain of degree $n-1$. It remains to show that this induces an isomorphism $\qgr_{n-1} D(\pgrop_n F) \cong D(\pgrop_{n-1}F)$. 

For this, consider the short exact sequence $0 \rightarrow \pgrop_{n-1} F \rightarrow \pgrop_n F \rightarrow (\pgrop_n F) / (\pgrop_{n-1} F) \rightarrow 0$. By construction of the polynomial filtration, $(\pgrop_n F) / (\pgrop_{n-1} F)$ is a {\em homogeneous} polynomial functor of degree $n$. On applying $D$, which is exact, the short exact sequence gives:
\[
0 
\rightarrow D ((\pgrop_n F) / (\pgrop_{n-1} F)) 
\rightarrow 
D (\pgrop_n F)
\rightarrow 
D(\pgrop_{n-1}F)
\rightarrow 
0.
\]
By the above, $D ((\pgrop_n F) / (\pgrop_{n-1} F))$ is a   a {\em homogeneous} polynomial functor of degree $n$, in particular applying $\qgr_{n-1}$ to this gives $0$. The  result follows.
\end{proof}

\begin{rem}
\label{rem:cts_dual}
The functor $D :\propoly \op \rightarrow \f_\omega (\gr\op)$ may be thought of as a  `continuous dual'.
\end{rem}

\begin{prop}
\label{prop:D_propoly_analytic_adjoint}
The duality functors of Lemma \ref{lem:propoly_analytic_duality} are adjoint:
\[
D : \propoly \op \rightleftarrows \f_\omega (\gr\op) : D.
\]
\end{prop}

\begin{proof}
This adjunction extends the duality adjunctions $D : \f_d (\gr)\op \rightleftarrows \f_d (\gr\op) :D$ between the categories of polynomial functors (for $d \in \nat$) as follows. Consider $G_\bullet \in \ob \propoly$ and $F \in \ob \f_\omega (\gr\op)$, then by definition, $\hom_{\propoly} (G_\bullet , DF)$ is the equalizer of the associated diagram:
\[
\prod _n  \hom_{\f (\gr) } (G_n, D (\pgrop_n F)) \rightrightarrows \prod _j \hom_{\f(\gr) } (G_{j+1}, D (\pgrop_jF)).
\]
By the duality adjunction, this can be rewritten as
\[
\prod _n  \hom_{\f (\gr) } (\pgrop_n F, D G_n) \rightrightarrows \prod _j \hom_{\f(\gr) } (\pgrop_j F, DG_{j+1}).
\]
Since $D(G_\bullet)$ is (by construction) analytic on $\gr\op$, with polynomial filtration such that $\pgrop_n (DG_\bullet) = D G_n$, using the defining property of the functors $\pgrop_n$ this can be rewritten as:
 \[
\prod _n  \hom_{\f (\gr) } (\pgrop_n F, D (G_\bullet)) \rightrightarrows \prod _j \hom_{\f(\gr) } (\pgrop_j F, D(G_\bullet)).
\]
This has equalizer $\lim_{\substack{\leftarrow \\ n}} \hom_{\f (\gr) } (\pgrop_n F, D (G_\bullet))$, which is isomorphic to $\hom_{\f_\omega (\gr)} (F, D (G_\bullet))$, since $F$ is analytic, by hypothesis.

This provides the required natural isomorphism $ \hom_{\propoly} (G_\bullet , DF) \cong \hom_{\f_\omega (\gr)} (F, D (G_\bullet))$.
\end{proof}

This leads to the following version of $\otimes_\gr$ for $\propoly$ and $\f_\omega(\gr\op)$:

\begin{prop}
\label{prop:otimes_gr_propoly_analytic}
The tensor product $\otimes_\gr : \f (\gr\op) \times \f (\gr) \rightarrow \kmod$ induces 
\[
\ogr : \f_\omega (\gr\op) \times \propoly \rightarrow \kmod
\]
given explicitly for $F \in \ob \f_\omega (\gr\op)$ and $G_\bullet \in \ob \propoly$ by:
\[
F \ogr G_\bullet := \lim_{\substack{\rightarrow \\ n}} (\pgrop_n F \otimes _\gr G_n),
\]
where the morphisms of the direct system are induced by the structure morphisms.

There are natural isomorphisms:
\[
\hom_{\kmod} (F \ogr G_\bullet, \kring)
\cong 
  \hom_{\propoly} (G_\bullet , DF) \cong \hom_{\f_\omega (\gr)} (F, D (G_\bullet)).
\]
\end{prop}

\begin{proof}
This is a direct consequence of the isomorphism given by Proposition \ref{prop:otimes_gr_poly}. To illustrate this, consider the structure morphisms of the direct system defining $F \otimes_\gr G_\bullet$. For $n \in\nat$, the natural inclusion $\pgrop_n F \hookrightarrow \pgrop_{n+1} F$ and the natural surjection $G_{n+1} \twoheadrightarrow G_n \cong \qgr_n (G_{n+1}) $ give the (natural) solid arrows in 
\[
\xymatrix{
\pgrop_n F \otimes_\gr G_{n+1} 
\ar[r]
\ar[d]_\cong 
&
\pgrop_{n+1} F \otimes_\gr G_{n+1}
\\
 \pgrop_n F \otimes_\gr G_{n},
 \ar@{.>}[ur]
}
\]
where the isomorphism is provided by Proposition \ref{prop:otimes_gr_poly}. This yields the required structure morphism, indicated by the dotted arrow making the diagram commute.

The given natural isomorphisms extend those of Corollary \ref{prop:otimes_gr_duality}, arguing as in the proof of Proposition \ref{prop:D_propoly_analytic_adjoint}.
\end{proof}

One may also consider $\otimes_\gr$ restricted to $\f_\omega (\gr\op) \times \f (\gr)$. This is compatible with the functor of Proposition \ref{prop:otimes_gr_propoly_analytic}:

\begin{prop}
\label{prop:compatibility_otimes_gr}
For $F \in \ob \f_\omega (\gr\op) $ and $G \in \ob \f(\gr)$, there is a natural isomorphism:
\[
F \otimes_\gr G \cong F \ogr \qgr_\bullet G.
\]
\end{prop}

%%%%%%%%%%%%%%%%%%%%%%%%%%%%%%%%%%%%%%%%%%%%%%%%%%%%%%%%%%%%%%%%%%%%%%%%%%%%%%%%
\subsection{The associated graded}

There is a functor $\propoly \rightarrow \f (\gr)$ that generalizes the construction of the associated graded of the polynomial filtration. This is given by 
\[
G_\bullet 
\mapsto 
\grad (G_\bullet):= 
\bigoplus_{n \in \nat} \ker \big( G_n \rightarrow G_{n-1} \big).
\]
This can be rewritten $
\grad (G_\bullet):= \bigoplus_{n \in \nat} \qhat{n} G_n.
$

As in Notation \ref{nota:fbcr}, by composing with the cross-effect functors, one obtains a functor with values in $\rat\fb$-modules (aka. $\f (\fb)$):

\begin{nota}
\label{nota:fbcr_propoly}
Denote by $\fbcr : \propoly \rightarrow \f (\fb)$ the functor that associates to $G_\bullet \in \ob \propoly$ the $\fb$-module $\mathbf{t} \mapsto \cre_t \qhat{t} (G_t) \cong \cre_t \qgr_t (G_t) \cong \cre_t G_t$.
\end{nota}

\begin{rem}
Using the same notation $\fbcr$ as in Notation \ref{nota:fbcr} should cause no confusion, since these functors are compatible via $\qgr_\bullet : \f (\gr) \rightarrow \propoly$.
\end{rem}

The category  of $\rat \fb$-modules is symmetric monoidal with respect to the Day convolution product $\odot$. Here it is useful to recall its construction in terms of representations of the symmetric groups: a $\rat\fb$-module can be  considered as as sequence of representations of the symmetric groups $M(n)$, $n \in \nat$. For two such, $M(-)$ and $N(-)$, the Day convolution $M\odot N$ has $n$th term
\[
\bigoplus _{s+t=n} M(s) \otimes N(t) \uparrow_{\sym_s \times \sym_t}^{\sym_n}.
\]
The unit is the $\rat\fb$-module $\kring$ (with $\kring (0) = \kring$ and $\kring (s)=0$ otherwise).

\begin{prop}
\label{prop:properties_fbcr_propoly}
The functor $\fbcr : \propoly \rightarrow \f (\fb)$ is exact. Moreover, it is symmetric monoidal with respect to the structures $(\propoly , \obar , \kring) $ and $(\f (\fb), \odot , \kring)$.
\end{prop}

\begin{proof}
That $\fbcr$ is exact follows from the fact that $\cre_d$ is exact (as recalled in Section \ref{subsect:poly_Q}). It remains to show that it is symmetric monoidal. 

First consider the behaviour of the cross-effect functor $\cre_d$ on the tensor product of two homogeneous polynomial functors $F:= \alpha_s M(s) =\A^{\otimes s} \otimes_{\sym_s}  M(s)$ and $G:= \alpha _t N(t) = \A^{\otimes t} \otimes_{\sym_t} N(t)$, for a $\sym_s$-module $M(s)$ and $\sym_t$-module $N(t)$. By construction, $\cre_j \qhat{j} F$ is zero unless $j=s$, when one recovers $M(s)$; similarly for $G$.

Now, $F \otimes G = (\A^{\otimes s} \otimes_{\sym_s}  M(s)) \otimes (\A^{\otimes t} \otimes_{\sym_t} N(t) \cong 
(\A^{\otimes s+t} \otimes_{\sym_s \times \sym_t} (M(s)\otimes N(t))$, so that 
$
F \otimes G
\cong \A^{\otimes s+t} \otimes_{\sym_{s+t}} (M(s)\otimes N(t))\uparrow_{\sym_s \times \sym_t} ^{\sym_{s+t}}$.

Hence, for $d \in \nat$, and the above homogeneous polynomial functors $F$, $G$:
\[
\cre_d (F \otimes G) 
\cong 
\left\{ 
\begin{array}{ll}
(M(s) \otimes N(t))\uparrow_{\sym_s \times \sym_t} ^{\sym_{d}} & d = s+t \\
0 & \mbox{otherwise}.
\end{array}
\right.
\]
From this, one deduces the result in the homogeneous case.

Now suppose that $F$ and $G$ are polynomial. Using the finite length polynomial filtrations (as  in Section \ref{subsect:poly_Q}) of $F$ and $G$ respectively and the fact that $\cre_d$ is exact, the above extends to give the natural isomorphism
\[
\cre_d \qgr_d(F \otimes G) \cong \bigoplus_{s+t=d} (\cre_s \qgr_s F \otimes \cre_t  \qgr_t G)\uparrow_{\sym_s \times \sym_t} ^{\sym_{d}}.
\] 

From this, it is straightforward to deduce the required natural isomorphism for $F$ and $G$ in $\propoly$:
\[
\fbcr (F \obar G) 
\cong 
\fbcr F \odot \fbcr G
\]
and that this makes $\fbcr$ a symmetric monoidal functor.
\end{proof}

It is a useful  fact that $\fbcr$ is conservative:

\begin{prop}
\label{prop:fbcr_conservative}
An object $G_\bullet$ of $\propoly$ is zero if and only if $\fbcr (G_\bullet) =0$. 

Hence a morphism $f : G_\bullet \rightarrow G'_\bullet$ is an isomorphism (resp. surjective, resp. injective) if and only if $\fbcr f$ is an isomorphism (resp. surjective, resp. injective).
\end{prop}

\begin{proof}
Clearly, if $G_\bullet=0$, then $\fbcr (G_\bullet)=0$. Otherwise, choose $\ell$ minimal such that $G_\ell \neq 0$. Then, by minimality of $\ell$, $\qhat{\ell} G_{\ell} \neq 0$ so that $\cre_\ell G_{\ell} \neq 0$. It follows that $\fbcr (G_\bullet)$ is non-zero. 

The remaining statements follow using that $\fbcr$ is exact, by Proposition \ref{prop:properties_fbcr_propoly}.
\end{proof}

%\newpage
%\input{group_ring}
\section{The group ring functor}
\label{sect:group-ring}

The purpose of this section is to review the structure of the group ring functors $G \mapsto \rat G$ when restricted to $\gr$, by using the polynomial filtration and the category $\propoly$ introduced in Section \ref{sect:propoly}.  In particular, this introduces the Passi functors.

%%%%%%%%%%%%%%%%%%%%%%%%%%%%%%%%%%%%%%%%%%%%%%%%%%%%%%%%%%%%
\subsection{Recollections}

For $G$ a group, the group ring $\kring G$ is equipped with a natural cocommutative Hopf algebra structure. In particular, the unit is induced by the inclusion $\{ e \} \subset G$ and the augmentation (or counit) by the canonical surjection of groups $G \twoheadrightarrow \{ e\}$. The product $\kring G \otimes \kring G \rightarrow \kring G$ is induced by the multiplication of the group $G$ and the coproduct makes the generators `grouplike' (i.e., $[g] \mapsto [g] \otimes [g]$).  

This structure gives the augmentation ideal $\aug G \lhd \kring G$; this has basis given by the elements $[g] - [e]$, for $g \in G \backslash \{ e \}$. One has the (natural) filtration of $\kring G$ by powers of the augmentation ideal:
\[
\ldots \subset \aug^n G \subset \aug^{n-1} G \subset \ldots \subset \aug G \subset \kring G.
\]
The associated graded ring $\grad (\kring G) $ is $\bigoplus_{n \in \nat} \aug^n G/ \aug^{n+1} G$; this is natural in $G$ and it is convenient to write $ \aug^n/ \aug^{n+1} (G)$ for $ \aug^n G/ \aug^{n+1} G$.

\begin{rem}
\label{rem:passi_functors}
For $d \in \nat$, $G \mapsto \kring G / \aug^{d+1} G$ is sometimes termed the $d$th Passi functor, in reference to Passi's work \cite{MR537126}.
\end{rem}

\begin{rem}
Restricting to $G$ in $\gr$, this structure transposes to the functor category $\f (\gr)$. The underlying functor is the projective $P_\zed$. This is equipped with a cocommutative Hopf algebra structure in $\f (\gr)$: the unit and counit induce (and are determined by) the natural splitting $P_\zed \cong \kring \oplus \pbar$, the product is a natural transformation $P_\zed \otimes P_\zed \rightarrow P_\zed$ (unital and associative), and the coproduct a natural transformation $P_\zed \rightarrow P_\zed \otimes P_\zed$ (counital, cocommutative and coassociative); conjugation is the natural transformation $P_\zed \rightarrow P_\zed$ corepresented by $[-x]$, for $x$ the chosen generator of $\zed$.

The product restricts to $\pbar \otimes \pbar \rightarrow \pbar$ and, under the above identifications, the functor $G \mapsto \aug^n G$ corresponds to the  image of the iterated multiplication 
$
\pbar^{\otimes n} 
\rightarrow 
\pbar 
\subset P_\zed.
$
\end{rem}

The quotient $\aug / \aug^2 (G) $ is naturally isomorphic to $\kring \otimes _\zed G_\abel$, where $G_\abel$ is the abelianization of $G$. Restricted to $\gr$, this has the following interpretation:

\begin{lem}
\label{lem:resolution_A}
The product of $P_\zed$ gives the exact sequence
$ 
\pbar \otimes \pbar \rightarrow \pbar \rightarrow \A \rightarrow 
0
$ 
which induces the isomorphism: $ \aug / \aug^2 (G) \cong \kring \otimes _\zed G_\abel$.
\end{lem}

%%%%%%%%%%%%%%%%%%%%%%%%%%%%%%%%%%%%%%%%%%%%%%%%%%%%%%%%%%%%%%%%%%
\subsection{The associated graded and more}

We use the following notation for the lower central series of a group (which is functorial):

\begin{nota}
\label{nota:lcs}
For $G$ a group, let $(\gamma_t G)$ denote the lower central series of $G$, defined recursively by $\gamma_1 G = G$ and $\gamma_{i+1} G = [\gamma_i G, G]$. 
\end{nota}

Recall the definition of the dimension subgroups of $G$ (with respect to $\kring$): the $n$th dimension subgroup is
\[
D_{n,\kring} (G) := G \cap (1 + \aug^n G) = \{ g \in G \ | \ [g] - [e] \in \aug^n G \}.
\]
This is a normal subgroup of $G$ and $D_{n+1,\kring}(G) \subseteq D_{n,\kring} (G)$; moreover, $[D_{i,\kring} (G) , D_{j,\kring} (G)] \subseteq D_{i+j,\kring} (G)$ so that there is a natural inclusion 
$ 
\gamma_t G \subseteq D_{t,\kring} G,
$  for all $0< t \in \nat$.

\begin{rem}
Jennings \cite{MR68540} identified   the dimension subgroups  over $\kring = \rat$ as 
$
D_{n , \rat} (G) = \sqrt {\gamma_n G},
$ 
where, for a subgroup $H \subseteq G$, $\sqrt H = \{ x \in G \ |\  \exists 0<n \in \nat, x^n \in H \}$. 
\end{rem}

\begin{exam}
\label{exam:magnus}
For $G$  a finite rank free group, Magnus \cite{MR1581549} showed that $ D_{t,\zed}G =\gamma_t G $, $\forall 0< t \in \nat$. 
 One also has $ D_{t,\rat}G =\gamma_t G $, $\forall 0< t \in \nat$. 
\end{exam}

The associated Lie algebra $\lr_\kring G$ (defined over $\zed$)  is the associated graded:
\[
\lr_\kring G := \bigoplus _{0< t \in \nat} D_{t,\kring} G / D_{t+1,\kring} G
\]
with Lie bracket induced by the group commutator. By construction, there is a map $D_{t, \kring} G \rightarrow \aug^t G$ given by $g \mapsto [g] -[e]$ and this induces a morphism of Lie algebras 
$
\lr _\kring G \rightarrow \grad (\kring G)$, 
where  the algebra $\grad (\kring G)$ is considered as a Lie algebra for the commutator Lie bracket. 

Quillen  proved:

\begin{thm}
\label{thm:quillen_assoc_graded}
\cite{MR231919}
For $\kring$ a field of characteristic zero, the induced morphism of Hopf algebras
$
U (\kring \otimes_\zed \lr_\kring G ) \rightarrow \grad (\kring G)
$
 is an isomorphism.
\end{thm}

\begin{exam}
For $G$ a finite rank free group, by the work of Magnus, $\lr G$ is isomorphic to the free Lie algebra (over $\zed$) on the abelianization $G_\abel$ of $G$ (see, for example, \cite[Theorem 5.12]{MR2109550}).  Thus, for $\kring $ a field of characteristic zero, Quillen's theorem gives the isomorphism of Hopf algebras
\[
U (\liealg (G_\abel \otimes \kring) ) 
\stackrel{\cong}{\rightarrow}
\grad (\kring G),
\]
where $\liealg (G_\abel \otimes \kring)$ is the free Lie algebra (over $\kring$) on $G_\abel \otimes \kring$. In particular, this shows that $\grad (\kring G)$ is isomorphic to the tensor Hopf algebra on $G_\abel \otimes \kring$. These isomorphisms are natural with respect to $G \in \ob \gr$.
\end{exam}

Essentially by construction of the dimension subgroups, one has the following (the proof  is left to the reader):

\begin{prop}
\label{prop:passi_nilpotent}
For $1 \leq t \leq n \in \nat$ and a group $G$, the canonical surjection $G \twoheadrightarrow G/ D_{n,\kring} G$ induces a natural isomorphism of $\kring$-algebras:
\[
\kring G / \aug^t G 
\stackrel{\cong}{\rightarrow} 
\kring (G/ D_{n,\kring} G) / \aug^t (G/D_{n,\kring} G).
\]

Hence, the canonical projection $G \twoheadrightarrow G/ \gamma_n G$ induces a natural isomorphism of $\kring$-algebras:
\[
\kring G / \aug^t G 
\stackrel{\cong}{\rightarrow} 
\kring (G/ \gamma_n G) / \aug^t (G/\gamma_n G).
\]
\end{prop}

\begin{rem}
Since $\gamma_t G \subseteq D_{t, \kring} G$, Proposition \ref{prop:passi_nilpotent} means that, when studying the Passi functor $G \mapsto \kring G / \aug^t G$, one can reduce to working with nilpotent groups of class $t-1$. 

In particular, for $G$ a finite rank free group and $\kring = \rat$, since $\gamma_t G = D_{t, \rat} G$, one can replace $G$ by $G / \gamma_t G$. The group $G/ \gamma_t G$ is a finitely-generated free nilpotent group of class $t-1$; it is torsion-free.
\end{rem}

%%%%%%%%%%%%%%%%%%%%%%%%%%%%%%%%%%%%%%%%%%%%%%%%%%%%%%%%%%%%%%%%%%%%%%%%%%%%%%%%%%%%%%%%
\subsection{Completion of the group ring}

Quillen \cite[Appendix A]{MR258031} showed the interest of studying the completion of the group ring $\kring G$ with respect to the $\aug G$-adic filtration: 
\[
\widehat{\kring G} = \lim_{\substack{\leftarrow\\ t}}\kring G/ \aug^t G, 
\]
with augmentation ideal $\widehat{\aug G}$.

This has the structure of a complete Hopf algebra, with coproduct  $
\Delta : 
\widehat{\kring G}
\rightarrow 
\widehat{\kring G}
\hat{\otimes}
\widehat{\kring G}
$.
 One has the primitive elements $\prim \widehat{\kring G}$, defined as $\{ x \in \widehat{\aug G} \ | \  \Delta x = x \hat{\otimes} 1 + 1 \hat{\otimes} x\}$ and the grouplike elements $\gplike \widehat{\kring G}$, defined as $\{ x \in 1 + \widehat{\aug G} \ | \ \Delta x = x \hat {\otimes } x \}$. The primitives yield a Lie subalgebra of $\widehat{\kring G}$ and the grouplike elements form a subgroup of $\widehat{\kring G} ^{\times}$. There is a natural map $G \rightarrow \gplike\widehat{\kring G}$. 
 
When $\kring$ is a field of characteristic zero, one has $\log$ and $\exp$ defined by the usual power series; Quillen \cite[Proposition A.2.6]{MR258031} showed that these induce filtration-preserving isomorphisms:
\[
\log : \gplike \widehat{\kring G} \stackrel{\cong}{\leftrightarrow} \prim \widehat {\kring G} : \exp.
\]

\begin{rem}
\label{rem:quillen_case_G_nilpotent}
Quillen explained the relationship between the passage from $G$ to $\gplike \widehat{\rat G}$ and Mal'cev completion of nilpotent groups. (The functor $G \mapsto \gplike \widehat{\rat G}$ is presented as Mal'cev completion in \cite[Chapter 8]{MR3643404}.)

Recall that a nilpotent group  is uniquely divisible if the map $x \mapsto x^n$ is bijective for each $0 \neq n \in \zed$. 
 In \cite[Corollary A.3.8]{MR258031}, Quillen established that, for $G$ a nilpotent group  and $j : G \rightarrow \gplike \widehat{\rat G}$ the canonical map, $j$ is the universal map to a nilpotent, uniquely divisible group and $j$ is characterized up to canonical isomorphism by the following conditions: 
\begin{enumerate}
\item 
$\gplike \widehat{\rat G}$ is nilpotent, uniquely divisible;
\item 
$\ker (j)$ is the torsion subgroup of $G$; 
\item 
for $x \in \gplike \widehat{\rat G}$, there exists $0 \neq n \in \nat$ such that $x^n \in \mathrm{image} j$.  
\end{enumerate}
Hence $\gplike \widehat{\rat G}$ is the Mal'cev completion of the nilpotent group $G$. For example, if $G$ is abelian, then $\gplike \widehat{\rat G}$ is isomorphic to $G \otimes_\zed \rat$.
\end{rem}

For a finite rank free group and working over a field of characteristic zero, these structures are well understood:

\begin{exam}
\label{exam:complete_free_group}
\cite[Example A.2.11]{MR258031}
For $G$ a finite rank free group, $\widehat{\rat G}$ is isomorphic as a complete Hopf algebra to the completed tensor Hopf algebra $\widehat{T} (G_\abel \otimes \rat)$. Under this isomorphism, $\prim \widehat{\rat G}$ is isomorphic to the completed free Lie algebra $\widehat{\liealg} (G_\abel \otimes \rat)$ and there is an isomorphism of complete Hopf algebras
\[
\widehat{U} (\liealg (G_\abel \otimes \rat)) 
\stackrel{\cong}{\rightarrow} 
\widehat{\rat G},
\]
where $\widehat{U}$ denotes the completion of the universal enveloping algebra with respect to the augmentation ideal filtration.
\end{exam}

\begin{prop}
\label{prop:torsion-free_nilp_prim}
Let $G$ be a finitely-generated, torsion-free nilpotent group.  Then $\prim \widehat{\rat G}$ is generated as a $\rat$-vector space by the elements $\{ \log (g) \ | \ g \in G \}$. 
\end{prop}

\begin{proof}
The canonical map $G \rightarrow \gplike \widehat{\rat G}$  is a model for the Mal'cev completion of the finitely-generated, torsion-free nilpotent group $G$ (see Remark \ref{rem:quillen_case_G_nilpotent}). In particular, for each $x \in \gplike \widehat{\rat G}$, there exists $0<n \in \nat$ such that $x^n \in G$. Since $\log (x^n)= n \log (x)$, it follows that $\log (x)$ is in the $\rat$-vector space generated by  $\{ \log (g) \ | \ g \in G \}$.

Since $\log$ induces an isomorphism between $\gplike \widehat{\rat G}$ and $\prim \widehat{\rat G}$, the result follows.
\end{proof}

%%%%%%%%%%%%%%%%%%%%%%%%%%%%%%%%%%%%%%%%%%(See%%%%%%%%%%%%%%%%%%%%%%%%%%%%%%%%%%%%%%%%%%%%%%%%%%%%%%
\subsection{The Passi functors restricted to $\gr$}

On restricting to the category of functors on $\gr$, we can exploit the polynomial filtration to study the underlying functor of the group ring functor $G \mapsto \kring G$, which identifies as above with $P_\zed$.

The Passi functors  are closely related to the polynomial filtration of $P_\zed$, as explained in  \cite{MR3505136}:

\begin{prop}
\label{prop:passi_filt_polynomial}
For $d \in \nat$, the Passi functor $G \mapsto \kring G / \aug^{d+1} G$ on $\gr$ is isomorphic to $\qgr_d P_\zed$ and the tower of Passi functors under $\kring G$ identifies with the polynomial tower 
$$
\ldots \twoheadrightarrow \qgr_n P_\zed 
\twoheadrightarrow \qgr_{n-1} P_\zed 
\twoheadrightarrow 
\ldots 
$$ under $P_\zed$.

The associated graded, $G \mapsto \grad (\kring G)$, is naturally isomorphic to the Tensor algebra $T (\A):= \bigoplus_{n \in \nat } \A^{\otimes n}$, considered as a cocommutative Hopf algebra.
\end{prop}

By Proposition \ref{prop:propoly_sym_mon}, $\qgr_\bullet$ is symmetric monoidal when restricted to $\f_{<\infty}(\gr)$. We would like to be able to apply this to tensor products of $P_\zed$, but this does not belong to $\f _{<\infty} (\gr)$. Instead, we proceed directly: 

\begin{prop}
\label{prop:qgr_tensor_Pzed}
For $d, t \in \nat$, the natural projection $P_\zed \twoheadrightarrow \qgr_d P_\zed$ induces an isomorphism:
\[
\qgr_d (P_\zed^{\otimes t}) 
\stackrel{\cong}{\rightarrow}
\qgr_d ((\qgr_d P_\zed)^{\otimes t} ). 
\]

Hence, there is a natural isomorphism $\qgr_\bullet (P_\zed ^{\otimes t} ) \cong (\qgr_\bullet P_\zed)^{\obar t}$ and this is $\sym_t$-equivariant.
\end{prop}

\begin{proof}
The final conclusion follows directly from the first statement, so we concentrate on this. 
By construction, the surjection $P_\zed \twoheadrightarrow \qgr_d P_\zed$ induces a surjection 
\[
\qgr_d (P_\zed^{\otimes t})
\twoheadrightarrow 
\qgr_d ((\qgr_d P_\zed)^{\otimes t} ).
\]
We require to show that this is an isomorphism. For this, it is sufficient to show that the canonical projection 
$P_\zed^{\otimes t} \twoheadrightarrow \qgr_d (P_\zed^{\otimes t})$ factors across the surjection $P_\zed^{\otimes t} \twoheadrightarrow (\qgr_d P_\zed)^{\otimes t}$. 
 Since $\qgr_d (P_\zed^{\otimes t})$ is polynomial of degree $d$, it suffices to show that any map $P_{\zed}^{\otimes t} \rightarrow F$, where $F$ has polynomial degree $d$, factors across $(\qgr_d P_\zed)^{\otimes t}$

By Proposition \ref{prop:passi_filt_polynomial}, there is an exact sequence in $\f (\gr)$ 
\[
\pbar^{\otimes d+1} \rightarrow P_\zed \rightarrow \qgr_d P_\zed \rightarrow 0,
\]
where the first map is induced by the product of $P_\zed$. This induces the following presentation of $(\qgr_d P_\zed ) ^{\otimes t}$:
\[
\bigoplus _{i=1}^t 
P_\zed ^{\otimes i-1} \otimes 
(\pbar^{\otimes d+1}) 
\otimes P_\zed ^{\otimes t-i}
\rightarrow P_\zed^{\otimes t} \rightarrow (\qgr_d P_\zed)^{\otimes t} \rightarrow 0,
\]

To prove the result, it suffices to show that, for $F$ polynomial of degree $d$ and any $i \in \{1, \ldots , t\}$:
\[
\hom_{\f(\gr)} (P_\zed ^{\otimes i-1} \otimes 
(\pbar^{\otimes d+1}) 
\otimes P_\zed ^{\otimes t-i}, F ) =0.
\]

By Proposition \ref{prop:dgr_adjoint}, one has  $\hom_{\f(\gr)} (P_\zed ^{\otimes i-1} \otimes 
(\pbar^{\otimes d+1}) 
\otimes P_\zed ^{\otimes t-i}, F ) \cong \hom_{\f(\gr)} (P_\zed ^{\otimes t-1} , (\dgr)^{d+1} F )$. Since $F$ has polynomial degree $d$  by hypothesis, $(\dgr)^{d+1} F=0$ and the result follows. 
\end{proof}

The category $\propoly$ provides a way of modelling the completed group ring functor (restricted to $\gr$):

\begin{prop}
\label{prop:P_zed_propoly}
The cocommutative Hopf algebra structure of $P_\zed$ in $\f (\gr)$ induces a cocommutative Hopf algebra structure on  $\qgr_\bullet P_\zed$ in $\propoly$ (with respect to $\obar$).

In particular, the coproduct induces the reduced coproduct:
\[
\overline{\Delta} : \qgr_\bullet \pbar \rightarrow \qgr _\bullet \pbar \obar \qgr_\bullet \pbar.
\]

Moreover, there is a natural  isomorphism $
\widehat{\rat G} \cong \compl (\qgr_\bullet P_\zed) (G)
$ of complete Hopf algebras,  for $G \in \ob \gr$.
\end{prop}

\begin{proof}
The first statement follows by using the properties of $\qgr_\bullet$, in particular
Proposition \ref{prop:qgr_tensor_Pzed}. The final statement is an immediate consequence of the identification of the Passi filtration given in Proposition \ref{prop:passi_filt_polynomial}. 
\end{proof}

\begin{cor}
\label{cor:compl_qgr_unit}
The adjunction unit $P_\zed \rightarrow \compl (\qgr_\bullet P_\zed)$ evaluated on $G \in \ob \gr$ identifies with the canonical morphism
$
\rat G \rightarrow \widehat{\rat G}.
$ 
In particular,  $P_\zed \rightarrow \compl (\qgr_\bullet P_\zed)$ is injective. 
\end{cor}

\begin{proof}
The first statement follows from Proposition \ref{prop:P_zed_propoly}, together with the fact that the adjunction unit is induced by the canonical surjections $P_\zed \twoheadrightarrow \qgr_d P_\zed$.  The injectivity follows from the fact that finitely-generated free groups are residually nilpotent, as established by Magnus.
\end{proof}

%%%%%%%%%%%%%%%%%%%%%%%%%%%%%%%%%%%%%%%%%%%%%%%%%%%%%%%%%%%%%%%%%%%%%%%%%%%%%%%%%%%%%%%%
\subsection{Dualizing the group ring functor}

One can dualize the functor $G \mapsto \kring G$ (here $G$ can be any discrete group and $\kring$ an arbitrary commutative unital ring); this gives the functor $G \mapsto \kring ^G$, the set of maps from $G$, since $\hom_\kring (\kring G, \kring) \cong \kring ^G$. This gives a contravariant functor of $G$.

However, for current purposes, the functor $G \mapsto \kring ^G$ is too large. Instead, we use a `continuous dual' with respect to the filtration by powers of the augmentation ideal. 

\begin{defn}
\label{defn:pmap}
For $s \in \nat$, let 
\begin{enumerate}
\item
$\pmap_s (G, \kring)$ be the submodule of $\kring ^G$ given by $\hom_\kring (\kring G / \aug^{s+1} G, \kring)$;
\item 
$\pmap_\infty (G, \kring)$ be the colimit $\lim_{\substack{\rightarrow \\ s} } \pmap _s (G, \kring)$, for the direct system induced by the filtration by powers of the augmentation ideal.
\end{enumerate}
\end{defn} 

\begin{rem}
The notation $\pmap_s (G, \kring)$ reflects the fact that this can be viewed as the module of `polynomial maps of degree $s$' from $G$ to $\kring$.
\end{rem}

By construction, $G \mapsto \pmap_s (G, \kring)$ and $G \mapsto \pmap _\infty (G, \kring) $ are contravariant functors of $G$. Moreover, there are natural inclusions: 
\[
\pmap_s (G, \kring ) \hookrightarrow \pmap _\infty (G, \kring) \hookrightarrow \kring ^G.
\] 

Restricting to $\gr$ and taking $\kring = \rat$, one identifies:

\begin{prop}
\label{prop:dualize_passi}
For $s\in \nat$, there are natural isomorphisms:
\begin{eqnarray*}
\pmap _s (-, \rat) & \cong & D (\qgr_s P_\zed) \\
\pmap_\infty (-, \rat) & \cong & D (\qgr_\bullet P_\zed),
\end{eqnarray*}
where for the second isomorphism, $D$ is the duality functor $\propoly \op \rightarrow \f_\omega (\gr\op)$.

In particular, $\pmap _s (-, \rat)$ is a polynomial functor of degree $s$ and $\pmap_\infty (-, \rat)$ is analytic and is not polynomial.
\end{prop}

\begin{rem}
Proposition \ref{prop:dualize_passi} generalizes to consider the duals of $\qgr_s \pbif$ and $\qgr_\bullet \pbif$ respectively, carrying over the bifunctoriality. If one neglects the contravariant functoriality of $\pbif$, for $n \in \nat$, the functor $D (\qgr_s P_{\zed^{\star n}})$ is isomorphic to $G \mapsto \pmap_s (G^{\times n}, \rat)$ and $D (\qgr_\bullet P_{\zed^{\star n}})$ is isomorphic to $G \mapsto \pmap _\infty (G^{\times n}, \rat)$.
\end{rem}

\begin{cor}
\label{cor:injectivity_pmap_infty}
For $n \in \nat$, the functor $G \mapsto \pmap_\infty (G^{\times n}, \rat)$ is injective in $\f_\omega (\gr\op;\rat)$. 
\end{cor}

\begin{proof}
Since $\qgr_\bullet P_{\zed^{\star n}}$ is projective in $\propoly$, by Corollary \ref{cor:proj_gen_propoly}, this follows from the duality adjunction of Proposition \ref{prop:D_propoly_analytic_adjoint}, together with 
the exactness of the duality functors given by Lemma \ref{lem:propoly_analytic_duality}.
\end{proof}

%%%%%%%%%%%%%%%%%%%%%%%%%%%%%%%%%%%%%%%%%%%%%%%%%%%%%%%%%%%%%%%%%%%%%%%%%%%%%%%%%%%%%%%%
\subsection{An explicit description of $\qgr_d P_\zed$}
\label{subsect:describe_qgr_d_PZ}

Fix $0< d \in \nat$ and work over $\rat$, for convenience. By Proposition \ref{prop:passi_filt_polynomial}, the associated graded of the polynomial filtration of $\qgr_d \pbar$ is:
\[
\grad( \qgr _d P_\zed) \cong \bigoplus_{i=0}^d \A^{\otimes i}.
\]
The right hand side is well-understood as a functor on $\gr$. However, this is not sufficient to describe the full structure of $\qgr_d P_\zed$.

Recall that $\pbar (G) \cong \aug G$, for $G \in \ob \gr$, and has basis $\{[g]- [e] \ | \ g \in G \backslash \{ e \} \}$. Moreover, one has the following standard relation between the product in $\aug G$ and that in $G$ (this has already been used implicitly in identifying the associated graded in Proposition \ref{prop:passi_filt_polynomial}):

\begin{lem}
\label{lem:aug_G_relations}
For $g, h \in G$, one has
\[
([g]- [e])
([h]- [e])
=
([gh]- [e])
-([g]- [e])
-([h]- [e]).
\]
In particular,
\[
([g]- [e])
([g^{-1}]- [e])
=
-([g]- [e])
-([g^{-1}]- [e])
=
([g^{-1}]- [e])([g]- [e]).
\]
Hence, for any $1 \leq t \in \nat$:
\[
 [g_i^{-1}] - [e] = \sum_{s=1}^t (-1)^s ([g_i] - [e]) ^s
 + (-1)^t  ([g_i] - [e]) ^t ([g^{-1}] - [e])
 .
 \]
\end{lem}

To proceed, we need to recall the structure of  $\gr$ (cf. \cite[Appendix A]{2021arXiv211001934P} for example). Consider the free group functor 
$
\mathrm{Free} : \finset \rightarrow \gr
$,  
where $\finset$ is the category of finite sets. This preserves coproducts, is faithful and is essentially surjective on objects. However, it is not full. 

Using the symmetric monoidal structure of $\gr$, the remaining morphisms can be obtained from:
\begin{eqnarray*}
p : \zed & \rightarrow & \{e \} \\
\chi : \zed & \rightarrow & \zed \\
\nabla : \zed & \rightarrow & \zed \star \zed
\end{eqnarray*}
where $\chi$ is the passage to the inverse $x \mapsto x^{-1}$ and $\nabla$ is the cogroup structure of $\zed$ given by sending the generator $x \in \zed$ to $x_1 x_2$, where $x_i$ are the generators of the copies of $\zed$ in the free product. 

For a $\rat$-vector space $V$, write $T^{\leq d} (V)$ for the quotient of the tensor algebra $T(V)$ obtained by killing the two-sided ideal generated by $V^{\otimes d+1}$, so that $V \mapsto T^{\leq d}(V)$ is a functor that takes values in associative algebras.

\begin{prop}
\label{prop:describe_qgr_Pzed}
For $1<d \in \nat$, $\qgr_d P_\zed$ is isomorphic to the functor with values in associative algebras 
\[
\mathrm{Free}(S) \mapsto T^{\leq d} (\rat S),
\]
where $\rat S$ is the $\rat$-vector space generated by the finite set $S$, such that:
\begin{enumerate}
\item
morphisms in the image of $\mathrm{free} : \finset \rightarrow \gr$ act via $\rat( -) \ : \ \finset \rightarrow \mathrm{mod}_\rat$ and the naturality of $T^{\leq d} (-)$; 
\item 
$p : \zed = \mathrm{Free} (x) \rightarrow \{ e \} = \mathrm{Free} (\emptyset)$ sends $x \in T^{\leq d} (\rat\{x\})$ to zero;
\item 
$\chi : \zed = \mathrm{Free} (x) \rightarrow \zed = \mathrm{Free}(x)$ sends $x \in T^{\leq d} (\rat\{x\})$ to 
$\sum_{s=1}^d (-1)^s x^s$; 
\item 
$\nabla : \zed = \mathrm{Free} (x) \rightarrow \zed \star \zed = \mathrm{Free} (x_1, x_2)$ acts via $x \mapsto x_1 + x_2 + x_1 x_2$. 
\end{enumerate}
\end{prop}

\begin{proof}
For a finite set $S$, the universal property of the truncated tensor algebra $T^{\leq d}(-)$ functor induces 
\[
T^{\leq d} (\rat S) 
\rightarrow 
\qgr_d P_\zed (\mathrm{Free}(S)) 
\]
by sending a generator corresponding to $x \in S$ to the class of $[x]-[e]$. This is clearly a natural transformation with respect to $\finset$ and is an isomorphism by Proposition \ref{prop:passi_filt_polynomial}.

It remains to check that it is a natural transformation with respect to $\gr$, using the structure on the domain defined in the statement. This follows immediately from the relations given in Lemma \ref{lem:aug_G_relations}. 
\end{proof}

%\newpage
%\input{malcev}
\section{The Mal'cev functors}
\label{sect:malcev}

This section serves to introduce a convenient set of projective generators of the category $\propoly$, working over $\kring = \rat$. Namely, $\malcev$ is introduced as a Lie algebra in $\propoly$; the underlying object of $\propoly$ is projective   and the  tensor products $\malcev ^{\obar s}$, $s \in \nat$, provide a family of projective generators.

%%%%%%%%%%%%%%%%%%%%%%%%%%%%%%%%%%%%%%%%%%%%%%%%%%%%%%%%%%%%%%%%%%%%%%%%%%%%%%%%%%%%%%%
\subsection{Introducing $\malcev$}

There are various ways of introducing the Mal'cev Lie algebras; here we adapt the approach of Quillen \cite[Appendix A]{MR258031} (as reviewed in Section \ref{sect:group-ring}), who defined the Mal'cev Lie algebra of a group $G$ as $\prim \widehat{\rat G}$, the primitives of the completed group ring.

We restrict to the category $\gr$ and  work with the category $\propoly$ (equipped with its symmetric monoidal structure $\obar$) rather than completing. Recall from Proposition \ref{prop:P_zed_propoly} that $\qgr_\bullet P_\zed$ is a cocommutative Hopf algebra in $\propoly$.

\begin{defn}
\label{defn:malcev}
The Mal'cev functor $\malcev$ is the kernel (in $\propoly$) of the reduced coproduct: 
\[
\overline{\Delta} : \qgr_\bullet \pbar \rightarrow \qgr _\bullet \pbar \obar \qgr_\bullet \pbar.
\]
For $d \in \nat$, the $d$th Mal'cev functor is $\malcev _d \in \ob \f_d (\gr)$ (i.e., the $d$th polynomial component of $\malcev$).
\end{defn}

There are natural inclusions $\malcev \hookrightarrow \qgr_\bullet\pbar \subset \qgr_\bullet P_\zed$ and  $\malcev _d \hookrightarrow \qgr_d \pbar \subset \qgr_d P_\zed$, for $d \in \nat$.

\begin{prop}
\label{prop:malcev_first_properties}
\ 
\begin{enumerate}
\item 
The Mal'cev functor $\malcev$ is a Lie algebra in $\propoly$. 
\item 
The inclusion $\malcev \hookrightarrow \qgr_\bullet P_\zed$ is a morphism of Lie algebras, where $\qgr_\bullet P_\zed$ is equipped with the commutator Lie structure arising from its associative algebra structure given by Proposition \ref{prop:P_zed_propoly}. 
\item
The completion $\compl (\malcev)$ is naturally isomorphic to the underlying functor of $G \mapsto \prim \widehat{\rat G} $.
\end{enumerate}
\end{prop}

\begin{proof}
The first two statements follow from the fact that $\malcev$ is defined as the primitives of the Hopf algebra structure of $\qgr_\bullet P_\zed$ given by Proposition \ref{prop:P_zed_propoly}. 

On applying the completion functor, one recovers $\prim \widehat{\rat G}$, using that $\compl (\qgr _\bullet P_\zed \obar \qgr_\bullet P_\zed)$ yields the functor $G \mapsto \widehat{\rat G} \hat{\otimes } \widehat {\rat G}$.
\end{proof}

The above structure is described at the level of the associated graded objects as follows (compare Proposition \ref{prop:passi_filt_polynomial}):

\begin{prop}
\label{prop:assoc_graded_malcev}
\ 
\begin{enumerate}
\item 
The associated graded of $\malcev$ is isomorphic (as a Lie algebra in $\f (\gr)$) to $\liealg (\A)$, the functor $\A$ composed with the free Lie algebra functor $\liealg(-)$.
\item 
The associated graded of the  inclusion $\malcev \hookrightarrow \qgr_\bullet P_\zed$ identifies as the canonical inclusion 
$ 
\liealg (\A) \hookrightarrow T (\A)
$ 
of Lie algebras in $\f (\gr)$.
\end{enumerate}
\end{prop}

\begin{proof}
Proposition \ref{prop:properties_fbcr_propoly} implies that the passage to the associated graded is an exact, symmetric monoidal functor from $\propoly$ to $\f (\gr)$. 

Now, the associated graded of the Hopf algebra $\qgr_\bullet P_\zed$ in $\propoly$ is the tensor algebra $T (\A)$ (considered as a primitively-generated Hopf algebra), by Propositions \ref{prop:passi_filt_polynomial} and \ref{prop:P_zed_propoly}. The result thus follows from the definition of $\malcev$ in $\propoly$.
\end{proof}

The universal enveloping algebra construction applies in the symmetric monoidal category $\propoly$; it is defined as the left adjoint to the restriction from unital associative algebras in $\propoly$ to Lie algebras in $\propoly$ given by the commutator bracket. This is analogous to the classical case.

Hence, one can construct $U \malcev$ in $\propoly$, which is  a cocommutative Hopf algebra with primitives $\malcev$.

\begin{prop}
\label{prop:U_malcev}
The inclusion $\malcev \hookrightarrow \qgr_\bullet P_\zed$ induces an isomorphism $
U \malcev \stackrel{\cong}{\rightarrow} \qgr_\bullet P_\zed
$ of cocommutative Hopf algebras in $\propoly$.
\end{prop}

\begin{proof}
The morphism is constructed using the universal property of $U(-)$ from the inclusion $\malcev \hookrightarrow \qgr_\bullet P_\zed$ and this is a morphism of cocommutative Hopf algebras in $\propoly$.

It remains to show that this is an isomorphism. By Proposition \ref{prop:fbcr_conservative}, it is sufficient to check this after passage to the associated graded, when it identifies with the isomorphism
 $
U (\liealg (\A)) \stackrel{\cong}{\rightarrow} T (\A)
 $
induced by the inclusion $\liealg (\A) \rightarrow T(\A)$ given in Proposition \ref{prop:assoc_graded_malcev}. 
\end{proof}

%%%%%%%%%%%%%%%%%%%%%%%%%%%%%%%%%%%%%%%%%%%%%%%%%%%%%%%%%%%%%
\subsection{The splitting}

Recall from Proposition \ref{prop:malcev_first_properties} that $\compl (\malcev)$ is isomorphic to the underlying functor of $G \mapsto \prim \widehat{\rat G}$. 

\begin{lem}
\label{lem:compl_malcev_zed}
There is an isomorphism $\compl (\malcev )(\zed) \cong \rat$, with generator given by $\log (x)$, for $x$ a generator of $\zed$.
\end{lem}

\begin{proof}
Since $\zed$ is a torsion-free nilpotent group, Proposition \ref{prop:torsion-free_nilp_prim} applies: hence $\compl (\malcev )(\zed)$ is generated as a $\rat$-vector space by the elements $\log (g)$, for $g \in \zed$. Since $\log (g)$ is always a scalar multiple of $\log (x)$, the result follows.
\end{proof}

Now, using the adjunction of Proposition \ref{prop:compl_right_adjoint}, the element $\log (x) \in \compl (\malcev) (\zed)$ gives the map 
$
\pi : \qgr_\bullet P_\zed \rightarrow \malcev$.

\begin{lem}
\label{lem:pi_surjective}
The map $\pi : \qgr_\bullet P_\zed \rightarrow \malcev$ is surjective.
\end{lem}

\begin{proof}
We require to show that, for each $d\in \nat$, the corresponding map $\qgr_d P_\zed \rightarrow \malcev_d$ is surjective. 
 By adjunction, this map corresponds to a map $P_\zed \rightarrow \malcev_d$ and hence to an element of $\malcev_d (\zed)$. By construction, this element is the image of $\log (x)$ under the surjection $\compl (\malcev) (\zed) \twoheadrightarrow \malcev_d (\zed)$. Hence we require to prove that, for any $G \in \ob \gr$, $\malcev_d (G)$ is generated as a $\rat$-vector space by the images of the elements $\log (g)$, for $g \in G$. 

By  Quillen's results (reviewed in Section \ref{sect:group-ring}),  $\compl (\malcev) (G)$ is given by the set of elements $\log (\tilde{g})$, for $\tilde{g}$ in $\gplike \widehat{\rat G}$. Hence $\malcev_d (G)$ is certainly generated by the set of images of these elements. 

By Proposition \ref{prop:passi_nilpotent}, for $G$ a finitely-generated free group, the quotient $G \twoheadrightarrow G/ \gamma_{d+1} G$ induces an isomorphism:
\[
\qgr_d P_\zed (G) 
\cong 
\rat (G/\gamma_{d+1} G) / \aug^{d+1}( G/\gamma_{d+1} G). 
\]
Now, Proposition \ref{prop:torsion-free_nilp_prim} implies that, for $\tilde{g}$ in $\gplike \widehat{\rat G}$ the image of  $\log (\tilde{g})$ in $\widehat{\rat (G/\gamma_{d+1} G)}$ is a $\rat$-linear combination of images of elements of the form $\log (g)$ for $g \in G$. Hence, by using the above isomorphism, the result follows.
\end{proof}

Recall that there is a splitting $P_\zed \cong \rat \oplus \pbar$ and an isomorphism $\qgr_1 \pbar \cong \A$. The composite surjection $P_\zed \twoheadrightarrow \A$ corresponds (under the corepresenting property of $P_\zed$) to the generator of $\A (\zed)$ given by the chosen generator of $\zed$.

\begin{nota}
\label{nota:rho}
Write $\rho : \malcev \rightarrow \qgr_\bullet \A$ for the morphism of $\propoly$ given by the composite of the canonical inclusion $\malcev \subset \qgr_\bullet P_\zed$ with the morphism induced by the surjection $P_\zed \twoheadrightarrow \A$.
\end{nota}

On completion, $\rho$ gives $\compl (\rho) : \compl (\malcev) \rightarrow \A$ in $\f (\gr)$.

\begin{lem}
\label{lem:compl_rho_iso}
Evaluating $\compl (\rho) : \compl (\malcev) \rightarrow \A$ on $\zed$ gives an isomorphism 
$
\rat \cong \compl (\malcev)(\zed) \stackrel{\cong}{\rightarrow} \A (\zed) \cong \rat.
$
\end{lem}

\begin{proof}
By Lemma \ref{lem:compl_malcev_zed}, a  generator of $\compl (\malcev)(\zed)$ is given by $\log (x)$ (for the chosen generator $x$ of $\zed$). Under the map $\compl (\rho)$, it is straightforward to check that this maps to the generator of $\A(\zed)$ given by the class of $x$.
\end{proof}

Recall that $\qgr_\bullet P_\zed$ is projective in $\propoly$, by Corollary \ref{cor:proj_gen_propoly}.

\begin{prop}
\label{prop:malcev_s=1}
\ 
\begin{enumerate}
\item 
The surjection $\pi : \qgr_\bullet P_\zed \twoheadrightarrow \malcev$ is a retract of the defining inclusion $\malcev \subset \qgr_\bullet P_\zed$.
\item 
$\malcev$ is projective in $\propoly$; it corepresents the functor 
$ 
G_\bullet \mapsto \cre_1 G_1
$. 
\end{enumerate}
\end{prop}

\begin{proof}
For the first statement, we require to show that the composite $\malcev \hookrightarrow \qgr_\bullet P_\zed \twoheadrightarrow \malcev$ is the identity. Since $\pi : \qgr_\bullet P_\zed \twoheadrightarrow \malcev$  is surjective by Lemma \ref{lem:pi_surjective}, it suffices to show that the composite 
\[
\qgr_\bullet P_\zed \stackrel{\pi}{\twoheadrightarrow} \malcev \hookrightarrow \qgr_\bullet P_\zed \twoheadrightarrow \malcev
\]
coincides with $\pi : \qgr_\bullet P_\zed \twoheadrightarrow \malcev$.

By adjunction, a map $\qgr_\bullet P_\zed \twoheadrightarrow \malcev$ corresponds to a map $P_\zed \rightarrow \compl (\malcev) $ and hence to an element of $\compl (\malcev ) (\zed)$. Now, by Lemma \ref{lem:compl_rho_iso}, $\compl (\rho)$ evaluated on $\zed$ is an isomorphism. Using this, one reduces to showing that the composite
\[
\malcev \hookrightarrow \qgr_\bullet P_\zed \twoheadrightarrow \malcev
\stackrel{\rho}\rightarrow \qgr_\bullet \A
\]
coincides with $\rho$. 

Now, the composite $\qgr_\bullet P_\zed \twoheadrightarrow \malcev
\stackrel{\rho}\rightarrow \qgr_\bullet \A$ can easily be seen to identify with $\qgr_\bullet (P_\zed \rightarrow \A)$, where the morphism $P_\zed \rightarrow \A$ is the surjection used in the construction of $\rho$. The result follows.
\end{proof}

%%%%%%%%%%%%%%%%%%%%%%%%%%%%%%%%%%%%%%%%%%%%%%%%%%%%%%%%%%%%%%%%%%%%%%%%%%%%%%
\subsection{A convenient set of projective generators of $\propoly$}

Proposition \ref{prop:malcev_s=1} allows the construction of a useful set of projective generators for $\propoly$.

\begin{thm}
\label{thm:proj_cover_malcev_s}
For $s \in \nat$, $\malcev ^{\obar s}$ 
\begin{enumerate}
\item 
is projective in $\propoly$; 
\item 
is the projective cover of $\qgr_\bullet (\A^{\otimes s})$; 
\item 
corepresents the functor $G_\bullet \mapsto \cre_s G_s$, where the $\sym_s$-action on $\hom_{\propoly} (\malcev ^{\obar s},G_\bullet )$ is given by place permutations of the factors $\malcev$ with respect to  $\obar$. 
\end{enumerate}
\end{thm}

\begin{proof}
The case $s=0$ is straightforward, so we suppose that $s>0$.

Since $\malcev ^{\obar s}$ is a direct summand of $(\qgr_\bullet P_\zed)^{\obar s}$, to establish projectivity, it suffices to show that the latter is projective. The exponential property of $\pbif$ gives the isomorphism $P_\zed ^{\otimes s} \cong P_{\zed ^{\star s}}$. 
Proposition \ref{prop:qgr_tensor_Pzed} thus gives the isomorphism $\qgr_\bullet P_{\zed ^{\star s}} \cong (\qgr_\bullet P_\zed)^{\obar s}$, hence  the latter is projective in $\propoly$, by Corollary \ref{cor:proj_gen_propoly}.

The surjection $\rho : \malcev \twoheadrightarrow \qgr_\bullet \A$ induces $\malcev ^{\obar s} \twoheadrightarrow (\qgr_\bullet \A)^{\obar s} \cong \qgr_\bullet (\A^{\otimes s})$, using the symmetric monoidal property of $\qgr_\bullet$ for the isomorphism.  Moreover, by construction, the kernel $K$ of this surjection in $\propoly$ has the property $K_t =0$ for $t \leq s$. 

To show that the surjection exhibits $\malcev ^{\obar s}$ as the projective cover of $\qgr_\bullet (\A^{\otimes s})$, we require to show that, if $Y \subset \malcev ^{\obar s}$ such that the composite surjects to $\qgr_\bullet (\A^{\otimes s})$, then $Y = \malcev^{\obar s}$. Suppose otherwise, so that $X:= \malcev^{\obar s} / Y \neq 0$. This has the property $X_t = 0$ for $t \leq s$ and there exists a minimal $\ell >s$ such that $X_\ell \neq 0$. Minimality implies that $\qgr_{\ell -1} X_\ell =0$, so that $X_\ell$ is homogeneous polynomial degree of $\ell$, so  that there is a non-zero map $X_\ell \rightarrow \A^{\otimes \ell}$. This gives a non zero-map $\malcev^{\obar s} \rightarrow  \qgr_\bullet (\A^{\otimes \ell})$, where $\ell > s$. 

We show that this is not possible, hence establishing a contradiction. One checks that it suffices to show that $\hom _{\f(\gr)} ((\malcev_\ell) ^{\otimes s}, \A^{\otimes \ell})$ is zero. Now $\malcev_\ell$ is a quotient of $\pbar$, hence it suffices to show that $\hom _{\f(\gr)} (\pbar ^{\otimes s}, \A^{\otimes \ell})$ is zero; this follows from Lemma \ref{lem:qgr_tensor}.

Finally, we show that that $\malcev ^{\obar s}$ corepresents the functor $G_\bullet \mapsto \cre_s G_s$.  
 As in the proof of Lemma \ref{lem:surj_compl_eval}, write $\overline{G}_\bullet$ for the quotient of $G_\bullet$ with  $\overline{G}_t$ equal to $G_t$ for $t \leq s$ and $G_s$ for $t \geq s$, with the obvious structure morphisms. Consider the associated short exact sequence: 
\[
0 
\rightarrow 
K 
\rightarrow 
G_\bullet 
\rightarrow 
\overline{G}_\bullet
\rightarrow 
0
\] 
in $\propoly$. By construction, $K_t =0$ for $t \leq s$.

 Arguing as above, one shows that $\hom_{\propoly} (\malcev^{\obar s}, \overline{G}_\bullet)$ is isomorphic to $\cre_s G_s$. Hence, to prove the result, it suffices to show that $\hom_{\propoly} (\malcev^{\obar s}, K)=0$. Suppose that such a non-zero  map exists and let $X$ be its image, a non-zero quotient of $\malcev^{\obar s}$ with $X_t =0$ for $t \leq s$. As above, this is not possible, establishing the contradiction.
\end{proof}

\begin{cor}
\label{cor:malcev_s_proj_generators}
The set $\{ \malcev^{\obar s} \ | \ s\in \nat\}$ is a set of projective generators for $\propoly$.
\end{cor}

\begin{proof}
Consider $G_\bullet \in \ob \propoly$. For each $s \in \nat$, using the corepresenting property of $\malcev ^{\obar s}$, one has the map
\[
\phi_s : \cre_s G_s \otimes \malcev ^{\obar s} \rightarrow G_\bullet
\]
that is adjoint to the identity on $\cre_s G_s$ (considered as a vector space). Here $\cre_s G_s \otimes \malcev ^{\obar s}$ is projective in $\propoly$.

Summing over $s$, the maps $\phi_s$ yield the map
$
 (\phi_s) : 
\bigoplus_{s \in \nat}
\big(
\cre_s G_s \otimes  \malcev ^{\obar s} 
\big)
\rightarrow 
G_\bullet
$. 
Again the domain of this map is projective in $\propoly$. Hence it suffices to show that the map is surjective. 

Suppose that $(\phi_s)$ is not surjective and write $C$ for its cokernel, so that there exists a minimal $\ell \in \nat$ such that $C_\ell \neq 0$. Minimality implies that $\cre_\ell C_\ell \neq 0$, so that the surjection $G_\ell \twoheadrightarrow C_\ell$ induces a non-zero surjection $\cre_\ell G_\ell \twoheadrightarrow \cre _\ell C_\ell$. This is in  contradiction with the construction of $\phi_\ell$. 
\end{proof}

\begin{rem}
For $d \in \nat$, $\{ (\malcev^{\obar s})_d \ |\ 0 \leq s \leq d \}$ is a set of projective generators of $\f_d (\gr)$. Heuristically, the set of projective generators $\{ \malcev^{\obar s} \ | \ s\in \nat\}$  is what one obtains `on passage to the limit' as $d \rightarrow \infty$.
\end{rem}

%%%%%%%%%%%%%%%%%%%%%%%%%%%%%%%%%%%%%%%%%%%%%%%%%%%%%%%%%%%%%%%%%%%%%%%%%%%%%%%%%%%%%%%%%%%%%%%%%%%%%%%%%%%%%%ù
\subsection{Decomposing $\qgr_\bullet P_\zed$}

By Corollary \ref{cor:proj_gen_propoly}, $\qgr_\bullet P_\zed$ is projective in $\propoly$. The aim of this subsection is to give a decomposition of this functor using the projective generators $\malcev^{\otimes s}$, for $s \in \nat$. This 
 can be interpreted as a functorial form of (a weak version of) the Poincaré-Birkhoff-Witt theorem.

By construction, there is a monomorphism $\malcev \hookrightarrow \qgr_\bullet P_\zed$ in $\propoly$. 
 Moreover, $\qgr_\bullet P_\zed$ has an associative algebra structure in $\propoly$, so that one can form the morphisms $\mu_s$  (for each $s \in \nat$), where $\mu_s$ is given by the composite
 \[
 \malcev^{\obar s} \hookrightarrow (\qgr_\bullet P_\zed)^{\obar s} \rightarrow \qgr_\bullet P_\zed
 \]
 where the second map is the product. (For $s =0$, this is just the map induced by the canonical inclusion $\rat \hookrightarrow P_\zed$.)
 
As usual, these induce an increasing filtration of $\qgr_\bullet P_\zed$, $\filt_n (\qgr_\bullet P_\zed)$, $n \in \nat$, where $\filt_n (\qgr_\bullet P_\zed)$ is the image of 
\[
\bigoplus_{s=0}^n  \malcev^{\obar s} \rightarrow  \qgr_\bullet P_\zed
\]
given by the morphisms $\mu_0, \ldots, \mu_n$. (By convention one takes $\filt_{-1} (\qgr_\bullet P_\zed) =0$.)

\begin{lem}
\label{lem:filt_qgr_bullet_PZ}
The filtration $\filt_n (\qgr_\bullet P_\zed)$ is exhaustive, i.e., $\qgr_\bullet P_\zed \cong \lim_\rightarrow \filt_n (\qgr_\bullet P_\zed)$. Moreover, for each $n \in \nat$, there is an isomorphism in $\propoly$:
\[
\filt_n (\qgr_\bullet P_\zed)/\filt_{n-1} (\qgr_\bullet P_\zed)
\cong 
\triv_n \otimes_{\sym_n} \malcev^{\obar n},
\]
where $\triv_n$ is the trivial representation of $\sym_n$.
\end{lem}

\begin{proof}
Proposition \ref{prop:U_malcev} gives that the inclusion $\malcev \hookrightarrow \qgr_\bullet P_\zed$ induces an isomorphism $
U \malcev \stackrel{\cong}{\rightarrow} \qgr_\bullet P_\zed
$ of Hopf algebras in $\propoly$.

The filtration $\filt_n (\qgr_\bullet P_\zed)$ corresponds under this isomorphism to the filtration of $U \malcev$  induced by the length filtration of the tensor algebra on $\malcev$. Using this, the result follows by using the classical Poincaré-Birkhoff-Witt theorem.
\end{proof}

\begin{rem}
\ 
\begin{enumerate}
\item 
The functor $\triv_n \otimes_{\sym_n} \malcev^{\obar n}$ can be considered as the $n$th symmetric power functor $S^n$  applied to $\malcev$ in $(\propoly, \obar, \rat)$, written $S^n (\malcev)$.
\item 
$S^n (\malcev)$  is a direct summand of $\malcev^{\obar n}$ (since we are working over $\rat$), hence is projective. More precisely, it is the projective cover of $\qgr_\bullet S^n (\A)$ ($S^n (\A)$ identifies with $\triv_n \otimes_{\sym_n} \A^{\otimes n}$). In particular, it is indecomposable.
\item 
$S^n (\malcev)$ has the property that $(S^n (\malcev))_\ell =0$ for $\ell <n$ and it is isomorphic to $S^n (\A)$ for $\ell=n$.
\end{enumerate}
\end{rem}

The canonical surjection $T^n \twoheadrightarrow S^n$  from the $n$th tensor power functor to the $n$th tensor power functor admits a unique section; this is given by $\prod_{i=1}^n x_i \mapsto \frac{1}{n!} \sum_{\sigma \in \sym_n}
 x_{\sigma(1)} \otimes \ldots \otimes x_{\sigma(n)}$. Hence, the projection $\malcev^{\obar n} 
 \twoheadrightarrow S^n (\malcev)$ has an induced section $ S^n (\malcev) \hookrightarrow \malcev^{\obar n}$ in $\propoly$.
 
 \begin{nota}
 For $n \in \nat$, write $\tilde{\mu}_n : S^n (\malcev) \rightarrow \qgr_\bullet P_\zed$ for the composite 
 \[
 S^n (\malcev) \hookrightarrow  \malcev^{\obar n} \stackrel{\mu_n}{\rightarrow}  \qgr_\bullet P_\zed.
 \]
 \end{nota}
 
 \begin{thm}
 \label{thm:pbw_decomposition}
 The morphisms $\tilde{\mu}_n$, $n \in \nat$, yield an isomorphism in $\propoly$
 \[
 \bigoplus_{n \in \nat} S^n (\malcev) 
 \stackrel{\cong}{\rightarrow} 
 \qgr_\bullet P_\zed.
 \]
This expresses $\qgr_\bullet P_\zed$ as a direct sum of indecomposable projectives in $\propoly$.
 \end{thm}
 
 \begin{proof}
 To establish the isomorphism, 
it suffices to prove that, for each $N \in \nat$, the map induced by $\tilde{\mu}_n$, for $n \in \{ 0, \ldots , N\}$
\[
\bigoplus_{n=0}^N S^n (\malcev) 
{\rightarrow} 
 \filt_N \qgr_\bullet P_\zed
\]
is an isomorphism. This is proved by induction upon $N$, the cases $N \in \{ 0, 1\}$ being clear. 

For the inductive step, by the five-lemma, it suffices to show that the composite map 
\[
 S^N (\malcev)
 \stackrel{\tilde{\mu}_N}{\rightarrow} 
 \filt_N \qgr_\bullet P_\zed
 \twoheadrightarrow 
 \filt_N \qgr_\bullet P_\zed/\filt_{N-1} \qgr_\bullet P_\zed
\]
is an isomorphism.

Since $S^N (\malcev)$ is the projective cover of $S^N (\A)$, it suffices to check behaviour after composing with the projection $S^N (\malcev) \twoheadrightarrow S^N (\A)$. 

By construction of $\tilde{\mu}_N$, one sees that the following diagram commutes:
\[
\xymatrix{
 S^N (\malcev)
 \ar[r]^{\tilde{\mu}_N}
 \ar@{->>}[d]
 &
  \filt_N \qgr_\bullet P_\zed
  \ar@{->>}[r]
  &
  \filt_N \qgr_\bullet P_\zed/\filt_{N-1} \qgr_\bullet P_\zed
  \ar[r]^(.6)\cong 
  &
  S^N (\malcev)
\ar@{->>}[d]
\\
S^N(\A) 
\ar[rrr]_=
&&&
S^N(\A),
}
\]
in which the  isomorphism in the top right is given by Lemma \ref{lem:filt_qgr_bullet_PZ}. This establishes the inductive step. 

Since  $S^n (\malcev)$ is an indecomposable projective in $\propoly$, the isomorphism exhibits a decomposition of $\qgr_\bullet P_\zed$ as a direct sum of indecomposable projectives. 
\end{proof}

%%%%%%%%%%%%%%%%%%%%%%%%%%%%%%%%%%%%%%%%%%%%%%%%%%%%%%%%%%%%%%%%%%%%%%%%%%%%%%%%%
\subsection{The dual Mal'cev functors}

Recall from Proposition \ref{prop:D_propoly_analytic_adjoint}
 that duality induces an adjunction: 
\[
D : \propoly \op \rightleftarrows \f_\omega (\gr\op) : D.
\]

Hence, one has the dual Mal'cev functor $D \malcev \in \ob \f_\omega (\gr\op)$, i.e. an analytic functor on $\gr\op$. 
Explicitly, this is constructed as the colimit
$
\lim_{\substack{\rightarrow \\ d}}
D \malcev_d$,  
where $D \malcev_d$ is a {\em finite} polynomial functor of degree $d$.

Likewise, for $s \in \nat$, one has the dual functor $D (\malcev^{\obar s})$.

\begin{lem}
\label{lem:dual_malcev_s}
For $s \in \nat$, there is a $\sym_s$-equivariant isomorphism:
\[
D (\malcev^{\obar s})
\cong 
(D \malcev)^{\otimes s}.
\]
\end{lem}

\begin{proof}
This is a straightforward verification using the definition of $\obar$ on $\propoly$.
\end{proof}

\begin{thm}
\label{thm:Dmalcev_injectives_analytic}
For $s \in \nat$, $D (\malcev^{\obar s})$ is injective in $\f_\omega (\gr\op)$. It is the injective envelope of 
the functor $(\A^\sharp)^{\otimes s}$.
\end{thm}

\begin{proof}
Since $\malcev^{\obar s}$ is projective in $\propoly$, the injectivity follows from the duality adjunction of Proposition \ref{prop:D_propoly_analytic_adjoint} together with the fact that $D : \f_\omega (\gr)\rightarrow \propoly$ is exact, by Lemma \ref{lem:propoly_analytic_duality}.

Moreover, the surjection $\malcev^{\obar s} \twoheadrightarrow \qgr_\bullet (\A^{\otimes s})$ induces the monomorphism
\[
(\A^\sharp)^{\otimes s} \hookrightarrow D(\malcev^{\obar s}).
\]
It remains to show that this exhibits $D (\malcev^{\obar s})$ as the injective envelope of $(\A^\sharp)^{\otimes s} $. Once again, this follows from the duality adjunction together with the fact that $\malcev^{\obar s}$ is the projective cover of $\qgr_\bullet (\A^{\otimes s})$. Details are left to the reader. 
\end{proof}

From this one deduces:

\begin{cor}
\label{cor:analytic_inj_cogen}
The set $\{ D (\malcev^{\obar s}) \ | \ s \in \nat \}$ is a set of injective cogenerators of $\f_\omega (\gr\op)$.
\end{cor}

%%%%%%%%%%%%%%%%%%%%%%%%%%%%%%%%%%%%%%%%%%%%%%%%%%%%%%%%%%%%%%%%%%%
\subsection{Understanding the structure of $\malcev$}

As in Section \ref{subsect:describe_qgr_d_PZ}, one can describe  the structure of the functors $\malcev_d$, for $d \in \nat$, and hence of $\malcev$. This uses a similar approach to that in Proposition \ref{prop:describe_qgr_Pzed}, recalling that there is a canonical inclusion
$
\malcev_d \subset \qgr_d P_\zed$, 
by construction. 

Here, rather than working with the truncated tensor algebra $T^{\leq d}(-)$, one works with the truncated Lie algebra $\liealg_{\leq d}( -)$; this embeds naturally $\liealg_{\leq d} (-) \hookrightarrow T^{\leq d} (-)$ as a functor to Lie algebras.

\begin{nota}
Write $\mathsf{bch}_d(x_1, x_2)$ for the truncated Baker-Campbell-Hausdorff element of $\liealg_{\leq d} (\rat \{ x_1, x_2 \})$ corresponding to the truncation of $\log (\exp (x_1)\exp(x_2))$.  
\end{nota}

In the following statement, $p$, $\chi$ and $\nabla$ are as in Section \ref{subsect:describe_qgr_d_PZ}. 

\begin{prop}
\label{prop:determine_malcev_d}
For $1<d \in \nat$, $\malcev_d$ is isomorphic to the functor with values in Lie algebras 
\[
\mathrm{Free}(S) \mapsto \liealg_{\leq d} (\rat S),
\]
where $\rat S$ is the $\rat$-vector space generated by the finite set $S$, such that:
\begin{enumerate}
\item
morphisms in the image of $\mathrm{free} : \finset \rightarrow \gr$ act via $\rat( -) \ : \ \finset \rightarrow \mathrm{mod}_\rat$ and the naturality of $\liealg_{\leq d} (-)$; 
\item 
$p : \zed = \mathrm{Free} (x) \rightarrow \{ e \} = \mathrm{Free} (\emptyset)$ sends $x \in  \liealg_{\leq d}(\rat\{x\})$ to zero;
\item 
$\chi : \zed = \mathrm{Free} (x) \rightarrow \zed = \mathrm{Free}(x)$ sends the generator $x \in \liealg_{\leq d} (\rat\{x\})$ to 
$-x$; 
\item 
$\nabla : \zed = \mathrm{Free} (x) \rightarrow \zed \star \zed = \mathrm{Free} (x_1, x_2)$ acts via $x \mapsto \mathsf{bch}_d (x_1, x_2)\in \liealg_{\leq d} (\rat \{ x_1, x_2\})$. 
\end{enumerate}
\end{prop}

\begin{proof}
Using the universal property of $\liealg_{\leq d} (-)$, there is a natural transformation of functors from $\finset$ to Lie algebras defined on a finite set $S$
\[
\liealg_{\leq d} (\rat S) \rightarrow \malcev_d (S) 
\]
by sending a generator $y \in S$ to the image of $\log(y)$ in $\malcev_d (S)$. This is easily seen to be an isomorphism.  

As in the proof of Proposition \ref{prop:describe_qgr_Pzed}, it remains to check that this is a natural transformation of functors on $\gr$, using the specified structure on the domain. 

The behaviour of $\chi$ follows from the identity (for any $g \in G$), $\log (g^{-1}) = - \log (g)$. The behaviour of $\nabla$ is a consequence of the {\em definition} of the Baker-Campbell-Hausdorff series. Here it is essential that the natural transformation above has been constructed as a natural transformation of Lie algebras. 
\end{proof}

%\newpage
%\input{opd}
\section{Operadic structures}
\label{sect:opd}

This section reviews the relationship between operads in $\kring$-modules and $\kring$-linear categories. Here, $\kring$ can  usually be taken to be an arbitrary commutative, unital ring; for all applications it will be $\rat$ and the reader may prefer to assume this throughout. 

The most important example here is the operad $\opd = \lie$, which encodes Lie algebras in $\kmod$. Again, the reader may prefer to focus upon this case.

This material is well-known, and is mostly treated  in  \cite{MR1854112}, for example. (Most of this background material is also covered in \cite{2021arXiv211001934P}; it is included here so as to make the exposition self-contained.)

%%%%%%%%%%%%%%%%%%%%%%%%%%%%%%%%%%%%%%%%%%%%%%%%%%%%%%%%%
\subsection{The category associated to an operad}

The category $\f (\fb\op)$ (aka. $\kring \fb\op$-modules) underlies the category of $\kring$-linear operads (i.e., operads in the symmetric monoidal category $\kmod$): an operad $\opd$ has underlying $\kring\fb\op$-module  given by the family (for $n\in \nat$) of $\sym_n\op$-modules $\opd (n)$ (or $\opd (\mathbf{n})$ in the notation for $\kring\fb\op$-modules), the term of {\em arity} $n$.

The  identity operad $I$ has $I(n)$ that is zero, except for $n=1$, when $I(1) = \kring$. Each operad $I$ is equipped with the canonical unit morphism $I \rightarrow \opd$.

\begin{nota}
\label{nota:cat_opd}
For $\opd$ an operad, the associated $\kring$-linear category is denoted by $\cat \opd$, so that $\opd (n)= \cat \opd (n ,1)$. This has set of objects $\nat$. The morphisms are given explicitly by:
\[
\cat \opd (s,t) = \bigoplus_{f : \mathbf{s} \rightarrow \mathbf{t}} \bigotimes_{i=1}^t \opd (f^{-1} (i)),
\]
where the sum is over set maps $f : \mathbf{s} \rightarrow \mathbf{t}$.  Composition is defined using the composition of the operad.  
\end{nota}

\begin{rem}
There is more structure: $\cat \opd$ is naturally a PROP, i.e., it is symmetric monoidal with the structure corresponding to $(\nat , +)$ on objects.
\end{rem}

The construction $\opd \mapsto \cat \opd$ is functorial: a morphism of operads $\phi : \opd \rightarrow \ppd$ induces a $\kring$-linear functor $\cat \phi : \cat \opd \rightarrow \cat \ppd$.

\begin{exam}
\label{exam:catI}
For $I$ the identity operad,  $\cat I$ is equivalent to  $\kring \fb$, the $\kring$-linearization of $\fb$. Explicitly, $\cat I (m,n)$ is zero unless $m=n$, when $\cat I (n,n) \cong \kring \sym_n$ as associative $\kring$-algebras.

For any operad $\opd$, the unit $I \rightarrow \opd$ induces $\cat I \rightarrow \cat \opd$. In particular, for $n\in \nat$, this induces a morphism of $\kring$-algebras  
$
\kring \sym_n \rightarrow \cat \opd (n,n).
$ 
\end{exam}

In the cases of interest here, the following applies:

\begin{lem}
\label{lem:cat_opd_connexe}
Suppose that $\opd (0)=0$ and $\opd (1)=\kring$, corresponding to the identity of the operad.  Then, for $m, n \in \nat$:
\begin{enumerate}
\item 
$\cat \opd (n,n) \cong \kring \sym_n$ as associative algebras; 
\item 
$\cat \opd (m,n) =0$ if $n>m$.
\end{enumerate}

In particular, there is a natural $\nat$-grading of the morphism modules of $\cat \opd$, where $\cat \opd (m,n)$ is placed in degree $m-n$; this is compatible with composition in the obvious sense.
\end{lem}

For any operad $\opd$, one has the categories of left (respectively right) modules over $\cat \opd$:
\begin{enumerate}
\item 
a left $\cat \opd$-module is a $\kring$-linear functor $\cat \opd \rightarrow \kmod$; 
\item 
a right $\cat \opd$-module is a $\kring$-linear functor $(\cat \opd)\op \rightarrow \kmod$; 
\end{enumerate}
morphisms are $\kring$-linear natural transformations.

\begin{nota}
\label{nota:catopd_modules}
For an operad $\opd$,
\begin{enumerate}
\item 
 the category of right (respectively left) $\cat \opd$-modules is denoted $\rmod[\cat\opd]$ (resp. $\lmod [\cat\opd]$).
\item 
$\rmod[\cat\opd]^{< \infty} \subset \rmod[\cat \opd]$ (respectively $\lmod[\cat\opd]^{< \infty} \subset \lmod[\cat \opd]$) denote the full subcategory of modules such that the underlying $\kring\fb\op$-module (resp. $\kring\fb$-module) has finite support ($M$ has finite support if $M(n)=0$ for $n\gg 0$).
\end{enumerate}
\end{nota}

\begin{rem}
\label{rem:alt_right_mod}
Using restriction along $\cat I \rightarrow \cat \opd$ (see Example \ref{exam:catI}), a right $\cat \opd$-module has an underlying right $\kring \fb$-module structure. This allows the following alternative formulation of the definition of a right $\cat \opd$-module 
(and analogously for left $\cat \opd$-modules).

A right $\cat \opd$-module is a sequence $M(n)$ of right $\kring\sym_n$-modules, for $n \in \nat$ (i.e., a $\kring\fb \op$-module), equipped with structure morphisms:
\begin{eqnarray}
\label{eqn:M_right_cat_opd}
M (n) \otimes_{\sym_n} \cat \opd (m,n) \rightarrow M(m)
\end{eqnarray}
that are unital and associative in the obvious sense; in particular the morphism  is $\sym_m\op$-equivariant. Natural transformations are identified similarly. 
\end{rem}

Clearly one has:

\begin{prop}
\label{prop:abelian_lrmod}
The categories $\lmod[\cat\opd]$ and $\rmod [\cat \opd]$ are abelian $\kring$-linear categories.
\end{prop}

\begin{exam}
\label{exam:atomic}
Using the identification of Example \ref{exam:catI}, the category of left $\cat I$-modules is  equivalent to $\f (\fb)$ and that of right $\cat I$-modules is equivalent to $\f(\fb\op) \cong \f( \fb)$.

Suppose that  $\opd (0)=0$ and $\opd (1)=\kring$, then $\opd$ is canonically augmented: there is a morphism of operads $\opd \rightarrow I$ that is an isomorphism in arity $1$. This induces the `augmentation' of $\kring$-linear categories $\cat \opd \rightarrow \cat I$. Hence any $\kring\fb$-module has a natural left $\cat \opd$-module structure (by restriction) and, likewise, any $\kring\fb\op$-module a natural right $\cat \opd$-module structure. 

In particular, for $n\in \nat$, a $\kring \sym_n$-module can be considered as an object of $\f (\fb)\cong \f(\fb\op)$ supported on $\mathbf{n}$ and hence as a left (respectively right) $\cat \opd$-module. 
\end{exam}

\begin{exam}
\label{exam:opd_right_cat_opd}
The $\kring \fb\op$-module underlying an operad $\opd$ has a canonical  {\em right} $\cat \opd$-module, with structure morphism 
$
\opd (n) \otimes_{\sym_n} \cat \opd (m, n)
\rightarrow 
\opd (m) 
$
given by the composition operation of the operad $\opd$.

This structure is natural in the following sense:  for $\phi : \opd \rightarrow \ppd$ a morphism of operads, $\opd$ is a right $\cat \opd$-module and $\ppd$ is a right $\cat \ppd$-module, hence a right $\cat \opd$-module by restriction along  $\cat \phi : \cat \opd \rightarrow \cat \ppd$. 
 With respect to these structures, the underlying morphism of $\phi : \opd \rightarrow \ppd$ is a morphism of right $\cat \opd$-modules. 
\end{exam}

%%%%%%%%%%%%%%%%%%%%%%%%%%%%%%%%%%%%%%%%%%%%%%%%%%%%%%%%%%%%%%%%%%%%%%%%%%%%%%%
\subsection{Canonical filtrations}
\label{subsect:canon_filt_catopd_mod}

Throughout this section, we suppose that $\opd (0)=0$ and $\opd (1) = \kring$, so that Lemma \ref{lem:cat_opd_connexe} applies. The following is then clear:

\begin{lem}
\label{lem:truncate_modules}
Let $d \in \nat$. 
\begin{enumerate}
\item 
There is an exact functor $(-)_{\leq d} : \lmod[\cat \opd] \rightarrow \lmod [\cat\opd]$ equipped with a canonical inclusion $(-)_{\leq d} \hookrightarrow \id$, where for a left $\cat \opd$-module $M$, $M_{\leq d}(n)=M(n)$ for $n \leq d$ and $0$ if $n>d$. 
\item 
There is an exact functor $(-)^{\leq d} : \rmod[\cat \opd] \rightarrow \rmod [\cat\opd]$ equipped with a canonical surjection  $\id \twoheadrightarrow (-)^{\leq d}$, where for a right $\cat \opd$-module $N$, $N^{\leq d}(t)=N(t)$ for $t \leq d$ and $0$ if $t>d$. 
\end{enumerate}
\end{lem}

The functor $(-)_{\leq d} : \lmod[\cat \opd] \rightarrow \lmod [\cat\opd]$ can be constructed as the right adjoint to the inclusion of the full subcategory of left $\cat \opd$-modules $M$ for which $M (t) =0$ for $t>d$. In particular, for 
 any left $\cat \opd$-module $M$, there is a canonical sequence of inclusions of subobjects of $M$ in $\lmod [\cat \opd]$:
\[
M_{\leq 0} \subset M_{\leq 1} \subset \ldots \subset M_{\leq d} \subset M_{\leq d+1} \subset \ldots  \subset M.
\]
Likewise, for a right $\cat \opd$-module $N$, there is a canonical tower under $N$ in $\rmod [\cat \opd]$:
\[
N \twoheadrightarrow \ldots \twoheadrightarrow N^{\leq d+1} \twoheadrightarrow N^{\leq d} \twoheadrightarrow \ldots 
\twoheadrightarrow N^{\leq 1} \twoheadrightarrow N^{\leq 0}.
\]

Moreover, the following is clear:

\begin{prop}
\label{prop:(co)lim_canon_filt}
\ 
\begin{enumerate}
\item 
For $M \in \ob \lmod[\cat \opd]$, the canonical morphism $\lim_{\substack{\rightarrow\\d}} M_{\leq d} \rightarrow M$ is an isomorphism. 
\item 
For $N \in \ob \rmod[\cat\opd]$, the canonical morphism $N \rightarrow \lim_{\substack{\leftarrow \\ d}} N^{\leq d}$ is an isomorphism.
\end{enumerate}
\end{prop}

As well as the functor $(-)_{\leq d} : \lmod[\cat \opd] \rightarrow \lmod [\cat\opd]$, we shall also require the following:

\begin{nota}
\label{nota:left_adjoint_left_cat_opd}
For $d \in \nat$, write $(-)_{[\leq d]} : \lmod[\cat \opd] \rightarrow \lmod [\cat\opd]$ for the functor induced by the left adjoint to the inclusion of the full subcategory of left $\cat \opd$-modules $M$ for which $M (t) =0$ for $t>d$.
\end{nota}

\begin{rem}
\label{rem:left_adjoint_left_cat_opd}
Let $d \in \nat$. 
\begin{enumerate}
\item 
For $N \in \ob \lmod [\cat \opd]$, the canonical surjection $N \twoheadrightarrow N_{[\leq d]}$ is the universal morphism to an object in the full subcategory of left $\cat \opd$-modules $M$ for which $M (t) =0$ for $t>d$.
\item 
Explicitly, $N_{[\leq d]}(t)$ is  the cokernel of
\[
\bigoplus_{s> d} \bigoplus_{f \in \cat \opd (s,t)} N (s) \rightarrow N(t),
\]
where the morphisms are given by the $N(f)$. 
\item 
The composite morphism $N_{\leq d} \hookrightarrow N \twoheadrightarrow N_{[\leq d]}$ is a surjection but is not in general an isomorphism. For instance, taking $\opd = \lie$, $\cat \lie (s, -)_{\leq d}$ is non-zero if $d>0$, whereas $\cat \lie (s, -)_{[\leq d]}$ is zero whenever $s>d$, as follows from the Yoneda lemma.
\end{enumerate}
\end{rem}

%%%%%%%%%%%%%%%%%%%%%%%%%%%%%%%%%%%%%%%%%%%%%%%%%%%%%%%%%%
\subsection{The tensor product over $\cat \opd$}

One has the tensor product 
\[
-\otimes_{\cat \opd} - \ : \ \rmod[\cat\opd] \times \lmod [\cat \opd] \rightarrow \kmod.
\] 
For $M$ a left $\cat \opd$-module and $N$ a right $\cat \opd$-module, $N \otimes_{\cat \opd} M$ is defined as the coequalizer of the diagram
\[
\bigoplus_{s,t}
\bigoplus_{f \in \cat \opd (s,t)} 
N(t) \otimes M(s) 
\rightrightarrows 
\bigoplus_u N(u) \otimes M(u),
\]
where the maps are given by the evident structure morphisms $N(s) \otimes M(s) \leftarrow N(t) \otimes M(s) \rightarrow N(t)\otimes M(t)$ induced by a morphism $f \in \cat \opd (s,t)$.

We note the following basic properties:

\begin{prop}
\label{prop:properties_otimes_cat_opd}
For $M$ a left $\cat \opd$-module and $N$ a right $\cat \opd$-module:
\begin{enumerate}
\item 
The filtration $(M_{\leq d})$ gives rise to a canonical isomorphism
 $$
 \lim_{\substack{\rightarrow \\ d}} \big( N \otimes_{\cat \opd} M_{\leq d}\big) \stackrel{\cong}{\rightarrow } N \otimes_{\cat \opd} M.
 $$
 \item 
 For $d\in \nat$, the canonical surjections  $N \twoheadrightarrow N^{\leq d}$  and $M \twoheadrightarrow M_{[\leq d]}$ induce isomorphisms:
 \[
 N^{\leq d} \otimes_{\cat\opd} M
\cong 
 N^{\leq d} \otimes_{\cat\opd} M_{[\leq d]}
 \cong 
 N \otimes _{\cat \opd} M_{[\leq d]}.  
 \]
\end{enumerate}
\end{prop}

\begin{proof}
The first statement follows from the fact that $N \otimes_{\cat \opd} -$ commutes with colimits. The second statement
is the analogue of Proposition \ref{prop:otimes_gr_poly}; it can be proved directly from the definition of $\otimes _{\cat \opd}$.
 \end{proof}

%%%%%%%%%%%%%%%%%%%%%%%%%%%%%%%%%%%%%%%%%%%%%%%%%%%%%%%%%%%%%%

\subsection{The convolution product for right $\cat \opd$-modules}
\label{subsect:convolution_cat_opd}

Recall that Day convolution defines a symmetric monoidal structure $(\f (\fb\op), \odot, \kring)$,  where $\kring$ is concentrated in arity $0$. Here, for two $\kring\fb\op$-modules, $M \odot N$ is defined on objects by:
\[
(M \odot N) (S) = \bigoplus_{S_1 \amalg S_2 = S} M(S_1) \otimes N(S_2),
\]
where $S$ is a finite set and the sum is indexed by ordered decompositions of $S$ into two subsets.

This extends to $\rmod[\cat\opd]$: by \cite[Proposition 1.6.3]{MR1854112} (which references \cite{MR1617616}), $\odot$ makes the category of right $\cat \opd$-modules into a symmetric monoidal category $(\rmod [\cat\opd],\odot, \kring)$. 
(See also \cite[Section 6.1]{MR2494775}, where the structure is defined for the category of right modules over the operad (defined using $\circ$), which is equivalent to the category of right $\cat\opd$-modules.)

\begin{rem}
\label{rem:conv_left_catopd}
Analogously, the convolution product on $\f (\fb)$ extends to a convolution product on {\em left} $\cat\opd$-modules (see  \cite{2021arXiv211001934P}).
\end{rem}

%\newpage
%\input{endo_malcev}
\section{The $\cat \lie$-module $\underline{\malcev}$ and modelling $\propoly$}
\label{sect:endo_malcev}

The purpose of this section is to prove that $\propoly$ is equivalent to the category $\rmod[\cat\lie]$ (see Theorem \ref{thm:equiv_propoly_modcatlie}), working over $\kring = \rat$. This follows from  the fact that the projective generators $\malcev^{\obar s}$, for $s \in \nat$, assemble to form a left $\cat \lie$-module, and these structure morphisms provide all possible morphisms.

%%%%%%%%%%%%%%%%%%%%%%%%%%%%%%%%%%%%%%%%%%%%%%%%%%%%%%%%%%%%%%%%%%
\subsection{The equivalence of categories}

By construction, $\malcev$ is a Lie algebra in  $(\propoly, \obar, \rat)$.

\begin{nota}
\label{nota:underline_malcev}
Denote by  $\underline{\malcev}$ the left $\cat \lie$-module $\underline{\malcev}$ given by 
$
\underline{\malcev} : s \mapsto \malcev ^{\obar s}$,
with $\cat \lie$ acting via the Lie algebra structure. 
\end{nota}

In particular, for $s, t \in \nat$, this gives the $\sym_s\op \times \sym_t$-equivariant morphism encoding this structure:
\begin{eqnarray}
\label{eqn:cat_lie_malcev}
\cat \lie (s, t) 
\rightarrow 
\hom_\propoly (\malcev^{\obar s}, \malcev^{\obar t}).
\end{eqnarray}
Here, $\cat \lie (s,t) =0$ if $s<t$, so we may assume without loss of generality that $s\geq t$.

\begin{prop}
\label{prop:malcev_full_subcat}
The full subcategory of $\propoly$ with objects $\malcev ^{\obar s}$, for $s\in \nat$, is equivalent to the $\rat$-linear category $\cat \lie$. 
\end{prop}

\begin{proof}
We require to prove that, for all $s, t \in \nat$, the structure morphism (\ref{eqn:cat_lie_malcev}) is an isomorphism. We first establish that (\ref{eqn:cat_lie_malcev}) is injective. 
 
The associated graded of $\malcev$ is $\liealg (\A)$, by Proposition \ref{prop:assoc_graded_malcev}, as a Lie algebra. 
Since passage to the associated graded is symmetric monoidal (cf. Proposition \ref{prop:properties_fbcr_propoly}), the associated graded of $\malcev^{\obar s}$ is $\liealg (\A) ^{\otimes s}$.  Hence, passage to the associated graded gives the  $\sym_s\op \times \sym_t$-equivariant morphism
\[
\hom_\propoly (\malcev^{\obar s}, \malcev^{\obar t})
\rightarrow 
\hom_{\f(\gr)}(\liealg (\A)^{\otimes s}, \liealg (\A) ^{\otimes t}).
\]

Moreover, the composite $\cat \lie \rightarrow \hom_{\f(\gr)}(\liealg (\A)^{\otimes s}, \liealg (\A) ^{\otimes t})$ is the structure morphism for the left $\cat \lie$ module $\underline{\liealg (\A)}$ associated to the Lie algebra  $\liealg (\A)$ in $\f (\gr)$.  
 Since $\liealg (-)$ is the free Lie algebra functor, it is essentially tautological that this composite is injective. This gives the required injectivity of (\ref{eqn:cat_lie_malcev}).

It remains to establish surjectivity, refining the above argument. By Theorem \ref{thm:proj_cover_malcev_s}, there is a natural isomorphism:
\[
\hom_\propoly (\malcev^{\obar s}, \malcev^{\obar t})
\cong 
\cre_s (\malcev^{\obar t})_s.
\]
The right hand side is the $\rat \sym_s$-module that encodes the homogeneous degree $s$ part of $\malcev^{\obar t}$. 
 Equivalently, it encodes the homogeneous degree $s$ part of the associated graded of $\malcev^{\obar t}$, which identifies as 
 $\liealg (\A)^{\otimes t}$, as above. 

We claim that  $\cre_s (\liealg (\A)^{\otimes t})$ is isomorphic to $\cat \lie (s,t)$ as a vector space. To see this, consider the Schur functor $V \mapsto \bigoplus_s \cat \lie (s, t) \otimes_{\sym_s} V^{\otimes s}$. By the construction of $\cat \lie$, this is isomorphic to $\liealg (V) ^{\otimes t}$ and $\cat \lie (s,t)$ is recovered from the component of polynomial degree $s$ in $V$ by applying the appropriate cross-effect functor. This calculation is equivalent to that of $\cre_s (\liealg (\A)^{\otimes t})$.

Now $\cat \lie (s,t)$ has finite dimension; the morphism considered is a linear monomorphism between vector spaces of the same (finite) dimension, hence is an isomorphism, as required.
\end{proof}

This has the immediate consequence:

\begin{thm}
\label{thm:equiv_propoly_modcatlie}
The functor $\hom_{\propoly} (\underline{\malcev}, -) $ induces an equivalence of categories
\[
\hom_{\propoly} (\underline{\malcev}, -) \ : \ 
\propoly \stackrel{\cong}{\rightarrow} \rmod[\catlie].
\]

The inverse equivalence is given by:
\[
- \otimes _{\cat \lie} \underline{\malcev} \ : \ \rmod[\catlie] \rightarrow \propoly.
\]
\end{thm}

\begin{proof}
By Theorem \ref{thm:proj_cover_malcev_s}, 
for $s \in \nat$, $\malcev^{\obar s}$ corepresents the functor $G_\bullet \mapsto \cre_s G_s$. This functor  commutes with colimits, hence $\malcev^{\otimes s}$ is a small projective in $\propoly$. Thus $\{ \malcev^{\otimes s} \ | \ s\in \nat\}$ is a generating set of small projectives. 

The result thus follows from Freyd's theorem (see \cite[Theorem 3.1]{MR294454} for a version over $\zed$).  
\end{proof}

\begin{rem}
\label{rem:identify_w_fbcr}
Recall that restriction along $\cat I \rightarrow \cat \lie$ induces the forgetful functor $\rmod[\cat \lie] \rightarrow \f (\fb\op)$ and that $\f (\fb\op)$ is isomorphic to $\f (\fb)$. Hence this gives the functor $\rmod[\cat \lie] \rightarrow \f (\fb)$.

Composing  $\hom_{\propoly} (\underline{\malcev}, -)$ with the restriction functor gives 
$
\propoly \rightarrow \f (\fb)
$. 
Using Theorem \ref{thm:proj_cover_malcev_s}, this identifies with the functor $\fbcr$ (cf. Notation \ref{nota:fbcr_propoly}) that encodes the associated graded of an object of $\propoly$. 

Thus Theorem \ref{thm:equiv_propoly_modcatlie} refines  $\fbcr$, retaining the necessary additional structure, namely the right action of $\cat \lie$.
\end{rem}

The equivalence of Theorem \ref{thm:equiv_propoly_modcatlie} restricts to the full subcategory of $\propoly$ corresponding to $\f_{<\infty}(\gr) $ (see Proposition \ref{prop:poly_propoly}) to give the following result, which strengthens the corresponding result of \cite{2021arXiv211001934P} (which restricted to functors in $\f_{< \infty}(\gr)$ with a finite composition series). 

\begin{cor}
\label{cor:finite_functors_equiv_rmod}
The category $\f_{< \infty} (\gr)$ is equivalent to the full subcategory $\rmod[\cat \lie]^{< \infty} \subset \rmod[\cat \lie]$ of right $\cat \lie$-modules $M$ such that $M(n) =0$ for $n \gg 0$. 
\end{cor}

\begin{proof}
It suffices to check that the two functors given in Theorem \ref{thm:equiv_propoly_modcatlie} restrict as required. 

A functor  $F\in \ob \propoly$ is in $\f_{< \infty} (\gr)\subset \propoly$ if and only if $\cre_s F_s =0$ for $s \gg 0$. This is equivalent to $\hom_{\propoly} (\malcev^{\obar s}, F) =0$ for $s \gg 0$. The result follows.
\end{proof}

This can be made more precise by comparing the canonical filtration of right $\cat \lie$-modules with the canonical filtration of objects of $\propoly$. Recall that, for $d \in \nat$ and $F_\bullet$ an object of $\propoly$, $F_d$ denotes the $d$th term of $F_\bullet$, which has polynomial degree $d$. There is a canonical surjection in $\propoly$ 
\[
F_\bullet \twoheadrightarrow \qgr_\bullet (F_d)
\]
to an object arising from a functor of $\f_d (\gr)$. 

\begin{prop}
\label{prop:compare_filtrations}
For $d \in \nat$ and $N \in \ob \rmod[\cat \lie]$, the canonical surjection $N \twoheadrightarrow N^{\leq d}$
induces an isomorphism:
\[
\qgr_\bullet \big((N \otimes_{\cat \lie} \underline{\malcev})_d\big)
\cong  
N^{\leq d} \otimes_{\cat \lie} \underline{\malcev}.
\]

In particular, there is an isomorphism in $\f_d (\gr)$:
\[
(N \otimes_{\cat \lie} \underline{\malcev})_d
\cong  
(N^{\leq d} \otimes_{\cat \lie} \underline{\malcev})_d.
\]
\end{prop}

\begin{proof}
The argument employed in the proof of Corollary \ref{cor:finite_functors_equiv_rmod} implies that $N^{\leq d} \otimes_{\cat \lie} \underline{\malcev}$ arises from a functor of $\f_d (\gr)$; explicitly:
\[
N^{\leq d} \otimes_{\cat \lie} \underline{\malcev} 
\cong 
\qgr_\bullet \big((N^{\leq d} \otimes_{\cat \lie} \underline{\malcev} )_d \big).
\]
This yields the surjection $\qgr_\bullet \big((N \otimes_{\cat \lie} \underline{\malcev})_d\big)
\twoheadrightarrow
N^{\leq d} \otimes_{\cat \lie} \underline{\malcev}$. Comparing the polynomial filtrations of both sides shows that this is an isomorphism.
\end{proof}

\begin{rem}
\label{rem:universal example}
It is instructive to consider the universal example for Proposition \ref{prop:compare_filtrations}, namely taking $N = \cat \lie$, so that $N^{\leq d}$ is $\cat ^{\leq d} \lie$. Then Proposition \ref{prop:properties_otimes_cat_opd} gives the isomorphism
 $
\cat ^{\leq d} \lie \otimes_{\cat \lie} \underline{\malcev}
\cong 
\underline{\malcev}_{[\leq d]}$.

Proposition \ref{prop:compare_filtrations} asserts that the $\cat \lie$-module $\underline{\malcev}_{[\leq d]}$ is given in arity $s$ by 
\[
\underline{\malcev}_{[\leq d]} (s) = \qgr_\bullet \big ( (\malcev^{\obar s})_d \big).
\]

That this holds can be seen by analysing the construction of $\underline{\malcev}_{[\leq d]}$. In arity $s$, this is given by 
the cokernel of 
\[
\bigoplus_{u >d} \bigoplus _{f \in \cat \lie (u, s) } \malcev^{\obar u} 
\rightarrow \malcev^{\obar s},
\]
as in Remark \ref{rem:left_adjoint_left_cat_opd}, where the action of $f$ on $\underline{\malcev}$ is induced by the Lie algebra structure of $\malcev$.

Combining Theorem \ref{thm:proj_cover_malcev_s} with Proposition \ref{prop:malcev_full_subcat}, one sees that the cokernel identifies as claimed.
\end{rem}
%%%%%%%%%%%%%%%%%%%%%%%%%%%%%%%%%%%%%%%%%%%%%%%%%%%%%%%%%%%%%
\subsection{The equivalence is symmetric monoidal}

Recall from Section \ref{subsect:convolution_cat_opd} that $\rmod[\catlie]$ is symmetric monoidal for the convolution product $\odot$ and that  we have the symmetric monoidal structure $(\propoly, \obar, \rat)$.

\begin{thm}
\label{thm:hom_underline_malcev_sym_monoidal}
The functor $\hom_{\propoly} (\underline{\malcev}, -)  : 
\propoly \stackrel{\cong}{\rightarrow} \rmod[\catlie]$ is symmetric monoidal.
\end{thm}

\begin{proof}
Using the identification given in Remark \ref{rem:identify_w_fbcr}, if we only retain the underlying $\rat \fb\op$-module structure, the result follows from Proposition \ref{prop:properties_fbcr_propoly}. 

It remains to check that this isomorphism is compatible with the right $\cat \lie$-module structures. This follows by analysing  the above construction  at the level of morphisms in $\propoly$.

 Consider two morphisms $f: \malcev^{\obar s} \rightarrow G_\bullet$ and $g :\malcev ^{\obar t} \rightarrow G'_\bullet$. One has  the tensor product 
\[
f \obar g :  \malcev ^{\obar s+t} \cong \malcev^{\obar s} \obar \malcev ^{\obar t} \rightarrow G_\bullet \obar G'_\bullet.
\]
This represents the element in $\cre_{s+t} (G_\bullet \obar G'_\bullet)_s$ given by forming the appropriate convolution product of $f \in \cre_s G_s$ and $g \in \cre_t G'_t$. 

To understand the action of the $\cat \lie$-morphisms, it suffices to consider the precomposition with morphisms of the form
\begin{eqnarray*}
\malcev^{\obar u} \obar \malcev ^{\obar t} & \rightarrow & \malcev^{\obar s} \obar \malcev ^{\obar t} \\
\malcev^{\obar s} \obar \malcev ^{\obar v} & \rightarrow & \malcev^{\obar s} \obar \malcev ^{\obar t}
\end{eqnarray*}
induced by morphisms in $\hom_\propoly (\malcev^{\obar u}, \malcev^{\obar s}) $ and $\hom_\propoly (\malcev^{\obar v}, \malcev^{\obar t}) $ respectively, where we may take $ u \geq s$ and $v \geq t$. 

These composites are determined  by the $\cat \lie$-module structures of  $\hom_\propoly (\underline{\malcev}, G_\bullet)$ and of  $\hom_\propoly (\underline{\malcev}, G'_\bullet)$ respectively. One checks that the composites  correspond to the right $\cat \lie$-module structure on the convolution product $\hom_\propoly (\underline{\malcev}, G_\bullet) \odot \hom_\propoly (\underline{\malcev}, G'_\bullet)$.
\end{proof}

%\newpage
%\input{catlie}
\section{Relating left $\cat\lie$-modules to functors on  $\gr\op$}
\label{sect:catlie}

This section revisits the relationship between left  $\cat \lie$-modules and $\f (\gr\op)$  that was established in \cite{2021arXiv211001934P}. The equivalence of categories is restated here as Theorem \ref{thm:analytic_grop}, using the exponential functor $\Phi U \lie$ in right $\cat \lie$-modules that is introduced below.

\begin{rem}
In  Section \ref{sect:compare} we analyse the relationship between this result and the analogous result for $\propoly$, Theorem \ref{thm:equiv_propoly_modcatlie}. 
\end{rem}

The operad $\uass$ encodes unital, associative algebras; it is related to the Lie operad $\lie$ by the morphism of operads $\la :  \lie \rightarrow \uass$ that encodes the underlying commutator algebra of an associative algebra (this factorizes across the associative operad $\ass$ as $\lie \rightarrow \ass \rightarrow \uass$, where $\ass \rightarrow \uass$ corresponds to forgetting the unit of an associative, unital algebra). The operad $\ucom$ encodes unital, commutative associative algebras, so that there is a morphism of operads $\uass \rightarrow \ucom$ that corresponds to forgetting about commutativity.

The results of \cite{2021arXiv211001934P} are based on the fact that the structure of the  projective generators of the category $\f_\omega (\gr\op)$ are encoded in the functor  $\gr \op \rightarrow \rmod[\cat \lie]$ given on objects by 
\[
\zed^{\star n} \mapsto \cat \uass (-, n), 
\]
where  $\cat \uass $ is considered as a right $\cat \lie$-module via the restriction of the regular right module structure along $\cat \la : \cat \lie \rightarrow \cat \uass$. In particular, this asserts that the naturality with respect to $\gr\op$ commutes with the right $\cat \lie$-module structure. A further crucial ingredient is that $\cat \uass$ is free as a right $\cat \lie$-module, a consequence of the Poincaré-Birkhoff-Witt theorem.

Here we provide a model for $\cat \uass$, together with its right $\cat \lie$-module structure,  exploiting the symmetric monoidal structure $(\rmod[\cat \opd], \odot, \rat)$ for an operad $\opd$.

\begin{prop}
\label{prop:opd_alg_right_mod}
The operad $\opd$ is an $\opd$-algebra in the category of right $\cat \opd$-modules.

If $\phi : \opd \rightarrow \ppd$ is a morphism of operads, then $\phi$ is a morphism of $\opd$-algebras in right $\cat\opd$-modules, where $\ppd$ is equipped with the respective restricted structures: i.e., considered as a right $\cat \opd$-module by restriction of the canonical right $\cat \ppd$-module structure; considered as a $\opd$-algebra by restriction of the canonical $\ppd$-algebra structure.
\end{prop}

\begin{proof}
This result is stated in \cite[Observation 9.1.3]{MR2494775} (in terms of right $\opd$-modules, defined in the operadic sense). For the convenience of the reader, a proof is outlined here.

For $n \in \nat$, by construction, there is an isomorphism of right $\cat \opd$-modules: 
\[
\cat \opd (-, n) \cong \opd(-)^{\odot n}
\]
and this is $\sym_n$-equivariant with respect to the canonical action on the left hand side and the `place permutation' action on the right hand side. The operad structure then induces the structure map:
\[
\opd(n) \otimes_{\sym_n} \opd(-)^{\odot n} \rightarrow \opd (-)
\]
in right $\cat \opd$-modules. This gives the natural $\opd$-algebra structure in right $\cat \opd$-algebras, as required.
\end{proof}

\begin{exam}
\label{exam:la}
By Proposition \ref{prop:opd_alg_right_mod},  $\lie$ is a Lie algebra in right $\cat \lie$-modules and $\uass$ is a unital, associative algebra in right $\cat \uass$-modules. Thus, by restriction along $\la : \lie \rightarrow \uass $ and its associated functor $\cat \la$, $\uass$ is a Lie algebra in right $\cat \lie$-modules and, with respect to this structure, $\lie \rightarrow \uass$ is a morphism  of Lie algebras in right $\cat\lie$-modules.
\end{exam}

The universal enveloping algebra functor $U$ from Lie algebras to unital associative algebras can be applied in the category of right $\cat \lie$-modules with respect to the structure $(\cat \lie , \odot, \rat)$. It is the left adjoint to the forgetful functor corresponding to restriction along $\la : \lie \rightarrow \uass$.

Applied to $\lie$, considered as a Lie algebra in $\rmod[\cat\lie]$, this yields $U \lie$, which is an $\uass$-algebra in right $\cat \lie$-modules. 

\begin{prop}
\label{prop:U_right_cat_Lie_modules}
The morphism $\la : \lie \rightarrow \uass$ induces 
$ 
U \lie \rightarrow \uass
$ 
that is an isomorphism of $\uass$-algebras in right $\cat \lie$-modules, where $\uass$ is equipped with the restricted module structure. 
\end{prop}

\begin{proof}
The morphism is given by the universal property of the universal enveloping algebra construction. 
 To prove that it is an isomorphism, it suffices to establish that the underlying morphism of $\rat\fb\op$-modules is an isomorphism. Since we are working over $\rat$, as exploited in \cite{2021arXiv211001934P},  it is sufficient to show that it induces an isomorphism on the associated Schur functors. 

We use the following notation for the Schur functor construction, following \cite{2021arXiv211001934P}. For $V$ a finite-dimensional vector space, let $\underline{V}$ be the $\rat\fb$-module with $\underline{V}(\mathbf{n}) = V^{\otimes n}$, the $\sym_n$-action given by place permutations of the tensor factors; this construction is natural in $V$. The  Schur functor construction then identifies as $- \otimes_\fb \underline{V}$ (naturally in $V$). This is an exact functor from $\rat\fb\op$-modules to functors from finite-dimensional $\rat$-vector spaces to $\rat$-vector spaces. 

Now,  $\uass \otimes_\fb \underline{V}$ is naturally isomorphic to the tensor algebra functor $T(V)$, since the latter is the free unital associative algebra on $V$. Likewise,  $U \lie \otimes_\fb \underline{V}$ is naturally isomorphic to the universal enveloping algebra $U \liealg (V)$ on the free Lie algebra on $V$ and the map $ (U \lie \rightarrow \uass) \otimes_\fb \underline{V}$ identifies with the canonical isomorphism 
$
U \liealg (V) \stackrel{\cong}{\rightarrow} T(V)$.
 This concludes the proof.
\end{proof}

More is true:

\begin{prop}
\label{prop:Hopf_Ulie}
The right $\cat \lie$-module $U \lie$ has a unique cocommutative Hopf algebra structure extending the associative, unital algebra structure, for which $\lie \subset U \lie$ lies in the primitives. 
\end{prop}

\begin{proof}
This generalizes the classical construction of the Hopf algebra structure on the universal enveloping algebra  of a Lie algebra. Namely, restricted to $\lie \subset U \lie$, the coproduct is defined to be 
\[
\lie \rightarrow \lie \oplus \lie \cong \lie \odot \rat \ \oplus \   \rat \odot \lie 
\rightarrow 
U \lie \odot U \lie,
\]
where the first map is the diagonal and the final map is induced by the canonical inclusions of $\lie $ and $\rat$ in $U\lie$. This is clearly a morphism of Lie algebras in right $\cat \lie$-modules, where $U \lie \odot U \lie$ is considered as the Lie algebra associated to the $\odot$-product associative algebra structure.  

Hence, this induces the coproduct $U \lie \rightarrow U\lie \odot U \lie$, a morphism of $\uass$-algebras in right $\cat\lie$-modules. Similarly, the conjugation is obtained as the anti-automorphism of algebras in right $\cat \lie$-modules extending $\lie \stackrel{ (-1)\times}{\rightarrow} \lie \subset U \lie$. This gives the required cocommutative Hopf algebra structure on $U \lie$ in $\rmod[\cat \lie]$. 
\end{proof}

\begin{rem}
\ 
\begin{enumerate}
\item 
Upon passing to the associated Schur functor $U \lie \otimes_\fb \underline{V}$ as in the proof of Proposition \ref{prop:U_right_cat_Lie_modules}, this recovers the primitively-generated Hopf algebra structure on $U \liealg (V)$ with primitives $\liealg (V) \subset U \liealg(V)$.
\item
The augmentation of the Lie operad $\lie \rightarrow I$ can be considered as a morphism of Lie algebras in $\rmod[\cat \lie]$, where the structure of $I$ is induced by restriction along the augmentation. On applying the universal enveloping algebra functor this gives a morphism of cocommutative Hopf algebras in right $\cat \lie$-modules $
U \lie \rightarrow U I$. 
Here, the underlying right $\cat \lie$-module of $U I$ is isomorphic to $\ucom$ (with `trivial' right $\cat \lie$-module structure) and, via the isomorphism of Proposition \ref{prop:U_right_cat_Lie_modules}, the underlying morphism of right $\cat \lie$-modules is the morphism 
$
\uass \rightarrow \ucom$.

On passage to Schur functors, this corresponds to the symmetrization map $T (V) \rightarrow S(V)$, where $S(V)$ is the symmetric algebra, which is a morphism of primitively-generated Hopf algebras (naturally with respect to $V$).
\end{enumerate}
\end{rem}

The exponential functor construction $\Phi$ can be applied (working in $(\cat \lie, \odot, \rat)$) to $U\lie$. 
Essentially by construction, one has: 

\begin{prop}
\label{prop:Phi_U_L_model}
The exponential functor $\Phi U \lie$ is a functor from $\gr\op$ to right $\cat \lie$-modules. This is isomorphic 
to $\cat \uass$, equipped with the structure of \cite{2021arXiv211001934P}.
\end{prop}

\begin{rem}
\label{rem:identifications_Phi_ULie}
In arity $s \in \nat$, the functor $(\Phi U \lie )(s)$ (evaluated with respect to the right $\cat \lie$-module structure) is isomorphic to the functor on $\gr\op$: 
\[
\zed^{\star t} \mapsto \cat \uass (s, t),
\]
with the right $\cat \lie$-module structure of $\Phi U \lie$ corresponding to the right $\cat \lie$-action on $\cat \uass$. 

Moreover, by the results of  \cite{2021arXiv211001934P},  $(\Phi U \lie)(s)$ is polynomial on $\gr\op$ of degree $s$ and there is a surjection 
\[
(\Phi U \lie) (s) \twoheadrightarrow (\A^{\sharp})^{\otimes s} 
\] 
that exhibits $(\Phi U \lie )(s)$ as the projective cover of $(\A^{\sharp})^{\otimes s} $ in $\f_\omega (\gr\op)$. 

It follows that the dual $D (\Phi U \lie )(s)$ is the injective envelope of $\A^{\otimes s}$ in $\propoly$. 
\end{rem}

Hence, the main result of \cite{2021arXiv211001934P} can be rewritten as follows, where $\f_\omega (\gr\op)$ is the full subcategory of analytic functors:

\begin{thm}
\label{thm:analytic_grop}
The functor $\Phi U \lie \otimes_{\cat \lie} - $ induces an equivalence of categories
\[
\Phi U \lie \otimes_{\cat \lie} - : \lmod[\cat \lie] \rightarrow \f_\omega (\gr\op).
\] 

This restricts to an equivalence of categories:
\[
\Phi U \lie \otimes_{\cat \lie} - : \lmod[\cat \lie]^{< \infty} \rightarrow \f_{<\infty} (\gr\op),
\] 
where $\lmod[\cat \lie]^{< \infty} \subset \lmod [\cat \lie]$ is the full subcategory of left $\cat \lie$-modules $M$ such that $M (n)=0$ for $n \gg 0$.
\end{thm}

\begin{exam}
For a Lie algebra $\g$, one has the associated left $\cat \lie$-module $\underline{\g}$ given by $\mathbf{n} \mapsto \g^{\otimes n}$, with morphisms of $\cat \lie$ acting via the Lie algebra structure of $\g$.  As proved in \cite{2021arXiv211001934P},  one has the natural isomorphism  
$
\Phi U \lie \otimes_{\cat \lie} \underline{\g} \cong \Phi (U \g).
$ 
as functors in $\f_\omega (\gr\op)$.

One can also consider $\lie$ as a right $\cat \lie$-module and form $\lie \otimes_{\cat \lie} \underline{\g}$. Since $\lie$ identifies with $\cat \lie (-, 1)$ as a right $\cat \lie$-module, which is projective, using a Yoneda lemma type argument, the above isomorphism `restricts' to 
$
\lie \otimes_{\cat \lie} \underline{\g}
\cong 
\g$,
and this is an isomorphism of Lie algebras, where the Lie algebra structure on the left is induced by considering $\lie$ as a Lie algebra in right $\cat \lie$-modules.

The naturality of these isomorphisms implies that they apply to $\malcev$, considered as a Lie algebra in $\propoly$. Hence there are natural isomorphisms:
\begin{eqnarray*}
\Phi U \lie \otimes_{\cat \lie} \underline{\malcev} &\cong & \Phi (U \malcev) \\
 \lie \otimes_{\cat \lie} \underline{\malcev } & \cong & \malcev.
\end{eqnarray*}
\end{exam}

\section{Relating the covariant and contravariant cases}
\label{sect:compare}

The purpose of this section is to make explicit the relationship between Theorem \ref{thm:equiv_propoly_modcatlie} and Theorem \ref{thm:analytic_grop}. This is based on the duality between $\propoly $ and $\f_\omega (\gr\op)$.

This involves relating $\underline{\malcev}$, which is a left $\cat \lie$-module in $\propoly$, and $\Phi U\lie$, a right $\cat \lie$-module in $\f (\gr\op)$. Theorem \ref{thm:otimes_gr_PhiUL_malcev} gives a concrete `duality' statement. Theorem \ref{thm:equiv_Umalcev_PhiUlie} explains {\em why} $\Phi U \lie$ arises naturally in the theory.

Essentially this section serves to verify that the theories behave as they should.

%%%%%%%%%%%%%%%%%%%%%%%%%%%%%%%%%%%%%%%%%%%%%%%%%%%%%%%%%%%%%%%%%%%%%%%%%%
\subsection{Analysing $\Phi U \lie$ via $\propoly$}

Theorem \ref{thm:equiv_propoly_modcatlie} gives the equivalence of categories:
\[
\hom_{\propoly} (\underline{\malcev}, - ) : \propoly \rightarrow \rmod[\cat \lie].
\]
Underlying this equivalence is the fact that $\hom_{\propoly} (\underline{\malcev}, \malcev )$ is naturally isomorphic to $\lie$ as a Lie algebra in $\rmod [\cat \lie]$, by Proposition \ref{prop:malcev_full_subcat}.

Now, the functor $\pbif : \Gamma \mapsto P_\Gamma$ is an exponential functor, via the natural equivalence $P_{\zed^{\star s} } \cong P_\zed^{\otimes s}$. This translates to $\propoly$:

\begin{lem}
\label{lem:expo_propoly}
The functor $\Gamma \mapsto \qgr_\bullet P_\Gamma$ from $\gr\op$ to $\propoly$ is exponential. 

Namely, $\qgr_\bullet \pbif \cong \Phi (\qgr_\bullet  P_\zed)$, where the exponential functor construction $\Phi$ is applied in $(\propoly, \obar, \rat)$.
\end{lem}

\begin{proof}
The result follows from Proposition \ref{prop:qgr_tensor_Pzed}.  
\end{proof}

\begin{rem}
One would have liked to prove Lemma \ref{lem:expo_propoly} by appealing to the  `symmetric monoidality' of $\qgr_\bullet$ given  by Proposition \ref{prop:propoly_sym_mon}; unfortunately the latter  is only established restricted to $\f_{<\infty} (\gr)$, considered as a full subcategory of $\propoly$. Proposition \ref{prop:qgr_tensor_Pzed} allows this problem to be circumvented.
\end{rem}

\begin{thm}
\label{thm:equiv_Umalcev_PhiUlie}
Under the equivalence of Theorem \ref{thm:equiv_propoly_modcatlie}, $\qgr_\bullet \pbif $ corresponds to $\Phi U \lie$, considered as a functor from $\gr\op$ to $\rmod [\cat\lie]$.
\end{thm}

\begin{proof}
The functor $\hom_{\propoly} (\underline{\malcev}, -) $ is symmetric monoidal, by Theorem \ref{thm:hom_underline_malcev_sym_monoidal}. In particular, it commutes with the exponential functor construction $\Phi$, up to natural isomorphism.  Since $\Gamma \mapsto \qgr_\bullet P_\Gamma$ is naturally isomorphic to $\Phi (\qgr_\bullet P_\zed)$ by Lemma \ref{lem:expo_propoly}, it suffices to show that there is an isomorphism 
\[
\hom_{\propoly} (\underline{\malcev}, \qgr_\bullet P_\zed) \cong U \lie
\]
as Hopf algebras in $\rmod[\cat \lie]$, where the Hopf algebra structure of the domain is induced by that of $\qgr _\bullet P_\zed$ in $\propoly$.

Proposition \ref{prop:U_malcev} gives the isomorphism $U \malcev \cong \qgr_\bullet P_\zed$ of Hopf algebras in $\propoly$. Thus, we require to establish the isomorphism of Hopf algebras
\[
\hom_{\propoly} (\underline{\malcev}, U \malcev) \cong U \lie.
\] 

Now, $\hom_{\propoly} (\underline{\malcev}, \malcev) \cong \lie$ as Lie algebras in $\rmod [\cat \lie]$, by Proposition \ref{prop:malcev_full_subcat} and the inclusion $\malcev \hookrightarrow U \malcev$ induces an inclusion of Lie algebras $\lie \hookrightarrow \hom_{\propoly} (\underline{\malcev}, U \malcev)$ in $\rmod[\cat \lie]$. This induces a morphism of Hopf algebras in $\rmod [\cat \lie]$:
\[
U \lie \rightarrow \hom_{\propoly} (\underline{\malcev}, U \malcev).
\]
To check that this is an isomorphism, one proceeds as in the proof of Proposition \ref{prop:U_malcev}, i.e., by reducing to considering the associated graded of $U \malcev$  and forgetting the right $\cat \lie$-module structure.
\end{proof}

By construction, the Passi functors are encoded in $\qgr_\bullet P_\zed$, considered as an object of $\propoly$. Combining Proposition \ref{prop:U_right_cat_Lie_modules} with the identification established above in the proof of  Theorem \ref{thm:equiv_Umalcev_PhiUlie}, one has:

\begin{cor}
\label{cor:passi_uass}
The Passi functors $\qgr_\bullet P_\zed$, considered as forming a unital associative algebra in $\propoly$, are the image under the equivalence of Theorem \ref{thm:equiv_propoly_modcatlie} of $\uass$, considered as a unital associative algebra in right $\cat \lie$-modules.
\end{cor}

%%%%%%%%%%%%%%%%%%%%%%%%%%%%%%%%%%%%%%%%%%%%%%%%%%%%%%%%%%%%%%%%%%%%%%%%%%%%%%%%%%%%%%
\subsection{Relating $\underline{\malcev}$ and $\Phi U \lie$}

The functor $\malcev^{\obar s}$ in $\propoly$ is the projective cover of $\A^{\otimes s}$, for $s \in \nat$, by Theorem \ref{thm:proj_cover_malcev_s}.  By Proposition \ref{prop:malcev_full_subcat}, these assemble to form the left $\cat \lie$-module $\underline{\malcev}$, where the $\cat\lie$-module structure encodes the morphisms between such functors. 

In the contravariant setting, as explained in Remark \ref{rem:identifications_Phi_ULie}, the functor $(\Phi U \lie)(s)$ is the projective cover in $\f_\omega (\gr\op)$ of $(\A^{\sharp})^{\otimes s}$ and these assemble to give the right $\cat \lie$-module $\Phi U \lie$. 

The purpose of this section is to relate these structures, showing that they are `dual' by using $\otimes_\gr$.

Recall from  Proposition \ref{prop:qgr_tensor_Pzed} that there is a $\sym_t$-equivariant isomorphism $
\qgr_\bullet (P_{\zed^{\star t}}) \cong (\qgr _\bullet P_\zed)^{\obar t}
$, 
for $t \in \nat$. In particular,  via this isomorphism one deduces that $\qgr_\bullet P_\zed$ is a unital, associative algebra in $\propoly$. This gives:

\begin{lem}
\label{lem:qgr_PZ_catuass}
The association $t \mapsto \qgr_\bullet (P_{\zed^{\star t}}) \cong (\qgr _\bullet P_\zed)^{\obar t}$ defines a left $\cat \uass$-module in $\propoly$, denoted $(\qgr_\bullet P_\zed)^{\obar *}$.
\end{lem}

On forming the tensor product over $\gr$ (using $\ogr$ of Proposition \ref{prop:otimes_gr_propoly_analytic}), one has
$
\Phi U \lie \ogr (\qgr_\bullet P_\zed)^{\obar *}$,
which has the structure of a right $\cat \lie$, left $\cat \uass$ bimodule, with the actions derived respectively from $\Phi U \lie$ and $(\qgr_\bullet P_\zed)^{\obar *}$. Explicitly, in bi-arity $(s,t)$ (where $s$ corresponds to the $\cat \lie$-structure and $t$ to the $\cat \uass$-structure), this identifies as:
\[
(\Phi U \lie \ogr (\qgr_\bullet P_\zed)^{\obar *})(s,t)
=
\Phi U  \lie (s) \ogr (\qgr_\bullet P_\zed)^{\obar t}.
\] 

\begin{rem}
Since $\Phi U \lie (s)$ is polynomial, in particular lies in  $\f_\omega (\gr\op)$,  $\Phi U  \lie (s) \ogr (\qgr_\bullet P_\zed)^{\obar t}$ is naturally isomorphic to  $\Phi U  \lie (s) \otimes_\gr P_\zed^{\otimes  t} \cong \Phi U  \lie (s) \otimes_\gr P_{\zed^{\star  t}}$, by Proposition \ref{prop:compatibility_otimes_gr}. 
\end{rem}

Recall that $P_\zed$ splits as $\pbar \oplus \rat$, hence $\Phi U \lie \ogr (\qgr_\bullet P_\zed)^{\obar *}$ contains $\Phi U \lie \ogr (\qgr_\bullet \pbar)^{\obar *}$ as a direct summand. The canonical surjection $\pbar \twoheadrightarrow \A$ induces $(\qgr_\bullet \pbar)^{\obar s} \twoheadrightarrow \qgr_\bullet \A^{\otimes s}$, for $s \in \nat$. Together with the surjection 
$  (\Phi U \lie ) (s) \twoheadrightarrow (\A^\sharp)^{\otimes s}$ in $\f_{< \infty} (\gr\op)$, this induces
\[
(\Phi U \lie) (s) \ogr (\qgr_\bullet \pbar)^{\obar s}
\twoheadrightarrow 
 (\A^\sharp)^{\otimes s} \otimes_\gr \qgr_\bullet \A^{\otimes s} 
 \cong 
 \hom_{\f_{<\infty} (\gr)} (\A^{\otimes s} , \A^{\otimes s}).
 \]
 
\begin{lem}
\label{lem:iso_pbar}
For $s \in \nat$, the map $(\Phi U \lie) (s) \ogr (\qgr_\bullet \pbar)^{\obar s}
\twoheadrightarrow 
 \hom_{\f_{<\infty} (\gr)} (\A^{\otimes s} , \A^{\otimes s})$ is an isomorphism of $\sym_s$-bimodules.
\end{lem} 
 
\begin{proof}
By construction, the map is a morphism of  $\sym_s$-bimodules, hence it suffices to show that it is an isomorphism. 
 
Now, $(\Phi U \lie) (s)$ has polynomial degree $s$ with respect to $\gr\op$ and the kernel of $(\Phi U \lie ) (s) \twoheadrightarrow (\A^\sharp)^{\otimes s}$ has polynomial degree less than $s$. The functor $- \ogr (\qgr_\bullet \pbar)^{\obar s}$ is exact and vanishes on functors of polynomial degree less than $s$; the result follows.
\end{proof}

\begin{nota}
For $s \in \nat$, let $\widetilde{\id_s} \in (\Phi U \lie \ogr (\qgr_\bullet P_\zed)^{\obar *})(s,s)$ be the element corresponding to the identity morphism of $\hom_{\f_{<\infty} (\gr)} (\A^{\otimes s} , \A^{\otimes s})$ via the isomorphism of Lemma \ref{lem:iso_pbar} and the inclusion $\Phi U \lie \ogr (\qgr_\bullet \pbar)^{\obar *}
\subset \Phi U \lie \ogr (\qgr_\bullet P_\zed)^{\obar *}$.
\end{nota}

Restricting along $\cat \lie \rightarrow \cat \uass$, induced by $\lie \rightarrow \uass$, $\cat \uass$ can be considered as a right $\cat \lie$,  left $\cat \uass$ bimodule. Proposition \ref{prop:Phi_U_L_model} then leads to:

\begin{prop}
\label{prop:iso_cat_uass}
There is an isomorphism of right $\cat \lie$, left $\cat \uass$ bimodules
\[
\cat \uass \stackrel{\cong}{\rightarrow} \Phi U \lie \ogr (\qgr_\bullet P_\zed)^{\obar *}.
\]

This is determined by the property that, for $s \in \nat$, the identity morphism of $\cat \uass (s,s)$ is sent to 
$\widetilde{\id_s} \in (\Phi U \lie \otimes_\gr (\qgr_\bullet P_\zed)^{\obar *})(s,s)$.
\end{prop}

\begin{proof}
For $t \in \nat$, one has the isomorphism $ (\qgr_\bullet P_\zed)^{\obar t} \cong \qgr_\bullet P_{\zed^{\star t}}$. Hence, for $s \in \nat$:
\[
(\Phi U \lie \ogr (\qgr_\bullet P_\zed)^{\obar *})(s,t) 
\cong 
\Phi U \lie(s) \ogr (\qgr_\bullet P_{\zed^{\star t}})
\cong 
\Phi U \lie (\zed^{\star t})(s) ,
\]
where the second isomorphism follows from the Yoneda lemma. 

From the construction, one identifies 
$
\Phi U \lie (\zed^{\star t}) \cong (U \lie)^{\odot t}
$, 
where the convolution product is formed in right $\cat \lie$-modules. In bi-arity $(s,t)$, this is isomorphic to $\cat \uass (s,t)$, by construction of the latter, using the isomorphism $U \lie \cong \uass$ of Proposition \ref{prop:U_right_cat_Lie_modules}, as in Proposition  \ref{prop:Phi_U_L_model}. 

To complete the proof of the first statement it remains to check that this corresponds to an isomorphism of bimodules. This is deduced from Proposition \ref{prop:U_right_cat_Lie_modules} together with  the properties of the exponential functor construction. 

Since the isomorphism is one of left $\cat \uass$-modules, by Yoneda it is uniquely determined by the images of the identity elements  in $\cat \uass (s,s)$, for $s \in \nat$. To conclude, one checks from the above identifications that these images are the elements $\widetilde{\id_s}$, as required.
\end{proof}

This Proposition is the key input in the following result that relates $\underline{\malcev}$ and $\Phi U \lie$.

\begin{thm}
\label{thm:otimes_gr_PhiUL_malcev}
There is a unique isomorphism of $\cat \lie$-bimodules 
\[
\cat \lie \stackrel{\cong}{\rightarrow} \Phi U \lie \ogr \underline{\malcev}
\]
that fits into the commutative diagram:
\begin{eqnarray}
\label{eqn:diag_PhiULie_malcev}
\xymatrix{
\cat \lie 
\ar[r]^(.4)\cong 
\ar@{^(->}[d] 
&
\Phi U \lie \ogr \underline{\malcev}
\ar@{^(->}[d]
\\
\cat \uass 
\ar[r]_(.4)\cong 
&
\Phi U \lie \ogr (\qgr_\bullet P_\zed)^{\obar *},
}
\end{eqnarray}
in which the left hand vertical map is induced by $\lie \hookrightarrow \uass$, the right hand vertical map by the canonical inclusion $\malcev \hookrightarrow \qgr_\bullet P_\zed$, and the lower horizontal isomorphism is given by 
 Proposition \ref{prop:iso_cat_uass}.
\end{thm}

\begin{proof}
The inclusion of $\rat$-linear categories $\cat \lie \hookrightarrow \cat \uass$ can be considered as a morphism of $\cat \lie$-bimodules, using the restricted structure on $\cat \uass$. Likewise, since $\malcev \hookrightarrow \qgr_\bullet P_\zed$ is a morphism of Lie algebras in $\propoly$ (for the commutator bracket structure on $\qgr_\bullet P_\zed$), the  map $\Phi U \lie \ogr \underline{\malcev} \rightarrow \Phi U \lie \ogr (\qgr_\bullet P_\zed)^{\obar *}$ is a morphism of $\cat \lie$-bimodules. It is injective since, as an object of $\propoly$, $\malcev$ is a direct summand of $\qgr_\bullet P_\zed$.

Hence, to construct the morphism of $\cat \lie$-bimodules, it suffices to show that the composite around the bottom of the  diagram under construction (\ref{eqn:diag_PhiULie_malcev}) factors across $\Phi U \lie \otimes_\gr \underline{\malcev}$. By the characterization given in Proposition \ref{prop:iso_cat_uass}, this reduces to showing that, for each $s \in \nat$, the element $\widetilde{\id_s}$, lies in the image of $\Phi U \lie \ogr \underline{\malcev} \rightarrow \Phi U \lie \ogr (\qgr_\bullet P_\zed)^{\obar *}$ in bi-arity $(s,s)$.

Now, $\malcev$ is a subobject of $\qgr_\bullet \pbar \subset \qgr_\bullet P_\zed$.  The inclusion $\malcev \hookrightarrow \qgr_\bullet \pbar $ induces an isomorphism 
\[
(\Phi U \lie \ogr \underline{\malcev}) (s,s) 
\stackrel{\cong}{\rightarrow} 
(\Phi U \lie \ogr (\qgr_\bullet \pbar)^{\obar *}) (s,s)
\] 
since  the projection $\qgr_\bullet \pbar \twoheadrightarrow \qgr_\bullet \A$ factors across the retract $\qgr_\bullet \pbar \twoheadrightarrow \malcev$. Hence, by the construction of the elements $\widetilde{\id_s}$, one obtains the required factorization. 

The resulting map $\cat \lie \rightarrow \Phi U \lie \ogr \underline{\malcev}$ is clearly injective. To conclude, it suffices to observe that the underlying bigraded vector spaces are of finite type (i.e.,  finite dimensional in each bi-arity) and isomorphic. This follows from Theorem \ref{thm:analytic_grop} together with the fact that $- \ogr \malcev^{\obar t}$ is equivalent to the $t$th cross-effect functor when restricted to objects of $\f_\omega (\gr\op)$ that take finite-dimensional values. The latter fact follows by duality from the corepresenting property of $\malcev^{\obar t}$ given in  Theorem \ref{thm:proj_cover_malcev_s}.
\end{proof}

Theorem \ref{thm:otimes_gr_PhiUL_malcev} has the following useful Corollary:

\begin{cor}
\label{cor:otimes_compatibility}
For $M \in \ob \lmod[\cat \lie]$ and $N \in \ob \rmod [\cat \lie]$, there is a canonical isomorphism:
\[
(\Phi U \lie \otimes _{\cat \lie} M) \ogr (N \otimes_{\cat \lie} \underline{\malcev})
\cong 
N \otimes_{\cat \lie} M.
\]
\end{cor}

\begin{proof}
The left hand side is naturally isomorphic to 
\[
(N \otimes_{\cat \lie} \underline{\malcev})
\widetilde{\otimes}_{\gr\op}
(\Phi U \lie \otimes _{\cat \lie} M)
\cong 
N \otimes_{\cat \lie} (\underline{\malcev}
\otimes_{\gr\op}
\Phi U \lie) \otimes _{\cat \lie} M,
\]
reversing the order of the tensor factors and then using associativity. (Here the tensor product $\widetilde{\otimes}_{\gr\op} : \propoly \times \f_\omega (\gr\op) \rightarrow \kmod$ is  as in Proposition \ref{prop:otimes_gr_propoly_analytic}, up to the change of order of the factors.)

Theorem \ref{thm:otimes_gr_PhiUL_malcev} gives that $\underline{\malcev}
\widetilde{\otimes}_{\gr\op}
\Phi U \lie$ is isomorphic to $\cat \lie$ as a $\cat \lie$-bimodule. The result follows, since 
$N \otimes_{\cat \lie} \cat \lie \otimes _{\cat \lie} M$ is naturally isomorphic to $N \otimes_{\cat \lie} M$.
\end{proof}

\begin{rem}
\label{rem:order_otimes_catlie}
In the isomorphism of Corollary \ref{cor:otimes_compatibility}, on the left hand side $M$ appears on the left of  $\ogr$ whereas, on the right hand side, it is on the right of $\otimes_{\cat \lie}$. This reflects the fact that left $\cat \lie$-modules correspond to analytic functors on $\gr\op$, which can be considered as {\em right} $\rat \gr$-modules, whereas right $\cat \lie$-modules are related to {\em left} $\rat \gr$-modules.
\end{rem}

%\newpage
%\input{bifunctors}
\section{Bimodules and Bifunctors}
\label{sect:bifunctors}

The purpose of this section is to put together the relations between left $\cat \lie$-modules and $\f (\gr\op)$ and right $\cat \lie$-modules and $\f (\gr)$ so as to explain the relationship between $\cat \lie$-bimodules and bifunctors, i.e., objects of $\f(\gr\op \times \gr)$.   

Due to the {\em analyticity} arising in considering functors on $\gr\op$ and the {\em pro-polynomiality} for functors of $\gr$, the general case has an unavoidable technical aspect. If one restricts to {\em polynomial bifunctors}, then the situation is much simpler. On first reading, the reader is encouraged to focus upon this case, for which the results are immediate consequences of the earlier ones.  

The general case can be modelled by passing to the appropriate category of ind-pro objects. Rather than follow this formal approach, we phrase the construction in terms of analytic functors on $\gr\op$ with values in $\propoly$.

Throughout, $\kring= \rat$.
%%%%%%%%%%%%%%%%%%%%%%%%%%%%%%%%%%%%%%%%%%%%%%%%%%%%%%%%%%%%%%%%%%%%%%%%%%
\subsection{$\cat \lie$-bimodules}
\label{subsect:catlie_bimodules}

By definition, a $\cat \lie$-bimodule is an object of $\f (\fb\op \times \fb)$ equipped with commuting left and right $\cat \lie$-actions. Explicitly, a $\cat \lie$-bimodule is a bigraded $\rat$-module, given by  $M (s,t)$, for $s , t \in \nat$, equipped with structure morphisms:
\begin{eqnarray*}
&&\cat \lie (t ,u) \otimes M (s,t) \rightarrow M (s,u) \\
&& M (s,t) \otimes  \cat \lie (r ,s)\rightarrow M (r,t)
\end{eqnarray*}
that commute and satisfy the unital and associativity constraints. (This fixes our conventions for denoting the bi-arities: for fixed $s$, $M(s,-)$ is a {\em left} $\cat\lie$-module and, for fixed $t$, $M(-,t)$ is a {\em right} $\cat \lie$-module.)

Then, extending Notation \ref{nota:catopd_modules}:

\begin{nota}
\label{nota:cat_lie_bimod}
Denote by 
\begin{enumerate}
\item 
$\bimod[\cat \lie]$ the category of $\cat \lie$-bimodules; 
\item 
$\bimod[\cat \lie]^{<\infty} \subset \bimod[\cat \lie]$ the full subcategory of $\cat \lie$-bimodules with finite support (i.e., such that $M(s,t)=0$ for $s+t \gg 0$).
\end{enumerate}
\end{nota}

One has the following standard identification:

\begin{lem}
\label{lem:bimodules_left_right}
The category $\bimod[\cat \lie]$ is equivalent to both the following
\begin{enumerate}
\item
the category of left $\cat \lie$-modules in $\rmod[\cat \lie]$; 
\item 
the category of right $\cat \lie$-modules in $\lmod[\cat \lie]$.
\end{enumerate}
\end{lem}

The canonical filtrations of Section \ref{subsect:canon_filt_catopd_mod} for left (respectively right) $\cat \lie$-modules induce the increasing filtration corresponding to the subobjects $M_{\leq d} \hookrightarrow M$ and the decreasing filtration corresponding to the quotients $M \twoheadrightarrow M^{\leq e}$
in  $\bimod[\cat \lie]$.

As usual, the tensor product $\otimes_{\cat \lie}$ induces a monoidal structure on bimodules:

\begin{prop}
\label{prop:monodial_bimodules}
There is a monoidal structure $(\bimod[\cat \lie], \otimes_{\cat \lie} , \cat \lie)$. 
\end{prop}

%%%%%%%%%%%%%%%%%%%%%%%%%%%%%%%%%%%%%%%%%%%%%%%%%%%%%%%%%%%%%%%%%%%%%%%%%%%
\subsection{The convolution product on $\cat \lie$-bimodules}

The Day convolution $\odot$ extends to a symmetric monoidal structure on $\lmod [\cat \lie]$ and on $\rmod [\cat \lie]$, as recalled in Section \ref{subsect:convolution_cat_opd}. To consider bimodules, we  start with the Day convolution product for $\rat\fb$-bimodules: 

\begin{lem}
\label{lem:biconvolution_fb}
There is a symmetric monoidal structure $(\f (\fb\op \times \fb), \circledcirc , \rat) $, where $\rat$ is considered as a $\rat\fb$-bimodule concentrated in arity $(0,0)$. The convolution product $\circledcirc$ is given on objects $B_1$, $B_2$ by 
\[
(B_1 \circledcirc B_2) (m,n) 
:= \bigoplus_{i+j =m} 
\bigoplus_{s+t=n}
\big (B_1 (i,s) \otimes B_2 (j,t)\big) \uparrow_{\sym_i\op \times \sym_j\op }^{\sym_m\op} \uparrow _{\sym_s \times \sym_t}^{\sym_n}. 
\]
\end{lem}

\begin{rem}
\label{rem:biconvolution}
This is related to the Day convolution product for $\f (\fb\op)$ and for $\f (\fb)$ as follows. 
 The category $\f(\fb\op \times \fb)$ is equivalent to $\f (\fb\op; \f(\fb))$. Here $\f(\fb)$ is equipped with the Day convolution product $\odot$. Then 
\[
(B_1 (-,*) \circledcirc B_2 (-,*)) (m) =    \bigoplus_{i+j =m} B_1 (i, *) \odot B_2 (j,*) \uparrow_{\sym_i\op \times \sym_j\op }^{\sym_m\op} 
\]
and the right hand side is simply the convolution product in $\f (\fb\op; \f (\fb))$ with respect to $\fb\op$ (using the convolution symmetric monoidal structure $\odot$ on $\f (\fb)$).
\end{rem}

\begin{prop}
\label{prop:convolution_bimodules}
The convolution product provides a symmetric monoidal structure $(\bimod[\cat \lie], \circledcirc , \rat)$. 
\end{prop}

\begin{proof}
The proof generalizes the analysis of $\circledcirc$ on $\f (\fb\op \times \fb)$ given in Remark \ref{rem:biconvolution}, using that $\bimod[\cat \lie]$ is equivalent to the category of left $\cat \lie$-modules in $\rmod[\cat \lie]$, by Lemma \ref{lem:bimodules_left_right}. One forms the convolution product in left $\cat \lie$-modules with respect to the symmetric monoidal structure on $\rmod [\cat \lie]$ given by the convolution product. The details are left to the reader.
\end{proof}

%%%%%%%%%%%%%%%%%%%%%%%%%%%%%%%%%%%%%%%%%%%%%%%%%%%%%%%%%%%%%%%%%%%%%%%%%%%
\subsection{Bifunctors}
\label{subsect:bimod_functors}

We now  consider bifunctors, i.e., the category $\f (\gr\op \times \gr)$. For a bifunctor $F$ and $d,e \in \nat$, there are canonical inclusions 
 $
\pgrop _d F \hookrightarrow \pgrop _{d+1} F \hookrightarrow F
$ 
and canonical surjections $
F \twoheadrightarrow \qgr_{e+1} F \twoheadrightarrow \qgr_e F,
$ 
given by using the respective polynomial filtrations with respect to $\gr\op$ and $\gr$ respectively.

\begin{lem}
\label{lem:compare_pq}
For a bifunctor $F$ and $d, e \in \nat$, there is a natural morphism $ \qgr_e \pgrop _d F \rightarrow \pgrop_d \qgr_e F  $. 
\end{lem}

\begin{proof}
Applying the functor $\pgrop_d$ to the canonical surjection $F \twoheadrightarrow \qgr_e F$ gives $\pgrop_d F \rightarrow \pgrop_d  \qgr_e F$. Since the codomain has polynomial degree $e$ with respect to $\gr$, this factorizes naturally as required.
\end{proof}

\begin{rem}
\label{rem:analytic_bifunctor}
This gives the notion of a bifunctor that is analytic with respect to $\gr\op$: namely, a bifunctor $F$ for which the canonical morphism
$
\lim_{\substack {\rightarrow \\ d} } \pgrop_d F \rightarrow F
$ 
is an isomorphism. This is equivalent to the condition that, for each $\Gamma \in \ob \gr$, the functor $F (- \times \Gamma)$ in $\f (\gr\op)$ is analytic. 

This can be placed in a more general framework: for any abelian category $\cala$, the above definition generalizes to give the full subcategory $\f_\omega (\gr\op; \cala) \subset \f (\gr\op; \cala)$ of analytic functors. Then, using the natural equivalence 
$\f (\gr\op \times \gr) \cong \f (\gr\op ; \f(\gr))$, the category of analytic bifunctors introduced above is the full subcategory
\[
\f_\omega (\gr\op; \f(\gr)) \subset \f (\gr\op; \f(\gr))\cong \f (\gr\op \times \gr).
\]
\end{rem}

The difference functors $\dgr$ and $\dgrop$  act on $\f (\gr\op \times \gr)$ and commute, so that one can adopt the following definition of polynomiality for bifunctors:

\begin{defn}
\label{defn:poly_bifunctor}
A functor $F \in \ob \f (\gr\op \times \gr)$ has polynomial degree $d$ if  $(\dgr)^i (\dgrop)^j F =0$ for all $i, j \in \nat$ such that $i+j>d$.
\end{defn}

\begin{nota}
The full subcategory of polynomial bifunctors of degree $d$ is denoted $\f_d (\gr\op \times \gr)$ and $\f_{<\infty} (\gr\op \times \gr)$ denotes $\bigcup_d  \f_d (\gr\op \times \gr)$. 
\end{nota}

Basic examples of polynomial bifunctors are given by using the exterior tensor product:

\begin{lem}
\label{lem:poly_bifunctors_exterior_tensor}
For $i, j\in \nat$, the exterior tensor product $\f (\gr\op) \times \f (\gr) \stackrel{\boxtimes}{\rightarrow} \f (\gr\op \times \gr)$ restricts to  $\f_i (\gr\op) \times \f_j (\gr) \stackrel{\boxtimes}{\rightarrow} \f_{i+j} (\gr\op \times \gr)$.
\end{lem}

Under the hypothesis of polynomiality, one has the following, which ensures that forming the respective polynomial filtrations behaves as expected.

\begin{prop}
\label{prop:compare_pq_polynomial}
For $F \in \ob \f_{< \infty} (\gr\op \times \gr)$ and $d, e \in \nat$, the  natural morphism $ \qgr_e \pgrop _d F \rightarrow \pgrop_d \qgr_e F  $ is an isomorphism.
\end{prop}

\begin{proof}
Restricted to $\f_{< \infty}(\gr\op \times \gr)$, the functors $\pgrop_d $ and $\qgr_e$ are exact. Using this, one checks that the natural transformation of Lemma \ref{lem:compare_pq} is an isomorphism, as follows. 
 Write $F'$ for the kernel of the natural surjection $F \twoheadrightarrow \qgr_e F$. Then, applying the natural surjection $\pgrop_d \twoheadrightarrow \qgr_e \pgrop_d$ gives a commutative diagram:
\[
\xymatrix{
\pgrop_d F' 
\ar@{->>}[d]
\ar[r]
&
\pgrop_d F 
\ar@{->>}[d]
\ar[r]
&
\pgrop_d \qgr_e F 
\ar@{->>}[d]|{\cong}
\\
\qgr_e \pgrop_d F' 
\ar[r]
&
\qgr_e \pgrop_d F
\ar[r]
&
\qgr_e\pgrop_d \qgr_e F 
}
\] 
in which the rows are short exact.

The result is equivalent to showing that the bottom right horizontal map is an isomorphism. This is equivalent to the vanishing of $ \qgr_e \pgrop_d F' $. Now $\qgr_e F'=0$, by construction; applying $\qgr_e$ to the inclusion $\pgrop_d F' \hookrightarrow F'$ therefore gives $\qgr_e \pgrop_d F' =0$, by exactness of $\qgr_e$, as required.
\end{proof}

%%%%%%%%%%%%%%%%%%%%%%%%%%%%%%%%%%%%%%%%%%%%%%%%%%%%%%%%%%%%%%%%%%%%%%%%%%%%%%%%%%%%%%%%%%
\subsection{Mixing analyticity and pro-polynomiality}

Our aim is to use $\bimod[\cat \lie]$ to model bifunctors. Theorem \ref{thm:analytic_grop} leads to an {\em analyticity} condition with respect to $\gr\op$, hence one should expect to restrict to analytic bifunctors, $\f _\omega (\gr\op; \f(\gr))$, as in Remark \ref{rem:analytic_bifunctor}. 

The relationship between $\rmod [\cat \lie]$ and $\f (\gr)$ as made precise by Theorem \ref{thm:equiv_propoly_modcatlie}, makes clear that one should not expect to model all analytic bifunctors on the nose. Namely, the target category should be replaced by $\propoly$, so that one can consider the abelian category $\f (\gr\op; \propoly)$ and its full subcategory $\f_\omega(\gr\op; \propoly)$ of analytic functors  (defined as in  Remark \ref{rem:analytic_bifunctor}). 

\begin{rem}
The above reflects the choice to work with $\f_\omega(\gr\op; \propoly)$, based upon the equivalence $\f (\gr\op \times \gr) \cong \f (\gr\op ; \f (\gr))$. One could equally well have used the equivalence $\f (\gr\op \times \gr) \cong \f (\gr ; \f (\gr\op))$. In this case, one is lead to study the appropriate category of pro-polynomial functors on $\gr$ with values in analytic functors on $\gr\op$. 
\end{rem}

\begin{lem}
\label{lem:include_finite_bifunctors}
The functor $\qgr_\bullet : \f (\gr)\rightarrow \propoly$ induces a functor $\f (\gr\op \times \gr) \cong \f (\gr\op ; \f(\gr)) 
\rightarrow \f(\gr\op;\propoly)$. 

This induces an exact functor $\f_{< \infty} (\gr\op \times \gr) \rightarrow \f_\omega (\gr\op ; \propoly)$ that is the inclusion of a full subcategory. 
\end{lem}

\begin{proof}
The first statement is immediate. The second follows using the fact that the restriction of $\qgr_\bullet$ to $\f_{<\infty} (\gr)$ is exact; one checks easily that the given functor takes values in analytic functors with respect to $\gr\op$.  The identification of the essential image is straightforward.
\end{proof}

\begin{prop}
\label{prop:complete_bifunctors}
\ 
\begin{enumerate}
\item 
The category $\f_\omega (\gr\op; \propoly)$ is abelian and the inclusion functor $$\f_\omega (\gr\op; \propoly)
\hookrightarrow \f(\gr\op; \propoly)$$ is exact.
\item 
The completion functor $\compl : \propoly \rightarrow \f(\gr)$ induces an exact functor $\f(\gr\op; \propoly) \rightarrow \f (\gr\op; \f(\gr)) \cong \f (\gr\op \times \gr)$. 
\item 
The composite gives an exact functor $\f _\omega (\gr\op; \propoly) \rightarrow \f (\gr\op \times \gr)$.
\item 
The composite of $\f _\omega (\gr\op; \propoly) \rightarrow \f (\gr\op \times \gr)$ with $\f_{< \infty} (\gr\op \times \gr) \rightarrow \f_\omega (\gr\op ; \propoly)$ of Lemma \ref{lem:include_finite_bifunctors} is naturally equivalent to the inclusion $\f_{< \infty} (\gr\op \times \gr) \subset \f (\gr\op \times \gr)$.
\end{enumerate}
\end{prop}

\begin{proof}
The first statement is clear; the second follows from the exactness of the completion functor given by Proposition \ref{prop:compl_exact}. The restriction to polynomial bifunctors is analysed directly from the definitions.
\end{proof}

The following is immediate:

\begin{prop}
\label{prop:sym_mon_odot_bifunctors}
The symmetric monoidal structure $(\propoly, \odot, \rat)$ induces a symmetric monoidal structure on $\f (\gr\op; \propoly)$.
 This restricts to a symmetric monoidal structure on the full subcategory $\f_\omega (\gr\op; \propoly)$.

Restricted to the subcategory $\f_{<\infty} (\gr\op \times \gr)\hookrightarrow \f_\omega (\gr\op; \propoly)$, this symmetric monoidal structure is equivalent to that induced by $\otimes$ on $\f (\gr\op \times \gr)$.
\end{prop}

%%%%%%%%%%%%%%%%%%%%%%%%%%%%%%%%%%%%%%%%%%%%%%%%%%%%%%%%%%%%%%%%
\subsection{Modelling bifunctors}
We combine the equivalence of Theorem \ref{thm:equiv_propoly_modcatlie} with that of Theorem \ref{thm:analytic_grop} to treat bifunctors.

\begin{thm}
\label{thm:equivalence_bimodules}
The functor $\Phi U \lie \otimes_{\cat \lie} (-) \otimes_{\cat \lie} \underline{\malcev}$ induces an equivalence of categories:
\[
\bimod[\cat \lie] \stackrel{\cong}{\rightarrow} \f_\omega (\gr\op; \propoly).
\]

This restricts to an equivalence of categories:
\[
\bimod[\cat \lie]^{<\infty}  \stackrel{\cong}{\rightarrow}  \f_{< \infty} (\gr\op \times \gr) 
\]
fitting into the commutative (up to natural isomorphism) diagram: 
\[
\xymatrix{
\bimod[\cat \lie]^{<\infty} 
\ar[r]^\cong
\ar@{^(->}[d]
&
\f_{< \infty} (\gr\op \times \gr) 
\ar@{^(->}[d]
\\
\bimod[\cat \lie]
\ar[r]_(.4)\cong 
&
\f_\omega (\gr\op; \propoly),
}
\]
in which the right hand vertical functor is as in Proposition \ref{prop:complete_bifunctors}.
\end{thm}

\begin{proof}
The category $\bimod[\cat \lie]$ is equivalent to the category of left $\cat \lie$-modules in $\rmod[\cat \lie]$ by Lemma \ref{lem:bimodules_left_right}.

One checks that  Theorem \ref{thm:analytic_grop} implies that $\Phi U \lie \otimes_{\cat \lie} -$ induces an equivalence of categories
\[
\bimod[\cat \lie] \rightarrow \f_\omega (\gr\op; \rmod [\cat \lie]).
\]
Indeed, $\Phi U \lie \otimes_{\cat \lie} -$ clearly induces an exact functor $\bimod[\cat \lie] \rightarrow \f (\gr\op; \rmod [\cat \lie])$ and this takes values in $\f_\omega (\gr\op; \rmod [\cat \lie])$. Theorem \ref{thm:analytic_grop} implies that this yields the stated equivalence.

Post-composing with the equivalence $- \otimes _{\cat \lie} \underline{\malcev} : \rmod[\cat \lie] \rightarrow \propoly$, this gives the equivalence 
\[
\bimod[\cat \lie] \stackrel{\cong}{\rightarrow} \f_\omega (\gr\op; \propoly).
\]

On restriction to $\bimod[\cat \lie]^{<\infty}$, one checks that the equivalence restricts as stated.
\end{proof}

The category $\bimod[\cat \lie]$ has the symmetric monoidal structure $(\bimod[\cat \lie], \circledcirc, \rat)$ by Proposition \ref{prop:convolution_bimodules} and $\f_\omega (\gr\op; \propoly)$ the symmetric monoidal structure of Proposition  \ref{prop:sym_mon_odot_bifunctors} (which is induced by the tensor product on $\f (\gr\op \times \gr)$). 

The equivalence of Theorem \ref{thm:equivalence_bimodules} is symmetric monoidal with respect to these:

\begin{thm}
\label{thm:bifunctors_sym_monoidal}
The functor 
\[
\Phi U \lie \otimes_{\cat \lie}( - )\otimes_{\cat \lie} \underline{\malcev} \ : \ 
\bimod[\cat \lie] \stackrel{\cong}{\rightarrow} \f_\omega (\gr\op; \propoly).
\]
is a symmetric monoidal equivalence.
\end{thm}

\begin{proof}
This follows from the fact that $\Phi U \lie \otimes_{\cat \lie} - $ is symmetric monoidal for the convolution product $\odot$ on $\lmod[\cat \lie]$ and the usual tensor product on $\f_{\omega}(\gr\op)$, by \cite{2021arXiv211001934P} and the corresponding result for $- \otimes _{\cat \lie} \underline{\malcev}$, which follows from Theorem \ref{thm:hom_underline_malcev_sym_monoidal}. 
\end{proof}

\begin{exam}
Under the equivalence of Theorem \ref{thm:equivalence_bimodules}, the bifunctors $\rat \boxtimes \A $ and $\A^{\sharp} \boxtimes \rat$  in $ \f_{< \infty} (\gr\op \times \gr)$ correspond  respectively to the $\cat \lie $-bimodules $\rat (1,0) $ and $\rat (0,1)$ given by $\rat$ concentrated in the indicated bi-arities. 

Using the convolution $\circledcirc$ for bimodules, $\rat (1,0) ^{\circledcirc s} \circledcirc \rat (0,1) ^{\circledcirc t} $ identifies as the bimodule $\rat \sym_t \boxtimes \rat \sym_s$ concentrated in bi-arity $(s,t$), with left action of $\sym_t$ and right action of $\sym_s$ (morphisms of $\cat \lie$ that do not preserve the arity necessarily act by zero).
 The associated bifunctor is $(\rat \boxtimes \A)^{\otimes s} \otimes (\A^{\sharp} \boxtimes \rat)^{\otimes t}$, which identifies as $(\A^\sharp) ^{\otimes t} \boxtimes \A^{\otimes s}$, which is isomorphic to the bifunctor obtained by applying $\Phi U \lie \otimes_{\cat \lie} (-) \otimes_{\cat \lie} \underline{\malcev}$ to $\rat \sym_t \boxtimes \rat \sym_s$.

This analysis can be extended  to treat all `homogeneous polynomial bifunctors' (i.e., those corresponding to a $\cat \lie$-bimodule supported on a single bi-arity).
\end{exam}

%%%%%%%%%%%%%%%%%%%%%%%%%%%%%%%%%%%%%%%%%%%%%%%%%%%%%%%%%%%%%%
\subsection{Comparing $\otimes_{\cat \lie}$ and  $\otimes_\gr$}

By Proposition \ref{prop:gr_bifunctors_monoidal}, $\otimes_\gr$ gives a monoidal structure on $\f (\gr \op \times \gr)$.
Analogously to the induced
\[
\ogr : \f_\omega (\gr\op) \times \propoly \rightarrow \kmod
\]
given in Proposition \ref{prop:otimes_gr_propoly_analytic}, there is an induced monoidal structure on $\f_\omega (\gr\op; \propoly)$.

\begin{prop}
\label{prop:otimes_gr_analytic_coanalytic}
The functor $\ogr$ induces a  monoidal structure  $(\f_\omega(\gr\op; \propoly), \ogr , \qgr_\bullet \pbif)$.
\end{prop}

\begin{proof}
One first checks that $\otimes_\gr : \f (\gr\op \times \gr) \times \f (\gr\op \times \gr) \rightarrow \f (\gr\op \times \gr)$ induces 
\[
\ogr : \f_\omega (\gr\op; \propoly) \times \f_\omega (\gr\op ; \propoly) \rightarrow \f_\omega (\gr\op; \propoly),
\]
as in  Proposition \ref{prop:otimes_gr_propoly_analytic} for $\ogr : \f_\omega (\gr\op) \times \propoly \rightarrow \kmod$. This is associative, by construction.

It remains to check that $\qgr _\bullet \pbif$ is the unit. This corresponds to the fact that $\pbif$ is the unit for $\otimes_\gr$ on $\f (\gr\op \times \gr)$, which follows from the behaviour exhibited in Example \ref{exam:otimes_gr}.
\end{proof}

We seek to compare this structure with the monoidal structure on $\bimod[\cat \lie]$ induced by $\otimes_{\cat \lie}$. Due to the behaviour stressed in Remark \ref{rem:order_otimes_catlie}, the {\em opposite} monoidal structure is used, defined for bimodules $M_1$, $M_2$ by 
\[
M_1 \otimes_{\cat \lie}\op M_2 := 
M_2 \otimes_{\cat \lie} M_1.
\]

\begin{thm}
\label{thm:monoidal_otimes_catlie_gr}
The functor $\Phi U \lie \otimes_{\cat \lie} (-) \otimes_{\cat \lie} \underline{\malcev} : 
\bimod[\cat \lie] \stackrel{\cong}{\rightarrow} \f_\omega (\gr\op; \propoly)$  is  monoidal for the structures $(\bimod[\cat \lie], \otimes_{\cat \lie}\op, \rat)$ and $(\f_\omega(\gr\op; \propoly), \ogr , \qgr_\bullet \pbif)$. 

In particular, for bimodules $M_1$ and $M_2$, there is a canonical isomorphism
\begin{eqnarray*}
&&\Phi U \lie \otimes_{\cat \lie} (M_1 \otimes_{\cat \lie} M_2) \otimes_{\cat \lie} \underline{\malcev}
\cong 
\\
&&
\quad \quad
(\Phi U \lie \otimes_{\cat \lie} M_2 \otimes_{\cat \lie} \underline{\malcev})
\ogr
(\Phi U \lie \otimes_{\cat \lie} M_1 \otimes_{\cat \lie} \underline{\malcev}).
\end{eqnarray*}
\end{thm}

\begin{proof}
This follows from Corollary \ref{cor:otimes_compatibility}.
\end{proof}

%\newpage
%\input{model_qgr_pbif}
 \section{Modelling the tower of categories $\qgr_\bullet \rat \gr$}
 \label{sect:model_qgr_pbif}

As a first application of the relationship between $\cat \lie$-bimodules and bifunctors, in this section we exhibit a model for the tower of categories $\qgr_\bullet \rat \gr$.
 
%%%%%%%%%%%%%%%%%%%%%%%%%%%%%%%%%%%%%%%%%%%%%%%%%%%%%%%%%%%%%%%%%%%
\subsection{The tower of categories $\cat ^{\leq d} \lie$}

For $d \in \nat$, as in Section \ref{subsect:catlie_bimodules}, let $\cat ^{\leq d} \lie$ be the quotient $\cat \lie$-bimodule obtained by applying the truncation functor $(-)^{\leq d}$ with respect to the right $\cat \lie$-module structure. 
 Explicitly, 
\[
\cat^{\leq d} \lie (s, t) = 
\left\{ 
\begin{array}{ll}
\cat \lie (s,t) & s \leq d \\
0 &\mbox{otherwise.}
\end{array}
\right. 
\] 

\begin{rem}
\label{rem:catlie_trunc_finite}
Since $\cat \lie (s,t) =0$ if $s <t$, one has $\cat^{\leq d}  \lie (s,t)=0$ if $t>d$ or if $s>d$;  it follows that the bimodule $\cat ^{\leq d} \lie$ belongs to $\bimod[\cat \lie]^{< \infty}$. 
\end{rem}

We have the following analogue of Proposition \ref{prop:pbif_bipolynomiality}, using the functor $(-)_{[\leq d]}$ on left $\cat\lie$-modules introduced in Notation \ref{nota:left_adjoint_left_cat_opd}:

\begin{prop}
\label{prop:catlie_left_adjoint}
For $d \in \nat$ the  morphism $(\cat \lie)_{[\leq d]} \rightarrow \cat^{\leq d} \lie$ induced by the surjection $\cat \lie \twoheadrightarrow \cat^{\leq d} \lie$ is an isomorphism.
\end{prop}

\begin{proof}
This can be proved by the same formal argument as used for  Proposition \ref{prop:pbif_bipolynomiality}.

It can also be proved directly, by exhibiting an inverse as follows. By definition, the canonical surjection $\cat \lie \twoheadrightarrow (\cat \lie)_{[\leq d]}$ sends $\cat \lie (s, t)$ to zero if $t >d$. In particular, if $s >d$, it sends $\cat \lie (s,s)$ to zero. It follows, by the Yoneda Lemma, that it sends $\cat \lie (s,t)$ to zero for all $s>d$ (and any $t$). This shows that this canonical surjection factorizes over $\cat \lie \twoheadrightarrow \cat^{\leq d} \lie$, whence the result.
\end{proof}

This helps motivate the following counterpart of Theorem \ref{thm:tower}. 

\begin{prop}
\label{prop:tower_cat_lie}
For $d \in\nat$, the composition in $\cat \lie$ induces a $\rat$-linear category structure on $\cat ^{\leq d} \lie$ such that 
the canonical surjection $\cat \lie \twoheadrightarrow \cat^{\leq d} \lie$ gives a $\rat$-linear functor  that is the identity on objects and the canonical surjection on morphisms. 

Moreover, these form a tower of $\rat$-linear categories under $\cat \lie$:
\[
\xymatrix{
\cat \lie
\ar@{.>}[d]
\ar[rd]
\ar[rrd]
\\
\ar@{.>}[r]
&
\cat^{\leq d+1}  \lie
\ar[r]
&
\cat^{\leq d}  \lie 
\ar[r]
&\ldots 
}
\]
\end{prop}

\begin{proof}
This can be proved by the reasoning applied for Theorem \ref{thm:tower}. It can also be seen directly, as follows.

The key  is that $\cat^{\leq d}\lie (s,t) =0$ if $t > d$ (this is encoded in Proposition \ref{prop:catlie_left_adjoint}). Thus non-trivial compositions are understood by restricting to $s\leq d $ and $t \leq d$, where it is given by that of $\cat \lie$.
\end{proof}

%%%%%%%%%%%%%%%%%%%%%%%%%%%%%%%%%%%%%%%%%%%%%%%%%%%%%%
\subsection{Passage to bifunctors}

By Remark \ref{rem:catlie_trunc_finite}, for $d \in \nat$, the bimodule $\cat^{\leq d}\lie$ belongs to $\bimod[\cat \lie]^{< \infty}$. It follows that applying the functor $\Phi U \lie \otimes_{\cat \lie} (-) \otimes_{\cat \lie} \underline{\malcev}$ of Theorem \ref{thm:equivalence_bimodules} yields a  bifunctor in $\f_{< \infty} (\gr\op \times \gr)$. One has the key identification:

\begin{prop}
\label{prop:catlie_trunc_d}
For $d \in \nat$, the bifunctor $\Phi U \lie \otimes_{\cat \lie} \cat ^{\leq d} \lie \otimes_{\cat \lie} \underline{\malcev}$
is isomorphic to $\qgr_d \pbif$.
\end{prop}

\begin{proof}
By Proposition \ref{prop:properties_otimes_cat_opd},
\[
\Phi U \lie \otimes_{\cat \lie} \cat ^{\leq d} \lie 
\cong 
(\Phi U \lie)^{\leq d} \otimes _{\cat \lie} \cat ^{\leq d} \lie 
\]
since $\cat^{\leq d} \lie (s,t)=0$ for $t>d$.  

Now, Proposition \ref{prop:catlie_left_adjoint}  gives the isomorphism $(\cat \lie)_{[\leq d]}\cong \cat^{\leq d}\lie$. Hence, 
Proposition \ref{prop:properties_otimes_cat_opd} implies that the surjection $\cat \lie \twoheadrightarrow \cat^{\leq d} \lie$ induces an isomorphism
\[
(\Phi U \lie)^{\leq d} \otimes _{\cat \lie} \cat  \lie
\stackrel{\cong}{\rightarrow} 
(\Phi U \lie)^{\leq d} \otimes _{\cat \lie} \cat ^{\leq d} \lie. 
\]

Putting these together gives the isomorphisms:
\[
\Phi U \lie \otimes_{\cat \lie} \cat ^{\leq d} \lie 
\cong 
(\Phi U \lie)^{\leq d} \otimes _{\cat \lie} \cat  \lie
\cong 
(\Phi U \lie)^{\leq d} 
.
\]
Hence $\Phi U \lie \otimes_{\cat \lie} \cat ^{\leq d} \lie \otimes_{\cat \lie} \underline{\malcev}$ is isomorphic to
$
(\Phi U \lie)^{\leq d} \otimes _{\cat \lie} \malcev.
$

By Theorem \ref{thm:equiv_propoly_modcatlie}, $- \otimes_{\cat \lie}  \underline{\malcev}$ induces an equivalence of categories 
$\rmod[\cat \lie] \stackrel{\cong}{\rightarrow} \propoly$.
 Under this equivalence, $\Phi U \lie$ corresponds to $\Phi U \malcev$, arguing as in Theorem \ref{thm:equiv_Umalcev_PhiUlie}, and the latter is isomorphic to $\qgr_\bullet \pbif$, by the Theorem.  Proposition \ref{prop:compare_filtrations} shows that applying the truncation $(-)^{\leq d}$ corresponds to  forming the quotient $\qgr_\bullet  \pbif \twoheadrightarrow \qgr_\bullet ( \qgr_d \pbif)$. The result follows. 
\end{proof}

Composition in $\cat ^{\leq d } \lie$ gives the morphism of $\cat \lie$-bimodules:
\[
\cat ^{\leq d} \lie \otimes_{\cat \lie} \cat ^{\leq d} \lie \rightarrow \cat ^{\leq d} \lie.
\]
This is unital and associative in the obvious sense. (More precisely, by Proposition \ref{prop:properties_otimes_cat_opd}, it is an isomorphism, corresponding to the fact that $\cat ^{\leq d} \lie$ is the unit for the restriction of $\otimes_{\cat \lie}$ to the full subcategory of bimodules that vanish whenever one of the arities is greater than $d$.)

\begin{thm}
\label{thm:tower_isomorphism}
For $d\in \nat$, 
under the equivalence of categories of Theorem \ref{thm:equivalence_bimodules}, the category $\cat^{\leq d}\lie$ corresponds to 
$\big( \qgr_d \rat \gr \big)\op$.

Under this identification, the  $\rat$-linear functor $\cat ^{\leq d+1} \lie \rightarrow \cat^{\leq d} \lie$ of Proposition \ref{prop:tower_cat_lie} corresponds to the functor
$
\big( \qgr_{d+1} \rat \gr \big)\op
\rightarrow 
\big( \qgr_d \rat \gr \big)\op
$  
of Theorem \ref{thm:tower}. 
\end{thm}

\begin{proof}
This follows from Proposition \ref{prop:catlie_trunc_d} combined with Theorem \ref{thm:monoidal_otimes_catlie_gr}, paying attention to the variance. 
\end{proof}

%\newpage
%\input{modular}
\section{The Casimir PROP associated to $\lie$ viewed as a cyclic operad}
\label{sect:modular}

The purpose of this section is to extend the analysis of the tower of categories $\qgr_\bullet \rat \gr$ in terms of the tower $\cat ^{\leq \bullet } \lie$ to the case of the modular (or Casimir) PROP associated to the Lie operad. 

This is inspired by the work of Habiro and Massuyeau \cite{MR4321214}, who introduced a category $\mathbf{A}$ (constructed from Jacobi diagrams in handlebodies) in their study of the Kontsevich integral for bottom tangles in handlebodies. The category $\mathbf{A}$ can be defined over $\rat$; it is $\nat$-graded and $\mathbf{A}_0$ is equivalent to $\rat \gr\op$. The positive part of $\mathbf{A}$ arises due to the {\em Casimir} structure. 

Part of the structure of $\mathbf{A}$ has been analysed by Katada \cite{2021arXiv210206382K};  she considers the structure of $\mathbf{A}(\mathbf{0}, -)$, restricting to the action of automorphism groups of finite rank free groups. Vespa \cite{2022arXiv220210907V}  has analysed part of the structure of $\mathbf{A}$ by using {\em beaded} Jacobi diagrams.

%%%%%%%%%%%%%%%%%%%%%%%%%%%%%%%%%%%%%%%%%%%%%%%%%%%%%%%%%%%%%%%%%%%%%%%%%%%%%%%%%%%%%
\subsection{From $\lie$ to the modular operad $\jac$ and the PROP $\pjac$}

The material of this subsection is inspired (and covered) by the work of Hinich and Vaintrob \cite{MR1913297}; the reader is referred to their paper for further details. 

The starting point is the fact that the operad $\lie$ underlies a  cyclic operad, $\cyclie$. In particular,  the cyclic operad $\cyclie$ has  underlying $\rat\fb\op$-module such that $\cyclie (\mathbf{0}) = \cyclie (\mathbf{1})=0$ and $ \cyclie (\mathbf{n+1}) \downarrow ^{\sym_{n+1}}_{\sym_n}\cong \lie (\mathbf{n}) $ as a $\sym_n$-module. 

One can go further and consider the associated {\em modular operad} $\jac$ and then the PROP $\pjac$ constructed from $\jac$;  these are constructed from {\em (open) Jacobi diagrams}.

Recall that a Jacobi diagram is a vertex-oriented, uni-trivalent graph (we will always suppose that the graph is finite); it is {\em open} if each connected component has at least one univalent vertex (see \cite[Chapter 5]{MR2962302} for details on this and other standard material, such as the AS and IHX relations). The vertex orientation is a cyclic order of  the half edges at a trivalent vertex; equivalently, this can be specified by giving a planar embedding, so that one has the following two possibilities at a trivalent vertex:

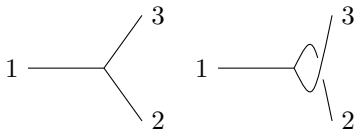
\begin{tikzpicture}[scale = .5]
 \draw (-1,0) -- (1,0);
 \node [left] at (-1,0) {$1$};
 \draw (1,0) -- (2, -1.4);
 \node [right] at (2,-1.4) {$2$}; 
\draw  (1,0) -- (2, 1.4);
\node [right] at (2,1.4) {$3$}; 
 \draw (4,0) -- (6,0);
  \node [left] at (4,0) {$1$};
 \draw (6,0) .. controls (6.5,1).. (7, -1.4);
  \node [right] at (7,-1.4) {$2$}; 
  \fill[white] (6.75,0) circle(.3);
\draw  (6,0) .. controls (6.5, -1).. (7, 1.4);
\node [right] at (7,1.4) {$3$}; 
\end{tikzpicture}

\noindent
indexing the half edges by $\{1, 2, 3 \}$ and taking the convention that the trivalent vertex inherits the standard orientation from the plane.

\begin{enumerate}
\item
For a (non-empty)  finite set $X$, $\jac (X)$ is the $\rat$-vector space generated by connected Jacobi diagrams equipped with a bijection  between the univalent vertices and $X$, modulo the following relations:
\begin{enumerate}
\item 
a generator corresponding to a closed (i.e., not open) Jacobi diagram is equivalent to zero; 
\item 
 the AS (antisymmetry) and IHX (Jacobi) relations.
 \end{enumerate}
The automorphisms $\mathrm{Aut} (X)$ act by relabelling. 
\item 
The underlying cyclic operad has composition operation defined as follows: for finite sets $X$, $Y$ and elements $x \in X$, $y \in Y$, the composition $\circ_{x,y} : \jac (X) \otimes \jac (Y)\rightarrow \jac ((X \amalg Y) \backslash \{x, y\})$ is induced by the glueing operation on uni-trivalent graphs that glues together the univalent vertices labelled by $x$ and $y$ respectively.
\item 
The contraction operation for the modular structure is defined similarly: for a  subset $\{x_1, x_2 \}\subset X$, the contraction $c_{x_1, x_2} : \jac (X) \rightarrow \jac (X \backslash \{ x_1, x_2 \}) $ is given by glueing together the univalent vertices labelled by $x_1$ and $x_2$ respectively.
\end{enumerate}

\begin{rem}
\
\begin{enumerate}
\item 
The composition and contraction operations on uni-trivalent graphs may give rise to {\em closed} diagrams; these are then identified to zero. 
\item 
There is an inclusion of cyclic operads $\cyclie \hookrightarrow \jac$ (forgetting the modular structure of $\jac$); the image corresponds to the sub cyclic operad represented by trivalent planar trees (i.e. connected, vertex-oriented uni-trivalent graphs that are simply-connected). 
\end{enumerate}
\end{rem}   

The {\em degree} of a uni-trivalent graph is defined as $\frac{1}{2} (\sharp \mathrm{vertices})$, where $\sharp \mathrm{vertices}$ is the total number of vertices; the degree  is a positive integer. This induces an $\nat_+$-grading of the underlying $\fb\op$-module of $\jac$ (however, composition is not additive with respect to this grading). 

The PROP  $\pjac$ is constructed from $\jac$ as follows, requiring  that each term from $\jac$ contributes at least one exit:
\begin{enumerate}
\item 
For finite sets $X$, $Y$:
\[
\pjac (X,Y) = \bigoplus_{\substack{X = \amalg_{i \in \mathcal{I}}X_i \\ Y = \amalg_{i \in \mathcal{I}} Y_i \\ Y_i \neq \emptyset}} \bigotimes _{i \in \mathcal{I}} \jac (X_i \amalg Y_i),
\]
where the sum is over all possible pairs of decompositions (indexed by the same set $\mathcal{I}$) $(X = \amalg_{i \in \mathcal{I}}X_i , Y = \amalg_{i \in \mathcal{I}} Y)$, subject to the condition that each $Y_i$ is non-empty. In particular, $\pjac (X, \emptyset) =0$ if $X\neq \emptyset$, whereas $\pjac (\emptyset, \emptyset) =\rat$, by convention.
\item 
Composition is induced by the modular operad structure of $\jac$, giving 
\[
\pjac (Y, Z) \otimes_{\mathrm{Aut} (Y)} \pjac (X, Y) \rightarrow \pjac (X, Z).
\]
\end{enumerate}

\begin{rem}
\ 
\begin{enumerate}
\item 
One can take $X= \emptyset$ in $\pjac (X,Y)$; there is a non-trivial composition from the left  $$\pjac (Y, Z) \otimes_{\mathrm{Aut} (Y)} \pjac (\emptyset, Y) \rightarrow \pjac (\emptyset, Z).$$
\item 
There is an inclusion of PROPs $\cat \lie \hookrightarrow \pjac$ induced by the inclusion of cyclic operads $\cyclie \hookrightarrow \jac$: namely, $\cat \lie(X,Y)$ identifies as the following subspace of $\pjac (X,Y)$: 
\[
\bigoplus_{\substack{X = \amalg_{i \in \mathcal{I}}X_i \\ Y = \amalg_{i \in \mathcal{I}} Y_i \\ |Y_i|=1 }} \bigotimes _{i \in \mathcal{I}} \cyclie (X_i \amalg Y_i)
\subset 
\bigoplus_{\substack{X = \amalg_{i \in \mathcal{I}}X_i \\ Y = \amalg_{i \in \mathcal{I}} Y_i \\ Y_i \neq \emptyset}} \bigotimes _{i \in \mathcal{I}} \jac (X_i \amalg Y_i),
\]
(note the condition $|Y_i|=1$ on the left hand side).
\end{enumerate}
\end{rem}

\begin{rem}
Hinich and Vaintrob \cite{MR1913297} place the above construction in a general framework.  Namely they exhibit a functor $(-)^C $ from cyclic operads to PROPs that sends the cyclic operad $\opd$ to the PROP $\opd^C$ for Casimir $\opd$-algebras. Applied to the cyclic operad $\cyclie$, this yields $\pjac$.

Recall (see \cite[Section 7]{MR4321214} or \cite{MR1913297}) that a Casimir Lie algebra is a Lie algebra $\g$ equipped with a Casimir element, i.e.,  $c \in \Gamma^2 (\g)\subset \g^{\otimes 2}$ that is $\mathrm{ad}$-invariant (explicitly, writing $c = c_i \otimes c'_i$ with implicit summation, then $c_i \otimes c'_i= c'_i \otimes c_i$ and, for any $x \in \g$, $[x, c_i] \otimes c'_i + c_i \otimes [x, c'_i]=0$).

Casimir Lie algebras form a category with morphisms those morphisms of Lie algebras that are compatible with the Casimir elements. There is an obvious forgetful functor from Casimir Lie algebras to Lie algebras. Hinich and Vaintrob observed (see \cite[Lemma 3.1.7]{MR1913297}) that the category of Casimir Lie algebras is precisely the category of algebras over the PROP $\pjac$; the Casimir element is encoded by the action of the element of $\pjac (\mathbf{0} , \mathbf{2})$ that is given by the canonical generator of $\cyclie (\mathbf{2})$. 
\end{rem}

\begin{rem}
Habiro and Massuyeau observed in \cite[Section 7]{MR4321214} that the universal enveloping algebra $U \g$ of a Casimir Lie algebra is a (cocommutative) Casimir Hopf algebra, as defined in \cite[Section 5]{MR4321214}. This is at the heart of the relationship between $\pjac$ and the category $\mathbf{A}$ introduced in \cite[Section 4]{MR4321214} and described in terms of Casimir Hopf algebras in \cite[Theorem 5.11]{MR4321214}. This is also related to the results of Hinich and Vaintrob, notably \cite[Theorem 7.1.1]{MR1913297}.
\end{rem}

%%%%%%%%%%%%%%%%%%%%%%%%%%%%%%%%%%%%%%%%%%%%%%%%%%%%%%%%%%%%%%%%
\subsection{First properties of $\pjac$}

Using the explicit construction of $\pjac$, one deduces  the following, which gives rise to a $\catlie$-bimodule structure.

\begin{prop}
\label{prop:nat_grading_ojac}
The PROP $\pjac$ is $\nat$-graded, i.e., $\pjac(-,-)= \bigoplus_{n \in \nat} \pjac _n (-,-)$ and this grading  is compatible with the PROP structure. This grading is determined by placing an element of $\jac (X \amalg Y) $ (with $Y \neq \emptyset$),  represented by a connected, vertex-oriented uni-trivalent graph of degree $d$ in grading $d - |X|$. 

Moreover, $\pjac_0$ identifies with $\cat \lie$ as a PROP. Hence, for each $n \in \nat$, $\pjac_n$ has the structure of a $\cat \lie$-bimodule.
\end{prop}

\begin{proof}
Using the PROP structure, it is clear that specifying the grading on terms of $\jac (X \amalg Y) $ determines the grading, by using the `horizontal' composition of the PROP given by the symmetric monoidal structure of the PROP.

One first checks that the given grading yields an $\nat$-grading (i.e., there are no terms in negative grading). The identification of the degree zero part is proved by induction on the number of trivalent vertices. 

A direct verification shows that  the grading is compatible with the `vertical' composition; indeed, the correction of the degree by subtracting the number of incoming univalent vertices is designed to ensure the appropriate additivity. 

The above analysis also yields the identification of $\pjac_0$ with $\cat \lie$ as a PROP. The composition of $\pjac$ therefore yields a $\cat \lie$-bimodule structure. 
\end{proof}

\begin{rem}
If the condition $Y_i \neq \emptyset$ were not imposed, one would not obtain an $\nat$-grading above. For instance, considering a generator of $\cyclie (\mathbf{3})$ (represented by the connected uni-trivalent graph with one trivalent vertex), as having three entries and no exits, the degree is $2$ so the grading would be $-1$. 
\end{rem}

The grading gives the following useful finiteness property: 

\begin{lem}
\label{lem:finiteness_pjac}
For $n, s,t \in \nat$: 
\begin{enumerate}
\item 
$\pjac_n (\mathbf{s}, \mathbf{t})$ has finite dimension; 
\item 
$\pjac_n (\mathbf{s}, \mathbf{t})=0$ if $t =0$ or $t > 2n+s$. 
\end{enumerate}
In particular $\sum_{t \in \nat} \dim_\rat \pjac_n (\mathbf{s}, \mathbf{t}) <\infty$. 
\end{lem}

\begin{proof}
For fixed degree $d$, there are only finitely-many uni-trivalent graphs of degree $d$. The finite-dimensionality of $\pjac_n (\mathbf{s}, \mathbf{t})$ follows. 

For the second statement, elements of $\pjac_n (\mathbf{s}, \mathbf{t})$ are represented by linear combinations of uni-trivalent graphs with $2(n+s)$ vertices, hence with at most $2n+s$ exit vertices. 

Putting these results together gives the total finite-dimensionality.  
\end{proof}

\begin{exam}
\label{exam:katada}
By Proposition, \ref{prop:nat_grading_ojac}, for any $n \in \nat$, $\pjac_n (\mathbf{0}, -)$ has the structure of a left $\cat \lie$-module. Moreover, the grading $n$ coincides with the  degree. One has $\pjac_n (\mathbf{0}, \mathbf{t})=0$ for $t > 2n$  and $\pjac_n (\mathbf{0}, \mathbf{t})$ is always finite-dimensional. 

The associated polynomial functor $J_n := \Phi U \lie \otimes_{\cat \lie} \pjac_n (\mathbf{0}, -)$ provided by the equivalence of Theorem \ref{thm:analytic_grop} lies in $\f_{2n} (\gr\op)$ and has a finite composition series.  This underlies the structures studied by Katada in \cite{2021arXiv210206382K}. (Katada does not work with functors on $\gr\op$, but restricts to the underlying representations of automorphism groups of free groups.)
\end{exam}

The symmetric monoidal structure of the PROP provides the `horizontal' compositions: 
$$
\pjac_m (\mathbf{s}, \mathbf{t}) \otimes \pjac_n (\mathbf{u}, \mathbf{v}) \rightarrow \pjac_{m+n} (\mathbf{s}\amalg \mathbf{u}, \mathbf{t}\amalg \mathbf{v}).
$$
These assemble to form a map of $\rat\fb\op \otimes \rat\fb$-modules, using the convolution product for bimodules:
\[
\pjac_m \circledcirc \pjac_n \rightarrow \pjac_{m+n}.
\]
This is associative and unital in the evident sense (noting that $\pjac (\emptyset, \emptyset) = \pjac_0 (\emptyset, \emptyset) =\rat$).

For each $t \in \nat$, we can consider $\pjac_t$ as a $\cat \lie$-bimodule. Hence,
using the convolution product structure of Proposition \ref{prop:convolution_bimodules}, one can consider $\pjac_m \circledcirc \pjac_n $ as a $\cat \lie$-bimodule.

\begin{prop}
\label{prop:horizontal}
For $m,n \in \nat$, the horizontal composition $\pjac_m \circledcirc \pjac_n \rightarrow \pjac_{m+n}$ is a morphism of right $\cat \lie$-modules. If both $m$ and $n$ are positive, it is not a morphism of left $\cat \lie$-modules.
\end{prop}

\begin{proof}
 One can restrict to considering partial compositions, and use the identification $\lie (s) = \cat \lie (s,1)$,  which underlines the fact that there is a single `exit'. Hence, precomposing an element in the image of the horizontal composition with this corresponds either to precomposition in $\pjac_m$ or precomposition in $\pjac_n$. Unwinding the definitions, this implies that the horizontal composition respects the  right $\cat \lie$-module structure. 

This argument breaks down for the  left $\cat \lie$-module structure: if $m>0$ and $n>0$, a  post-composition with $\cat \lie (2,1)$ can always be found that mixes the contributions from $\pjac_m$ and $\pjac_n$.
\end{proof}

%%%%%%%%%%%%%%%%%%%%%%%%%%%%%%%%%%%%%%%%%%%%%%%%%%%%%%%%%
\subsection{Generating $\pjac $ as a $\cat \lie$-bimodule}

Unfortunately, a full understanding of the $\cat \lie$-bimodule structure of $\pjac$ (or even the underlying $\rat\fb$-bimodule) is out of reach; this stems from the fact that the $\sym_n$-representations $\jac (\mathbf{n})$, for $n \in \nat$, are not known for large $n$. 

Nevertheless, there is an explicit $\rat\fb$-bimodule that generates $\pjac$ as a $\cat \lie$-bimodule. This is based on the following, which is immediate from the definitions:

\begin{lem}
\label{lem:no_trivalent}
\
\begin{enumerate}
\item 
$\pjac_0 (\mathbf{1}, \mathbf{1}) \cong \rat$ (with trivial action of $\sym_1\op  \times  \sym_1$), with generator $\iota$ corresponding to the identity of $\lie (\mathbf{1})= \cyclie(\mathbf{2})$;
\item 
$\pjac_1 (\mathbf{0}, \mathbf{2}) \cong \rat$, with trivial action of $\sym_0\op \times \sym_2$, with generator $c$ corresponding to the Casimir element.
\end{enumerate}

Moreover, $c$ and $\iota$ generate the subspace of $\pjac$ generated by connected open Jacobi diagrams with no trivalent vertex.
\end{lem}

\begin{proof}
There is a unique connected open Jacobi diagram with no trivalent vertex. After distinguishing the exit vertices (denoted by $\circ$) and the entrance vertices (denoted $\bullet$), since there must be at least one $\circ$ vertex, there are the following two possibilities
\begin{enumerate}
\item 
\begin{tikzpicture}[scale = .25]
\draw  (0,0) -- (0,2); 
\draw [fill=white,thick] (0,0) circle[radius = .2];
\draw [fill=black,thick] (0,2) circle[radius = .2];
 \end{tikzpicture}
  representing the generator $\iota$; 
\item 
\begin{tikzpicture}[scale = .25]
\draw  (-1,0) .. controls (0, 1).. (1, 0);
\draw [fill=white,thick] (1,0) circle[radius = .2];
\draw [fill=white,thick] (-1,0) circle[radius = .2];
 \end{tikzpicture}
 representing the generator $c$. 
 \end{enumerate}
 (We use the convention that entries are at the top and exits the bottom.)
\end{proof}

The symmetric monoidal structure of $\pjac$ allows one to generate a sub $\rat\fb$-bimodule of $\pjac$ from the elements $c$ and $\iota$. This can be expressed by using the convolution product $\circledcirc$ for $\rat\fb$-bimodules; here a more concrete description is given: 

\begin{defn}
\label{defn:chord}
Let $\chord \subset \pjac$ be the sub $\fb$-bimodule  generated by open Jacobi diagrams with no trivalent vertices
 and let $\chord_n \subset \pjac_n$ be the subspace of grading $n$.
 
 By convention, $\chord (\mathbf{0}, \mathbf{0}) = \pjac  (\mathbf{0}, \mathbf{0}) = \rat$, in grading $0$.
\end{defn}

\begin{rem}
As suggested by the notation, $\chord$ is related to {\em open} chord diagrams appearing in studying finite type invariants (see \cite[Chapter 4]{MR2962302}, for example).
\end{rem}

Via the symmetric monoidal structure of $\pjac$, $\chord$ is generated by $c$ and $\iota$ of Lemma \ref{lem:no_trivalent}. In particular one has:

\begin{lem}
\label{lem:chord_n}
For $n,s  \in \nat$,
\begin{enumerate}
\item 
 $\chord_n (\mathbf{s}, \mathbf{t})=0$ unless $t = 2n +s$; 
\item 
the vector space $\chord_n (\mathbf{s}, \mathbf{2n+s})$ has basis given by pairs given by an injection $\alpha : \mathbf{s} \hookrightarrow \mathbf{2n+s}$ together with a decomposition of $\mathbf{2n+s} \backslash \alpha (\mathbf{s})$ into $n$ subsets of cardinal two. 
\item
The group $\sym_s \op\times \sym_{2n+s}$ acts in the obvious way on this basis. 
\end{enumerate}
In particular, $\chord_n (\mathbf{s}, \mathbf{2n+s})$ is a permutation representation of $\sym_s \op\times \sym_{2n+s}$ .
\end{lem} 

\begin{proof}
This follows by identifying the uni-trivalent graphs with no trivalent vertex and applying the condition that each connected component must have at least one exit vertex (see the following example). The associated permutation representation can be written down explicitly. 
\end{proof}

The $\fb$-bimodule $\chord$ gives generators for $\pjac$, as follows. (This result is implicit in the work of Habiro and Massuyeau \cite{MR4321214}.)

\begin{prop}
\label{prop:chord_generates}
For $n \in \nat$, $\pjac_n$ is generated as a $\cat \lie$-bimodule by $\chord_n$, i.e., there is a surjection of $\cat \lie$-bimodules:
\[
\cat \lie \otimes_\fb \chord_n \otimes_\fb \cat \lie \twoheadrightarrow \pjac_n. 
\]
\end{prop}

\begin{proof}
This is proved by a straightforward induction on the number of vertices. 
\end{proof}

\begin{exam}
The surjection of Proposition \ref{prop:chord_generates} gives, in particular, the surjection 
\[
\cat \lie (2n, -) 
\otimes_{\sym_{2n}} 
\chord_n (\mathbf{0}, \mathbf{2n}) 
\twoheadrightarrow 
\pjac _n (\mathbf{0}, -).
\]
This has been used by Katada (with a very different formulation) in \cite{2021arXiv210206382K}, based on the results of \cite{MR4321214}.
\end{exam}

\begin{rem}
The surjection of Proposition \ref{prop:chord_generates} is far from being an isomorphism. 
\begin{enumerate}
\item 
For $n=0$, $\chord_0$ identifies as the $\fb$-bimodule with $\chord_0 (\mathbf{s}, \mathbf{s}) \cong \rat \sym_s$. The surjection is $\cat \lie \otimes_\fb \cat \lie \twoheadrightarrow \cat \lie$ induced by composition. 
\item 
Moreover, for $n>0$, the surjection does not taken into account: 
\begin{enumerate}
\item 
the 4T-relation for chord diagrams (see \cite{MR2962302}, for example); 
\item 
analogous commutation relations between $c$ and $\iota$ (using the notation of Lemma \ref{lem:no_trivalent}).
\end{enumerate}
\end{enumerate}
\end{rem}

%%%%%%%%%%%%%%%%%%%%%%%%%%%%%%%%%%%%%%%%%%%%%%%%%%%%%%%%%
\subsection{Passage to bifunctors}
The general theory of Section \ref{sect:bifunctors} applies, generalizing the considerations of Section \ref{sect:model_qgr_pbif}, where $\cat \lie = \pjac_0$ was treated.

\begin{nota}
\label{nota:ajac}
For $n \in \nat$, let $\ajac(n)$ denote the image of $\pjac_n$ in $\f_\omega (\gr\op; \propoly)$ under the equivalence of Theorem  \ref{thm:bifunctors_sym_monoidal}, 
$
\ajac (n):= \Phi U \lie \otimes _{\cat \lie} \pjac_n \otimes _{\cat \lie} \underline{\malcev}
$.
\end{nota}

Proposition \ref{prop:chord_generates} immediately gives the following, which provides some intuition when considering $\ajac(*)$: 

\begin{prop}
\label{prop:generating_ajac_chord}
For $n \in \nat$, there is a surjection in $\f_\omega (\gr\op; \propoly)$:
\[
\Phi U \lie \otimes _{\fb} \chord_n \otimes _\fb \underline{\malcev}
\twoheadrightarrow 
\ajac (n).
\]
\end{prop}

By construction, $\ajac(n)$ is given by the tower of bifunctors:
\begin{eqnarray}
\label{eqn:tower_ajac}
\ldots \twoheadrightarrow \ajac(n)_d \twoheadrightarrow \ajac(n)_{d-1} \twoheadrightarrow \ldots \ajac(n)_1 \twoheadrightarrow 
\ajac(n)_0,
\end{eqnarray}
where $\ajac(n)_d$ is polynomial of degree $d$ with respect to $\gr$. 

\begin{exam}
The tower $\ajac(0)_\bullet$ coincides with the tower $\qgr_\bullet \pbif$. 
\end{exam}

Lemma \ref{lem:finiteness_pjac} implies the following: 

\begin{prop}
\label{prop:properties_ajac}
For $d, n \in \nat$, $\ajac(n)_d$ lies in $\f_{< \infty} (\gr\op \times \gr)$ and has a finite composition series. 
 More precisely, $\ajac(n)_d$ has polynomial degree $d$ with respect to $\gr$ and polynomial degree $\leq 2n+ d$ with respect to $\gr\op$.

Hence,  (\ref{eqn:tower_ajac}) is a tower of surjections between finite objects in $\f_{< \infty} (\gr\op \times \gr)$. 
\end{prop}

The (vertical) composition in the PROP $\pjac_*$ induces composition operations between the bifunctors $\ajac(*)$. This requires using the symmetric monoidal structure $\ogr$ on $\f_\omega (\gr\op; \propoly)$ given in Proposition \ref{prop:otimes_gr_analytic_coanalytic}.

\begin{lem}
\label{lem:composition_ajac}
For $m,n \in \nat$, the composition operation in $\cat\lie$-bimodules
\[
\pjac_m \otimes_{\cat\lie} \pjac_n \rightarrow \pjac_{m+n}
\]
induces the natural transformation 
$
\ajac(n) \ogr \ajac (m) \rightarrow \ajac (m+n)
$
in $\f_\omega (\gr\op;\propoly)$. 

Via the natural projections in the respective towers, in polynomial degree $d$ with respect to $\gr$, this is determined by the morphism of polynomial bifunctors:
\[
\ajac(n)_d \otimes_\gr \ajac (m)_{2n+d} \rightarrow \ajac (m+n)_d. 
\]
\end{lem}

\begin{proof}
The first statement follows  from Theorem \ref{thm:monoidal_otimes_catlie_gr}. The second statement follows from Proposition \ref{prop:otimes_gr_poly} by using the fact that $\ajac(n)_d$ is polynomial of degree $2n+d$ with respect to $\gr\op$, so that there is a natural isomorphism
$ 
\ajac(n)_d \otimes_\gr \ajac (m) \cong \ajac(n)_d \otimes_\gr \ajac (m)_{2n+d}$.
\end{proof}

\begin{rem}
Contrary to the behaviour exhibited in Theorem \ref{thm:tower_isomorphism}, composition on $\ajac(*)$ does not correspond to a tower of monoid structures on $\ajac(*)_d$, for $d \in \nat$. This justifies the additional work in Section \ref{sect:bifunctors} in developing the framework based upon $\f_\omega (\gr\op; \propoly)$.
\end{rem}

Putting these facts together, one has:

\begin{thm}
\label{thm:ajac}
The (pro)bifunctors $\ajac(*)$ have the structure of a unital, $\nat$-graded associative monoid in 
the monoidal category 
$(\f_\omega (\gr\op; \propoly), \ogr, \qgr_\bullet P_\zed)$, with $\ajac(0)$ isomorphic to $(\qgr_\bullet \rat \gr)\op$ considered as a monoid.
\end{thm}

%%%%%%%%%%%%%%%%%%%%%%%%%%%%%%%%%%%%%%%%%%%%%%%%%%%%%%%%%%%%%%%%%%%%%
\subsection{Relating to $\mathbf{A}$}
\label{subsect:sketchy}

This section sketches how  $\ajac(*)$ is related to Habiro and Massuyeau's category $\mathbf{A}$ \cite{MR4321214}.

By definition, $\ob \mathbf{A}= \nat$  and $\mathbf{A}$ is $\nat$-graded. It is equipped with an embedding $\rat \gr\op \hookrightarrow \mathbf{A}$ that sends the free group $\zed^{\star r}$ to the object $r \in \nat$ and induces an equivalence $\rat \gr\op \cong \mathbf{A}_0$. In particular, for each grading $n \in \nat$, $\mathbf{A}_n(-,-)$ is a bifunctor in $\f (\gr\op \times \gr) \cong \f (\gr\op ; \f (\gr))$.

Post-composing with $\qgr_\bullet$ gives 
$
\qgr_\bullet \mathbf{A}_n (-, -) \in \f (\gr\op; \propoly).
$ 
More precisely, one has the analogue of Proposition \ref{prop:properties_ajac}, namely that $\qgr_d \mathbf{A}_n$ has polynomial degree at most $2n +d$ with respect to $\gr\op$. It follows in particular that 
\[
\qgr_\bullet \mathbf{A}_n (-, -) \in \f_\omega (\gr\op; \propoly).
\]
Moreover, the composition of $\mathbf{A}$ induces an associative composition operation in $\f _\omega (\gr\op; \propoly)$:
\[
\qgr_\bullet \mathbf{A} \ogr \qgr_\bullet \mathbf{A} 
\rightarrow 
\qgr_\bullet \mathbf{A}
\] 
making $\qgr_\bullet \mathbf{A}$ into an associative monoid with unit $\qgr_\bullet \mathbf{A}_0 \cong (\qgr_\bullet \rat \gr)\op$.

\begin{thm}
\label{thm:ajac_versus_A}
There is an isomorphism of $\nat$-graded associative monoids in 
$(\f_\omega (\gr\op; \propoly), \ogr, \qgr_\bullet P_\zed)$:
\[
\ajac (*) \cong \qgr_\bullet \mathbf{A}.
\]
\end{thm}

\begin{proof}
(Indications.) 
The proof that there is a natural isomorphism $\ajac(n) \cong \qgr_\bullet \mathbf{A}_n$ in  $\f_\omega (\gr\op; \propoly)$ uses the methods employed in the proof of Theorem \ref{thm:tower_isomorphism}. (A partial result in this sense is given in \cite[Theorem 6.7]{2022arXiv220210907V}.)

A careful analysis of the definition of $\mathbf{A}$ in \cite[Section 4]{MR4321214} shows that these natural isomorphisms respect the respective monoid structures.  (Heuristically, this follows from the presentation of the category $\mathbf{A}$ that is given in \cite[Theorem 5.11]{MR4321214} and is related to \cite[Theorem 7.1.1]{MR1913297}.)
\end{proof}

\begin{rem}
Using the constructions of this paper, the monoid $\ajac (*)$ arises directly from the Casimir PROP $\pjac$ constructed from the Lie operad. The significance of Theorem \ref{thm:ajac_versus_A} is that it shows that the objects  $\ajac (n)$ arise from honest bifunctors of $\f (\gr\op \times \gr)$, rather than just $\f_\omega (\gr\op; \propoly)$.
\end{rem}

%\newpage
%\nocite{*}
%\bibliographystyle{amsalpha}
%\bibliography{passi.bib}

\begin{thebibliography}{{Pow}21}

\bibitem[CDM12]{MR2962302}
S.~Chmutov, S.~Duzhin, and J.~Mostovoy, \emph{Introduction to {V}assiliev knot
  invariants}, Cambridge University Press, Cambridge, 2012. \MR{2962302}

\bibitem[DPV16]{MR3505136}
Aur\'{e}lien Djament, Teimuraz Pirashvili, and Christine Vespa,
  \emph{Cohomologie des foncteurs polynomiaux sur les groupes libres}, Doc.
  Math. \textbf{21} (2016), 205--222. \MR{3505136}

\bibitem[EML54]{MR65162}
Samuel Eilenberg and Saunders Mac~Lane, \emph{On the groups {$H(\Pi,n)$}. {II}.
  {M}ethods of computation}, Ann. of Math. (2) \textbf{60} (1954), 49--139.
  \MR{65162}

\bibitem[Fre98]{MR1617616}
Benoit Fresse, \emph{Lie theory of formal groups over an operad}, J. Algebra
  \textbf{202} (1998), no.~2, 455--511. \MR{1617616}

\bibitem[Fre09]{MR2494775}
\bysame, \emph{Modules over operads and functors}, Lecture Notes in
  Mathematics, vol. 1967, Springer-Verlag, Berlin, 2009. \MR{2494775}

\bibitem[Fre17]{MR3643404}
\bysame, \emph{Homotopy of operads and {G}rothendieck-{T}eichm\"{u}ller groups.
  {P}art 1}, Mathematical Surveys and Monographs, vol. 217, American
  Mathematical Society, Providence, RI, 2017, The algebraic theory and its
  topological background. \MR{3643404}

\bibitem[HM21]{MR4321214}
Kazuo Habiro and Gw\'{e}na\"{e}l Massuyeau, \emph{The {K}ontsevich integral for
  bottom tangles in handlebodies}, Quantum Topol. \textbf{12} (2021), no.~4,
  593--703. \MR{4321214}

\bibitem[HPV15]{MR3340364}
Manfred Hartl, Teimuraz Pirashvili, and Christine Vespa, \emph{Polynomial
  functors from algebras over a set-operad and nonlinear {M}ackey functors},
  Int. Math. Res. Not. IMRN (2015), no.~6, 1461--1554. \MR{3340364}

\bibitem[HV02]{MR1913297}
Vladimir Hinich and Arkady Vaintrob, \emph{Cyclic operads and algebra of chord
  diagrams}, Selecta Math. (N.S.) \textbf{8} (2002), no.~2, 237--282.
  \MR{1913297}

\bibitem[Jen55]{MR68540}
S.~A. Jennings, \emph{The group ring of a class of infinite nilpotent groups},
  Canadian J. Math. \textbf{7} (1955), 169--187. \MR{68540}

\bibitem[{Kat}21]{2021arXiv210206382K}
Mai {Katada}, \emph{{Actions of automorphism groups of free groups on spaces of
  Jacobi diagrams. I}}, arXiv e-prints (2021), arXiv:2102.06382.

\bibitem[KM01]{MR1854112}
M.~Kapranov and Yu. Manin, \emph{Modules and {M}orita theorem for operads},
  Amer. J. Math. \textbf{123} (2001), no.~5, 811--838. \MR{1854112}

\bibitem[Mag37]{MR1581549}
Wilhelm Magnus, \emph{\"{U}ber {B}eziehungen zwischen h\"{o}heren
  {K}ommutatoren}, J. Reine Angew. Math. \textbf{177} (1937), 105--115.
  \MR{1581549}

\bibitem[Mit72]{MR294454}
Barry Mitchell, \emph{Rings with several objects}, Advances in Math. \textbf{8}
  (1972), 1--161. \MR{294454}

\bibitem[MKS04]{MR2109550}
Wilhelm Magnus, Abraham Karrass, and Donald Solitar, \emph{Combinatorial group
  theory}, second ed., Dover Publications, Inc., Mineola, NY, 2004,
  Presentations of groups in terms of generators and relations. \MR{2109550}

\bibitem[Pas79]{MR537126}
Inder Bir~S. Passi, \emph{Group rings and their augmentation ideals}, Lecture
  Notes in Mathematics, vol. 715, Springer, Berlin, 1979. \MR{537126}

\bibitem[{Pow}21]{2021arXiv211001934P}
Geoffrey {Powell}, \emph{{On analytic contravariant functors on free groups}},
  arXiv e-prints (2021), arXiv:2110.01934.

\bibitem[PV18]{PV}
G.~{Powell} and C.~{Vespa}, \emph{{Higher Hochschild homology and exponential
  functors}}, ArXiv:1802.07574 (2018).

\bibitem[Qui68]{MR231919}
Daniel~G. Quillen, \emph{On the associated graded ring of a group ring}, J.
  Algebra \textbf{10} (1968), 411--418. \MR{231919}

\bibitem[Qui69]{MR258031}
Daniel Quillen, \emph{Rational homotopy theory}, Ann. of Math. (2) \textbf{90}
  (1969), 205--295. \MR{258031}

\bibitem[{Ves}22]{2022arXiv220210907V}
Christine {Vespa}, \emph{{On the functors associated with beaded open Jacobi
  diagrams}}, arXiv e-prints (2022), arXiv:2202.10907.

\end{thebibliography}

\providecommand{\bysame}{\leavevmode\hbox to3em{\hrulefill}\thinspace}
\providecommand{\MR}{\relax\ifhmode\unskip\space\fi MR }
\providecommand{\MRhref}[2]{%
  \href{http://www.ams.org/mathscinet-getitem?mr=#1}{#2}
}
\providecommand{\href}[2]{#2}

\end{document}